\documentclass[10pt]{amsart}

\usepackage{graphicx}

\usepackage{graphics,color}

\usepackage{psfrag}

\graphicspath{{figures/}}

\newcommand{\F}{{\mathcal {F}}}
\newcommand{\Tgn}{{\bf {T} }_{g,n} }
\newcommand{\Mgn}{{\bf {M}}_{g,n} }
\newcommand{\PSL}{{\bf {PSL}}(2,\Z)}
\newcommand{\Mob}{{\bf {Mob}}(\Ha)}
\newcommand{\Mobi}{{\bf {Mob}_{\infty}}(\Ha)}
\newcommand{\D}{{\bf D}}
\newcommand{\Ha}{{\bf H}}
\newcommand{\ang}{{\Theta}}

\newcommand{\R}{{\bf R}}
\newcommand{\Z}{{\bf Z}}
\newcommand{\N}{{\bf N}}

\newcommand{\h}{{\bf h}}
\newcommand{\z}{{\bf z}}
\newcommand{\p}{{\bf p}}
\newcommand{\at}{{\bf t}}
\newcommand{\abs}{{\bf a}}
\newcommand{\taso}{{\bf b}}
\newcommand{\LM}{{\Lambda}}
\newcommand{\shc}{{\bf {r}}}
\newcommand{\dis}{{\bf d}}
\newcommand{\Thick}{{\bf {Th}}}
\newcommand{\Thin}{{\bf {Thin}}}
\newcommand{\pe}{{\bf {e}}}
\newcommand{\Lab}{{\mathcal {L}}}
\newcommand{\flow}{{\bf {g}}}
\newcommand{\Col}{{\mathcal {C}}}
\newcommand{\Hor}{{\mathcal {H}}}
\newcommand{\Tr}{{\mathcal {T}}}
\newcommand{\E}{{\mathcal {E}}}
\newcommand{\A}{{\mathcal {A}}}
\newcommand{\B}{{\mathcal {B}}}
\newcommand{\U}{{\mathcal {U}}}
\newcommand{\Pol}{{\mathcal {P}}}
\newcommand{\FD}{{\mathcal {D}}}
\newcommand{\cl}{{\mathcal {K}}}
\newcommand{\TB}{T^{1}}
\newcommand{\NB}{N^{1}}
\newcommand{\NBG}{{\bf {N}}^{1}\Gamma}
\newcommand{\Hom}{H_{1}(\overline{S},\Cus(S))}
\newcommand{\Mes}{{\mathcal {M}}}
\newcommand{\ph}{\widehat{\partial}}

\newcommand{\refl}{{\mathcal {R}}}
\newcommand{\tra}{{\mathcal {S}}}
\newcommand{\inv}{{\mathcal {I}}}
\newcommand{\fta}{{\bf {f}}}
\newcommand{\tft}{\bar{{\bf {f}}}}
\newcommand{\numb}{{\bf {n}}}
\newcommand{\gen}{{\bf {g}}}

\newcommand{\wt}{\widetilde}
\newcommand{\wh}{\widehat}

\newcommand{\Area}{\operatorname{Area}}
\newcommand{\Teich}{\operatorname{Teich}}
\newcommand{\ct}{\operatorname{ct}}
\newcommand{\Cus}{\operatorname{Cusp}}
\newcommand{\Gen}{\operatorname{Gen}}
\newcommand{\midp}{\operatorname{mid}}
\newcommand{\IM}{\operatorname{Im}}

\newcommand{\rat}{\operatorname{rat}}
\newcommand{\Sol}{\operatorname{Sol}}
\newcommand{\intg}{\operatorname{int}}
\newcommand{\lab}{\operatorname{lab}}

\newcommand{\rot}{\operatorname{rot}}

\newcommand{\Proj}{\operatorname{Pr}}
\newcommand{\Rot}{{\operatorname {R}}}
\newcommand{\Jac}{{\operatorname {Jac}}}
\newcommand{\ft}{{\operatorname {foot}}}
\newcommand{\ent}{{\operatorname {ent}}}
\newcommand{\ex}{{\operatorname {ex}}}
\newcommand{\supp}{{\operatorname {supp}}}

\newtheorem{theorem}{Theorem}[section]
\newtheorem{definition}{Definition}[section]
\newtheorem{lemma}{Lemma}[section]

\newtheorem{proposition}{Proposition}[section]

\theoremstyle{remark}
\newtheorem*{remark}{Remark}

\begin{document}

\title[Random ideal triangulations] {Random ideal triangulations and the Weil-Petersson distance between finite degree covers of punctured Riemann surfaces}
\author[Kahn and Markovic]{Jeremy Kahn and Vladimir Markovic}
\address{\newline Mathematics Department \newline  Stony Brook University \newline Stony Brook, 11794 \newline  NY, USA \newline  and
\newline University of Warwick \newline Institute of Mathematics \newline Coventry,
CV4 7AL, UK}
\email{kahn@math.sunysb.edu,  v.markovic@warwick.ac.uk}

\today

\subjclass[2000]{Primary 20H10}

\begin{abstract} Let $S$ and $R$ be two hyperbolic finite area surfaces with cusps. We show that for every $\epsilon>0$ there are finite degree unbranched covers $S_{\epsilon} \to S$ and $R_{\epsilon} \to R$, such that the Weil-Petersson distance between $S_{\epsilon}$ and $R_{\epsilon}$ is less than $\epsilon$ in the corresponding Moduli space.
\end{abstract}

\maketitle

\section{Introduction}
We say that a hyperbolic Riemann  surface is of finite type if it has the finite area with respect to the underlining hyperbolic metric. 
Such surfaces are  
either closed or are obtained from closed surfaces after removing at most finitely punctures. 
All Riemann surfaces in this paper are hyperbolic and of finite type
(except the unit disc/upper half space  which is the universal cover of such surfaces). 
Let $S$ and $R$ be two finite type Riemann surfaces that are both either closed or both have at 
least one puncture. Then there is always a common holomorphic (possibly branched) cover of $S$ and $R$. 
However a generic pair of such surfaces will not have a common holomorphic, unbranched  finite degree cover.  
Except the universal cover, from now on  all covers in this  paper will be assumed to be holomorphic, unbranched and finite degree,  
so every time we use the term cover this will be understood. Since we assume that both $S$ and $R$ are either closed or have at least one puncture, one can find  covers  $S_1$ and $R_1$, of $S$ and $R$ respectively, that are quasiconformally  equivalent (there is a quasiconformal map between them). 
This is equivalent to saying that $S_1$ and $R_1$ have the same genus $\gen$ and the same number of punctures $\numb$ (both $S_1$ and $R_1$ 
are of the type $(\gen,\numb)$).

\vskip .1cm

The well-known Ehrenpreis conjecture asserts that for a given $\epsilon>0$ one can find covers $S_1$ and $R_1$,
of $S$ and $R$ respectively, so that $S_1$ and $R_1$ are quasiconformally equivalent and the distance between them
is less than $\epsilon$. Since $S_1$ and $R_1$ are quasiconformally equivalent they belong to the same moduli space $\Mgn$, where $\gen$ is the genus and $\numb$ is the number of 
punctures in $S_1$.  Recall that $\Mgn$ is the space of all hyperbolic metrics one can put on a surface that has genus $\gen$ and $n$ punctures. The  distance between  them  is measured in terms of a metric that exists on $\Mgn$. Originally the problem was posed in terms of any natural metric on the corresponding $\Mgn$ (see \cite{eh}). Note that there are two cases of this conjecture, the first is when $S$ and $R$ have punctures, and the second when they are both closed. 
\begin{remark} In fact there is no easy way showing that one can find the corresponding covers $S_1$ and $R_1$ so that the distance between them in their moduli space is strictly less than the distance between $S$ and $R$ in their moduli space (providing that $S$ and $R$ are homeomorphic). 
\end{remark}

Recall that the Teichm\"uller space is the universal holomorphic cover of $\Mgn$. The two standard (and most studied) metrics on  
$\Tgn$ that are well defined on $\Mgn$ are the Teichm\"uller metric and the Weil-Petersson metric. In this paper we prove this conjecture in the case of 
punctured surfaces and where the distance is measured with respect to the normalised Weil-Petersson metric. We stress that our results do not imply the case when the distance is measured  with respect to the  Teichm\"uller metric. 

\begin{remark}
Let $M$ be a finite type surface and $\pi:M_1 \to M$ be a cover (the genera of $M$ and $M_1$ is $\gen$ and $\gen _1$ respectively and the number of punctures is $\numb$ and $\numb _1$  respectively). Then the covering map $\pi$ induces an embedding $\iota:\Tgn \to {\bf T}_{\gen _1,\numb_1}$. This embedding is an isometry of $\Tgn$ onto its image with respect to the Teichm\"uller metric. However, if one takes the traditional definition of the Weil-Petersson metric this embedding increases the distance by multiplying it by the  square root of the degree of the cover. Therefore it is important that we normalise the Weil-Petersson metric on $\Tgn$ (and therefore on $\Mgn$) by dividing it by the square root of the hyperbolic area of a surface of the type $(\gen,\numb)$  (every such surface has the area equal to $2\pi(2\gen-2+\numb)$). 
After this normalisation the embedding $\iota$ becomes an isometry with respect to the Weil-Petersson metric as well.
\end{remark}

\vskip .1cm
\begin{theorem}   Let $S$ and $R$ be two finite type Riemann surfaces that both have at least one puncture. 
Then given $\epsilon>0$ one can find covers $S_{\epsilon}$ and $R_{\epsilon}$ of $S$ and $R$ respectively, so that $S_{\epsilon}$ and $R_{\epsilon}$ are 
quasiconformally equivalent and the Weil-Petersson  distance between them is less than $\epsilon$.
\end{theorem}
\vskip .1cm
We construct the covers $S_{\epsilon}$ and $R_{\epsilon}$ explicitly. We believe that  it can be recovered from our construction
that the degree of the covers $S_{\epsilon} \to S$ and $R_{\epsilon} \to R$  is of the order $P({{1}\over{\epsilon}})$, where $P$ is a polynomial that depends on $S$ and $R$. We believe that the degree of this polynomial is independent of $S$ and $R$. 
\vskip .1cm
It has been shown in \cite{m-s} that for every $\epsilon>0$ there are covers of the Modular torus (this is the punctured torus that is isomorphic to $\Ha/G$ where $G$ is a finite index subgroup of $\PSL$) that are not conformally the same but are $\epsilon$ close in the corresponding moduli space and with respect to the Teichm\"uller metric. See  \cite{b-n-s} for equivalent formulations of the Ehrenpreis conjecture (see also \cite{p-r}, \cite{m-s-1} for related results). 
\vskip .1cm
We say that a Riemann surface $S_0$ is modular if  $S_0$ is isomorphic to $\Ha/G$ where $G$ is a finite index subgroup of $\PSL$. Such surfaces are characterised by having an ideal triangulation where all the shears are equal to zero (a shear coordinate that corresponds to an edge $\lambda_i$  of an ideal  triangulation is the signed hyperbolic distance between the normal projections to $\lambda_i$  of the centres of the two ideal triangles that contain $\lambda_i$ as their edge). Give a Riemann surface $S$, for every $r>>0$ we construct  a finite degree cover $S(r) \to S$ such that $S(r)$ has an ideal triangulation where the shear coordinates are ''small". Then there exists a modular surface $S_0(r)$ such that all the shear coordinates of  the corresponding ideal triangulation are equal to zero. One has to make this precise and in particular be able to estimate the Weil-Petersson distance between
$S(r)$ and $S_0(r)$. We show that the Weil-Petersson distance between $S(r)$ and $S_0(r)$ less than $e^{ -{{r}\over{8}}}$. It can be recovered from the construction that the degree of the cover  $S(r) \to S$ is less than $P(e^{r})$, where $P(r)$ is a polynomial in $r$.
\vskip .1cm
In Section 2. we discuss the Weil-Petersson distance and obtain the needed estimates of this distance in terms of the shear coordinates.
In Section 3. we develop the method of construction finite degree covers of $S$ by gluing ideal immersed ideal triangles in $S$. In Section 4. we discuss measures on triangles and the notion of transport of measure. We state Theorem 4.1 which claims existence of certain measures on the space of immersed ideal triangles in $S$. We prove Theorem 1.1 using Theorem 4.1. In Section 5. we construct the measures from Theorem 4.1 and prove this theorem. Heavier computations needed in the proof of Theorem 4.1 are done in Section 6 and Section 7. In the appendix we prove an ergodic type theorem about the geodesic flow on a finite area hyperbolic surface with cusps. This theorem is most likely known but in the absence of an appropriate reference we offer a proof.

\section{The shear coordinates and the Weil-Petersson metric}

\subsection{The Weil-Petersson metric}
Let $S$ be a Riemann  surface of the type $(\gen,\numb)$. Let $\rho(z)|dz|$ be the line 
element for the hyperbolic metric on $S$ (here $z=x+iy$ is the local parameter). 
Denote by $Q(S)$ the Banach space of holomorphic quadratic differentials 
on $S$ with the norm given by
$$
||\phi||_2= \sqrt{ {{1}\over{2\pi(2\gen-2+\numb)}}\int\limits_{S}\rho^{-2}(z)|\phi(z)|^{2}\, dxdy }.
$$
\noindent
Here $\Area(S)=2\pi(2\gen-2+\numb)$ represents the hyperbolic area of $S$. 
Note that if the $\numb \ge 1$ than elements of $Q(S)$ have at most first 
order poles at the punctures. 
\vskip .1cm
By $L^{\infty}(S)$ we  denote the Banach space of Beltrami differentials 
on $S$. 
These are measurable $(-1,1)$ forms with the finite 
supremum norm $||\mu||_{\infty}$ for $\mu \in L^{\infty}(S)$. 
We introduce the equivalence relation on $L^{\infty}(S)$ by saying that 
$\mu \sim \nu$ if
$$
\int\limits_{S}\mu \phi(z)\, dxdy=\int\limits_{S}\nu \phi(z)\, dxdy,
$$
\noindent
for every $\phi \in Q(S)$.  The equivalence class of $\mu \in 
L^{\infty}(S)$ is denoted by $[\mu]$.
The space $L^{\infty}(S)/\sim$ is a finite dimensional vector space. The 
induced supremum norm
on  $L^{\infty}(S)/\sim$ is given by 
$$
||[\mu]||_{\infty}=\inf_{\nu \in [\mu]}||\nu||_{\infty}.
$$
\vskip .1cm
The Teichm\"uller space $\Tgn$ is a complex manifold and its tangent 
space $T_{S}(\Tgn)$ at $S$ is 
identified with the vector space $L^{\infty}/\sim$. The corresponding 
cotangent space $T^{*}_{S}(\Tgn)$ 
is identified with $Q(S)$. The Weil-Petersson pairing on $Q(S)$ is given 
by
$$
\left< \phi,\psi \right>_{WP}={{1}\over{2\pi(2\gen-2+\numb)}}\int\limits_{S}\rho^{-2}(z)\phi(z) 
\overline{\psi(z)}\, dxdy.
$$
\noindent
The induced scalar product on  $T_{S}(\Tgn)=L^{\infty}(S)/ \sim$ is called the 
Weil-Petersson product. The corresponding norm of a vector
$[\mu] \in L^{\infty}(S)/ \sim$ is given by
$$
||[\mu]||_{WP}=\sup_{\phi \in Q(S)} {{1}\over{2\pi(2\gen-2+\numb)||\phi||_2}}
\left|\int\limits_{S}\mu\phi \, dxdy \right|
$$
\noindent
Below we show how to define the norm $||[\mu]||_{WP}$ in terms of  harmonic Beltrami differentials.
The induced Riemannian metric on $\Tgn$ is called the Weil-Petersson 
metric, and by $d_{WP}$ we denote the corresponding distance
on $\Tgn$. 
\vskip .1cm
This definition of the Weil-Petersson distance is a 
modification of the usual one. If $P,Q \in \Tgn$ than the usual Weil-Petersson distance between $P$ and $Q$ is given by 
$\sqrt{ 2\pi(2\gen-2+\numb)}  d_{WP}(P,Q)$. 
Clearly we have just rescaled the distance by the factor 
$\sqrt{ 2\pi(2\gen-2+\numb) }$. If $S_1$ is a Riemann surface of the 
type
$(g_1,n_1)$ that covers $S$ and if $\iota:\Tgn \to {\bf {T_{g_1,n_1}}}$ is 
the induced embedding, than $\iota$ is an isometric embedding
with our definition of the Weil-Petersson distance. Also, according to 
\cite{li}   we have that $d_{WP}(P,Q)$
is strictly less than the Teichm\"uller distance between $P$ and $Q$.
\vskip .1cm

We have 
\begin{lemma} Let $\mu \in L^{\infty}(S)$. Then
$$
||[\mu]||^{2}_{WP}\le 9||[\mu]||_{\infty} \left( \inf_{\nu \in 
[\mu]}\int\limits_{S}|\nu|\rho^{2}(z)\, dxdy \right).
$$
\end{lemma}
\begin{proof}
Let $G$ be the covering group of M\"obius transformations acting on the upper half plane $\Ha$
so that the Riemann surface $S$ is isomorphic to $\Ha /G$. Let $S_0$ be a fundamental domain for the
action of $G$. The lift of $\mu \in L^{\infty}(S)$ to $\Ha$ we also denote by $\mu$. Recall that the
density of the hyperbolic metric on $\Ha$ is given by $\rho^{2}(z)=y^{-2}$ where $z=x+iy \in \Ha$. 
Let $w=u+iv \in \Ha$ denote another complex parameter. Set
$$
K(z,w)={{1}\over{(\bar{z}-w)^{4}}},
$$
\noindent
where $z,w \in \Ha$. The function $K(z,w)$ is the Bergman kernel for $\Ha$.
The following are the well known properties of $K(z,w)$. We have 
\begin{itemize}
\item $|K(z,w)|=|K(w,z)|$.

\item For any M\"obius transformation $g:\Ha \to \Ha$ we have
\begin{equation}\label{kernel-1}
|K(g(z),g(w))||g'(z)|^{2}|g'(w)|^{2}=|K(z,w)|,
\end{equation}

\item  For every $z,w \in \Ha$ we have
\begin{equation}\label{kernel-2}
{{4v^{2}}\over{\pi}} \int\limits_{\Ha} |K(z,w)|\, dxdy=
{{4y^{2}}\over{\pi}} \int\limits_{\Ha} |K(z,w)|\, dudv=1.
\end{equation}

\end{itemize}

\vskip .1cm
Let
$$
\phi[\mu](w)={{-12}\over{\pi}} \int\limits_{\Ha} \overline{\mu(z)} K(z,w)\, dxdy.
$$
\noindent
The differential  $v^{2}\phi[\mu](w)$ is called the harmonic Beltrami differential. The Weil-Petersson
norm $||[\mu]||_{WP}$ can be expressed as (see \cite{ah})
$$
||[\mu]||^{2}_{WP}=\int\limits_{S_0} v^{2}|\phi[\mu](w)|^{2}\, dudv.
$$

Let $\nu, \nu_1 \in [\mu]$ be any Beltrami dilatations from $[\mu]$ and use the same notation for the lifts of $\nu$ and $\nu_1$ to $\Ha$. Then
$\phi[\mu](w)=\phi[\nu](w)=\phi[\nu_1](w)$.
We have
$$
||[\mu]||^{2}_{WP} \le \sup_{w \in \Ha}|v^{2} \phi[\nu_1](w)|
\left(\int\limits_{S_0} |\phi[\nu](w)|\, dudv \right).
$$
\noindent
From (\ref{kernel-2}) we have
 
\begin{equation}\label{WP-1}
||[\mu]||^{2}_{WP} \le 3||[\mu]||_{\infty}
\int\limits_{S_0} |\phi[\nu](w)|\, dudv.
\end{equation}

\vskip .1cm
We have 
 
$$
\int\limits_{S_0} |\phi[\nu](w)|\, dudv=
{{12}\over{\pi}} \int\limits_{S_0} \left|\int\limits_{\Ha}\overline{\nu(z)} K(z,w)\, dxdy \right|\, dudv \le
$$

$$
\le {{12}\over{\pi}} \int\limits_{S_0} \left(\int\limits_{\Ha}|\nu(z)||K(z,w)|\, dxdy \right)\, dudv.
$$
\noindent
We partition $\Ha$ into the sets $g(S_0)$, $g \in G$. This gives

$$
\int\limits_{S_0} |\phi[\nu](w)|\, dudv \le
{{12}\over{\pi}} \sum\limits_{g \in G} \int\limits_{S_0} \left(\int\limits_{g(S_0)}|\nu(z)||K(z,w)|\, dxdy \right)\, dudv.
$$
\noindent
We have

$$
{{12}\over{\pi}} \sum\limits_{g \in G} \int\limits_{S_0} \left(\int\limits_{g(S_0)}|\nu(z)||K(z,w)|\, dxdy \right)\, dudv=
$$

$$
={{12}\over{\pi}} \sum\limits_{g \in G} \int\limits_{S_0} 
\left(\int\limits_{S_0}|\nu(g(z))||K(g(z),w)||g'(z)|^{2}\, dxdy \right)\, dudv=
$$

$$
={{12}\over{\pi}} \sum\limits_{g \in G} \int\limits_{g^{-1}(S_0)} 
\left(\int\limits_{S_0}|\nu(g(z))||K(g(z),g(w))||g'(z)|^{2}|g'(w)|^{2}\, dxdy \right)\, dudv.
$$
\noindent
Since $|\nu(g(z))|=|\nu(z)|$ and from (\ref{kernel-1}) we get

$$
\int\limits_{S_0} |\phi[\nu](w)|\, dudv \le 
{{12}\over{\pi}} \sum\limits_{g \in G} \int\limits_{g^{-1}(S_0)} 
\left(\int\limits_{S_0}|\nu(z)||K(z,w)|\, dxdy \right)\, dudv=
$$

$$
={{12}\over{\pi}} \int\limits_{\Ha} 
\left(\int\limits_{S_0}|\nu(z)||K(z,w)|\, dxdy \right)\, dudv.
$$
\noindent
We exchange the integrals in the above inequality to get

$$
\int\limits_{S_0} |\phi[\nu](w)|\, dudv \le 
{{12}\over{\pi}} \int\limits_{S_0} 
\left(\int\limits_{\Ha}|\nu(z)||K(z,w)|\, dudv \right)\, dxdy.
$$
\noindent
From (\ref{kernel-2}) we have that 
$$
{{12}\over{\pi}} \int\limits_{\Ha}|K(z,w)|\, dudv \le 3y^{-2},
$$
\noindent
which shows that

$$
\int\limits_{S_0} |\phi[\nu](w)|\, dudv \le 
3\int\limits_{S_0} y^{-2}|\nu(z)|\, dxdy=3\int\limits_{S} \rho^{2}(z)|\nu(z)|\, dxdy,
$$

\noindent
where $\rho^{2}(z)$ is the density of the hyperbolic metric on $S$. Together with (\ref{WP-1})
this proves the lemma.

\end{proof}
 
\subsection{The Shear coordinates for $\Tgn$} 
The notation we introduce here remains valid throughout this section. Fix a surface $S$ of genus $\gen$ and with $\numb \ge 1$ punctures (here $S$ is
not assumed to be a Riemann surface, in fact we will equip $S$ with various complex structures). 
By $\Cus(S)=\{c_1(S),c_2(S),...,c_{\numb}(S) \}$  we  denote the set of punctures (recall that $\overline{S}$ is a closed surface and $S$ is obtained by removing
the set $\Cus(S)$ from $\overline{S}$). Let $\tau$ be an ideal triangulation of $S$. 
This means that  $\tau$ is a triangulation of $\overline{S}$ where the vertex set is exactly equal 
to $\Cus(S)$ (from now on all triangulations will be ideal triangulations and we will not use the term ideal anymore).  
By  $\lambda=\{\lambda_1,...,\lambda_{|\lambda|} \}$ we denote  the ordered set of  edges of the 
triangles from $\tau$. Here $|\lambda|$  denotes the total number of edges. 
For topological reasons we have $|\lambda|=6\gen-6+3\numb$.  Let $|\tau|$ denote the total number of triangles in $\tau$. We have
$|\tau|=2(2\gen-2+\numb)$ and therefore the equality

$$
{{2}\over{3}}|\lambda|=|\tau|,
$$

\noindent
holds.
We say that two triangulations $\tau$ and $\tau_1$ of $S$ are isotopic if there is a 
homeomorphism $f:\overline{S} \to \overline{S}$ that pointwise preserves the set
$\Cus(S)$ and  that is homotopic to the identity map on $S$ modulo the set $\Cus(S)$. 

\vskip .1cm

Fix a triangulation $\tau$ of $S$. We define the set $X(\tau) \subset \R^{|\lambda|}$ as follows. 
Let $c_i(S) \in \Cus(S)$ and let $\lambda_{i(1)},..,\lambda_{i(k)}$ be 
the subset of edges from the set $\lambda$  that have $c_i(S)$ as an end point. 
Also, let $\sigma: \lambda \to \{1,2\}$ be defined so that for $\lambda_i \in \lambda$ we have $\sigma(\lambda_i)=1$ if the endpoints of $\lambda_i$ 
represent different punctures on $S$ and $\sigma(\lambda_i)=2$ if the two endpoints represent the same puncture on $S$.
Let $\shc=(\shc_1,...,\shc_{|\lambda|}) \in \R^{|\lambda|}$. Then  $\shc \in X(\tau)$ if for every $i=1,...,\numb$ we have 

$$
\sum\limits_{j=1}^{j=k} \sigma(\lambda_{i(j)}) \shc_{i(j)}=0.
$$

\noindent

Clearly $X(\tau)$ is a linear subset of $\R^{|\lambda|}$. In particular, we have  $0=(0,...,0) \in X(\tau)$.

\vskip .1cm

Now we define the map $F_{\tau}:X(\tau) \to \Tgn$. Recall that $\Tgn$ is the space of complete 
hyperbolic metrics on $S$ with marking. To a point $\shc=(\shc_1,...\shc_{|\lambda|}) \in X(\tau)$ we associate 
an element $F_{\tau}(\shc) \in \Tgn$ that is represented by the marked hyperbolic metric 
with the following properties. There exists a triangulation $\tau'$ (with the corresponding 
set of edges $\lambda'$) of $S$ that is 
isotopic to $\tau$ and such that all the edges in $\lambda'$ are geodesics with respect to this 
metric and so that the following holds (a triangulation will be called a geodesic triangulation if all the edges are geodesics
in the corresponding hyperbolic metric). 
Let $\lambda'_i \in \lambda'$ and let $T_1$ and $T_2$ be the two triangles from the triangulation
$\tau'$ that have $\lambda'_i$ as an edge.  We lift $\lambda'_i$, $T_1$ and $T_2$ to the universal 
cover $\Ha$ (the upper half space). The edge $\lambda'_i$ lifts to the geodesic that connects $0$ 
and $\infty$. One of the triangles $T_1$ or $T_2$ (depending on how we lift $\lambda'_i$) lifts to 
the ideal triangle with vertices at $-1,0,\infty$ and the other triangle lifts to the triangle with 
vertices $0,e^{\shc_i},\infty$.  This requirement defines the map $F_{\tau}$.
This definition does not depend on how we lift the geodesic $\lambda'_i$ (there are exactly two ways we can lift it). 
The number $\shc_i$ can also be defined as the signed hyperbolic distance between 
the points on the geodesic $\lambda'_i$ that are the orthogonal projections of the centres of the triangles 
of $T_1$ and $T_2$ respectively  (equivalently one can see these two points as the feet of the 
perpendiculars from the two vertices of $T_1$ and $T_2$ respectively that  are opposite to $\lambda'_i$).

\vskip .1cm 

This maps was originally defined by Thurston 
(this concept has been developed by Bonahon, Penner, Fock, Chekhov...). 
The corresponding coordinates on $\Tgn$ are called the shear coordinates. 

\begin{proposition} 

The map $F_{\tau}:X(\tau) \to \Tgn$ is well defined real analytic homeomorphism.
 
\end{proposition}

See \cite{c-p} for the proof of this theorem. This parametrisation of $\Tgn$ depends on the choice of
the triangulation $\tau$ and different triangulations produce different parametrisation. 
\vskip .1cm
Recall that the Farey tessellation $\F$ is the ideal triangulation of $\Ha$ 
with the property that for any two adjacent triangles the feet of the perpendiculars that are dropped from the vertices of
these two triangles, that are opposite to their common edge, coincide. To make $\F$ be a unique tessellation satisfying this property we require that one its triangles is the one with the vertices at $0,1,\infty$. It is well known that this triangulation is preserved by the action of the group $\PSL$.

\begin{proposition} Let $\tau$ be a triangulation of a finite type surface $S$ ($S$ has at least one puncture). Then
the Riemann surface that corresponds to $F_{\tau}(0)$ is obtained as the quotient of $\Ha$ by a finite index subgroup of
$\PSL$.
\end{proposition}

\begin{proof} Let $R$ be the Riemann surface that corresponds to $F_{\tau}(0)$. We lift the corresponding geodesic triangulation $\tau'$ of $R$ to $\Ha$ and assume that
the triangle with vertices at $0,1,\infty$ belongs to this lift. Let  
$G$ be the corresponding covering group of M\"obius transformation acting on $\Ha$. 
Since $R$ corresponds to $F_{\tau}(0)$ we conclude that the corresponding tessellation of
$\Ha$ is the Farey tessellation $\F$ and $G$ preserves $\F$.
On the other hand, every M\"obius transformation that preserves $\F$ must be in $\PSL$.
This shows that $G$ is a subgroup of $\PSL$. The fact that $G$ is of finite index follows from the assumption that
$S$ is of finite type so $R$ has finite hyperbolic area.
\end{proof}

This simple proposition is important for us.

\vskip .1cm
For $\shc \in X(\tau)$ we define the supremum norm $||\shc||_{\infty}=\max \{ \shc_1,...,\shc_{|\lambda|} \}$ as usual (this norm does not depend on $\tau$ of course). 
Next, we define a norm on $X(\tau)$ that depends on $\tau$. We define the Oscillation norm $O_{\tau}(\shc)$ for $\shc \in X(\tau)$ as follows. Let $G$ be the covering group of M\"obius transformations acting on 
$\Ha$ so that $S$ is isomorphic to $\Ha/G$ (here $S$  has the complex structure that corresponds to $F_{\tau}(\shc)$). 
Let $\{\lambda_{i(1)},...,\lambda_{i(k)} \}$ be a k-tuple of edges from $\lambda$. We say that this k-tuple
is a k-tuple of consecutive edges if we can find the lifts $\{\lambda'_{i(1)},...,\lambda'_{i(k)} \}$ to $\Ha$ so that each curve $\lambda'_{i(j)}$
has $\infty$ as its endpoint, and so that each curve $\lambda'_{i(j)}$ is to the left of the curve $\lambda'_{i(j+1)}$, where $j=1,...k_1$.
In this case we also say that the k-tuple of consecutive edges $\{\lambda_{i(1)},...,\lambda_{i(k)} \}$ is left oriented (one similarly defines a right oriented k-tuple of consecutive edges).
\vskip .1cm

Set

$$
O_{\tau}(\shc)=\sup_{ \{\lambda_{i(1)},...,\lambda_{i(k)} \} } |\shc_{i(1)}+...+\shc_{i(k)}|,
$$
\noindent
where the supremum is taken among all consecutive k-tuples  $\{\lambda_{i(1)},...,\lambda_{i(k)} \} $ (and for any $k \in \N$).
Note that if $k$ is equal to the number of all edges that enter the puncture that corresponds to $\infty$ (an edge is counted twice if the
puncture on $S$ that corresponds to $\infty$ is equal to both of its endpoints) then by definition of $X(\tau)$ we have 
$\shc_{i(1)}+...+\shc_{i(k)}=1$. This shows that the supremum in the above definition is achieved and this shows that $O_{\tau}(\shc)$ is a well defined
non-negative real number. Note that $||\shc||_{\infty} \le O_{\tau}(\shc)$.
\vskip .1cm

\subsection{Estimating the Weil-Petersson distance in terms of the shear coordinates}
Our aim here is to estimate from above the Weil-Petersson distance $d_{WP}(F_{\tau}(0),F_{\tau}(\shc))$
for a given $\shc \in X(\tau)$. We make this estimate in terms of the vector $\shc$ (and under certain assumptions on $\shc$).  
Until the end of this section $\shc \in X(\tau)$ is a fixed vector.

\vskip .1cm
Let $\psi:[0,1] \to \Tgn$ be given by $\psi(t)=F_{\tau}(t\shc)$. The map $\psi$ is differentiable (since $F_{\tau}$ is differentiable),
and we compute its first derivative in order to estimate the distance. For $t \in [0,1]$ let $S_t$ be the Riemann surface that 
corresponds to the point $F_{\tau}(t\shc) \in \Tgn$. Fix $t_0 \in [0,1]$ and identify
$\Tgn$ with $\Teich(S_{t_0})$ in the standard way. In order to estimate the distance $d_{WP}(F_{\tau}(0),F_{\tau}(\shc))$ we 
estimate the Weil-Petersson norm of the vector ${{\partial{\psi}}\over{\partial{t}}}(t_0)$ in the tangent space of $\Tgn$ at the point $F_{\tau}(t_{0}\shc)$. First we construct an explicit quasiconformal map 
$f_t:S_{t_{0}} \to S_t$ so that the pair $(S_t,f_t)$ represents the point $\psi(t)$ in $\Teich(S_{t_{0}})$ 
(the requirement is that $f_t$ is homotopic to the identity as a map of $S$ onto itself). 

\begin{remark} The homotopy class of map $f_t$ has been studied by Penner-Saric in \cite{p-s}. Passing to the universal cover of $S_t$, they explicitly construct the quasisymmetric map of the unit circle that determines the homotopy class of $f_t$ in terms of the corresponding ideal triangulation of the unit disc. This has various important applications.
\end{remark}

Let $\tau(t)$ be the geodesic triangulation of $S_t$ that is homotopic to $\tau$. By $\lambda(t)$ we denote the corresponding set of edges (these edges are now geodesics). 
For a triangle $T \in \tau$ we denote by $T(t)$ the corresponding triangle in $\tau(t)$ (for $\lambda_i \in \lambda$ we denote by $\lambda_i(t)$
the corresponding element of $\lambda(t)$). Let $\ct(T(t))$ be the geometric centre of $T(t)$. For adjacent triangles $T_1,T_2 \in \tau$ let $l(T_1(t),T_2(t))$ denote the 
geodesic segment between the centres  $\ct(T_1(t))$ and  $\ct(T_2(t))$.  We define the map $f_t$ at the centres
$\ct(T(t_{0}))$ by setting $f_t(\ct(T(t_0)))=\ct(T(t))$.  We define $f_t$ on each $l(T_1(t_0),T_2(t_0))$ so that 
$f_t(l(T_1(t_0),T_2(t_0)))=l(T_1(t),T_2(t))$ and by the requirement that $f_t$ stretches the hyperbolic distance by the factor
$$
{{|l(T_1(t),T_2(t))|}\over {|l(T_1(t_0),T_2(t_0))|}},
$$
\noindent
where $|l(T_1(t),T_2(t))|$ and  $|l(T_1(t_0),T_2(t_0))|$ are the corresponding hyperbolic lengths. If $\lambda_i(t) \in \lambda$ is the common edge for
$T_1$ and $T_2$ then we have

$$
f_t(l(T_1(t_0),T_2(t_0)) \cap \lambda_i(t_0))= l(T_1(t),T_2(t)) \cap \lambda_i(t),
$$
\noindent
because the point $l(T_1(t),T_2(t)) \cap \lambda_i(t)$ is always the middle point (in terms of the hyperbolic distance) of the geodesic
segment $l(T_1(t),T_2(t))$.
\vskip .1cm
Fix $T \in \tau$ and let $T_1,T_2,T_3 \in \tau$ be triangles adjacent to $T$. The triangle $T(t)$ is partitioned into sets $E_1(t),E_2(t),E_3(t)$
where these three sets are separated by the segments $ l(T(t),T_i(t)) \cap T(t)$, $i=1,2,3$. In particular, the boundary of 
set $E_1(t)$ is the union of the curve $ (l(T(t),T_1(t)) \cap T(t)) \cup (l(T(t),T_2(t)) \cap T(t))$ and the two geodesic rays lying on the
edges that separate the pairs $T(t),T_1(t)$ and $T(t),T_2(t)$ respectively. Denote by $\lambda_{i(1)}(t)$ and $\lambda_{i(2)}(t)$
the corresponding edges that separate the pairs $T(t),T_1(t)$ and $T(t),T_2(t)$ respectively. The map $f_t$ is already defined
on the part of the boundary of $E_1(t_0)$ and we define it on the rest of $E_1(t_0)$ as follows.

\vskip .1cm
We use the same notation for the lifts of triangles from $\tau(t)$ and the edges from $\lambda(t)$ to the universal cover $\Ha$. 
We assume that for every $t \in [0,1]$ the triangle $T(t) \in \tau(t)$ lifts to the triangle in $\Ha$ that has the vertices at $0,1,\infty$ (our definition of $f_t$ does not depend on this normalisation). 
We may assume that $T_1(t)$ is to the left of the triangle $T(t)$ in the universal cover. Then the edges $\lambda_{i(1)}$ and  $\lambda_{i(2)}$ lift to the geodesics in $\Ha$ with the endpoints $0,\infty$ and $1,\infty$ respectively. 
We have already seen that the two vertices of the triangle $T_1(t)$ are $0$ and $\infty$. The third vertex of $T_1(t)$ is at the point $-e^{-t\shc_{i(1)}}$ (see Figure \ref{wp}).
Similarly, the vertices of $T_2(t)$ are at $1,(1+e^{t\shc_{i(2)}}),\infty$. We need to keep track of the vertices for $T_1(t)$ and $T_2(t)$ in order
to be able to verify later that the map $f_t$ is well defined on the geodesics $\lambda_{i(1)}(t_0)$ and  $\lambda_{i(2)}(t_0)$. 
\vskip .1cm

\begin{figure}
  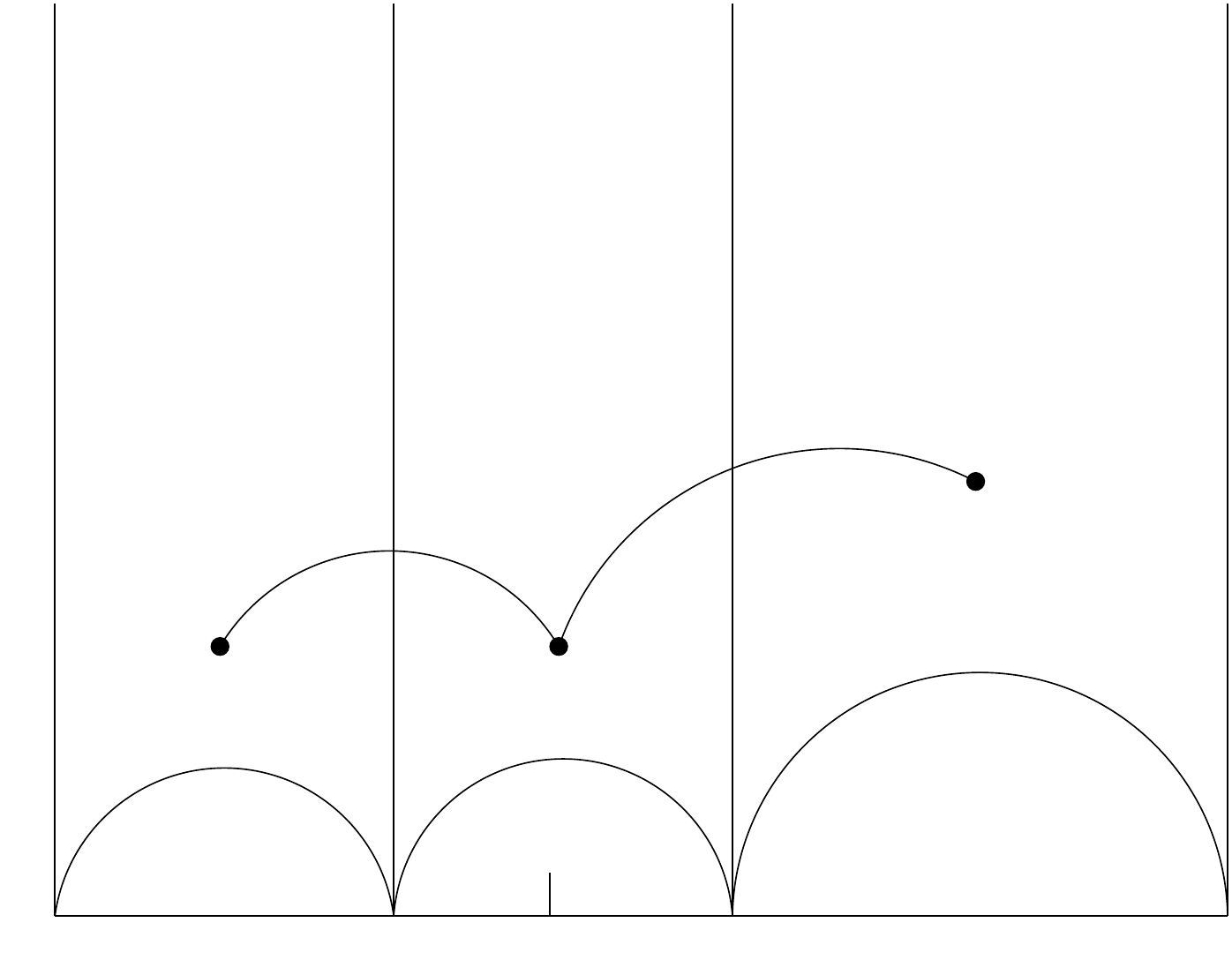
\caption{}
  \label{wp}
\end{figure}

Set

$$
L_t= (l(T(t),T_1(t)) \cap T(t)) \cup (l(T(t),T_2(t)) \cap T(t)).
$$
\noindent
Let $u(t,.):[0,1] \to (0, \infty)$ so that $L_t$ is the graph of the function $u(t,x)$. The function $u(t,x)$ has all derivatives
everywhere except at the point ${{1}\over{2}}$ (but it has both  left and the right derivatives at this point).
At the point $t_0$ we have that the  function $u(t_0,x)$ depends only on the value of the shear coordinates $\shc_{i(1)}$ and $\shc_{i(2)}$. 
If we fix an upper bound on the sum $|\shc_{i(1)}|+|\shc_{i(2)}|$ we have that the set of such functions $u(t,.):[0,1] \to (0, \infty)$ is
compact in the $C^{\infty}$ topology on both intervals $[0,{{1}\over{2}}]$ and $[{{1}\over{2}},1]$. 
This shows that there is a constant $C_1>0$ that depends only on $|\shc_{i(1)}|+|\shc_{i(2)}|$ so that 

\begin{equation}\label{u-estimate}
  ||u(t_0,.)||_{\infty}||, ||{{1}\over{u(t_0,.)}}||_{\infty},   ||u_x(t_0,.)||_{\infty} \le C_1.
\end{equation}

\vskip .1cm
The set $E_1(t)$ is given by

$$
E_1(t)=\{z \in \Ha: 0\le x \le 1,y \ge u(t,x) \}.
$$
\noindent
The map $f_t$ is already defined on $L_{t_0}$. Let $\alpha(t,.):[0,1] \to [0,1]$ and $\beta(t,.):[0,1] \to (0,\infty)$ be the functions
so that
$$
f_t(x+iu(t_0,x))=\alpha(t,x)+i\beta(t,x).
$$
\noindent
The functions $\alpha(t,.)$ and $\beta(t,.)$ depend only on the values  $\shc_{i(1)}$ and $\shc_{i(2)}$ and both of them have all derivatives at every
point in $(0,1)$ except ${{1}\over{2}}$ (but it has the left and the right derivative at this point). 
Note that $\alpha(t,.)$ fixes the points  $0,{{1}\over{2}},1$. 
Also, $\alpha(t_0,x)=x$ and $\beta(t_0,x)=u(t_0,x)$. Again, by fixing an upper bound on the sum $|\shc_{i(1)}|+|\shc_{i(2)}|$ we have that the set
of all such functions $\alpha(t,.):[0,1] \to [0,1]$ and $\beta(t,.):[0,1] \to (0,\infty)$ is compact in the $C^{\infty}$ topology  on the intervals
$[0,{{1}\over{2}}]$ and $[{{1}\over{2}},1]$. This shows that there are functions $\varphi_1, \varphi_2:[0,1] \setminus{{{1}\over{2}}} \to \R$
that depend only on $|\shc_{i(1)}|+|\shc_{i(2)}|$ so that

\begin{equation}\label{alpha-beta-1}
\alpha_x(t,x)=1+\varphi_1(x)(t-t_0)+o(t-t_0), \, \, \beta_x(t,x)=u_x(t_0,x)+\varphi_2(x)(t-t_0)+o(t-t_0),
\end{equation}

\noindent
for every $x \in (0,1)$, $x \ne {{1}\over{2}}$. Moreover, the functions $\alpha(t,x)$ and $\beta(t,x)$ depend on the real variables 
$\shc_{i(1)}$ and $\shc_{i(2)}$ and have all derivatives with respect to these variables too. 
If we fix an upper bound on the sum $|\shc_{i(1)}|+|\shc_{i(2)}|$ then the set of functions $\varphi_1(x)$ and $\varphi_2(x)$ is compact in the variables
$x$, $\shc_{i(1)}$ and $\shc_{i(2)}$ in the $C^{\infty}$ topology. In particular, the derivatives of $\varphi_1$ and $\varphi_2$ with respect to $\shc_{i(1)}$ and $\shc_{i(2)}$ are bounded from above, and this bound depends only on the sum  $|\shc_{i(1)}|+|\shc_{i(2)}|$.
This shows that there is  a constant $C_2>0$ that depends only on the upper bound on the sum 
$|\shc_{i(1)}|+|\shc_{i(2)}|$, so that

\begin{equation}\label{alpha-beta-2}
||\varphi_1||_{\infty},||\varphi_2||_{\infty} \le C_2(|\shc_{i(1)}|+|\shc_{i(2)}|).
\end{equation}
\noindent
The estimate (\ref{alpha-beta-2}) will be used in the case when the sum $|\shc_{i(1)}|+|\shc_{i(2)}|$ is small.
\vskip .1cm
We define the map $f_t$ so that $f_t(E_1(t_0))=E_1(t)$. We have 
$$
f_t(z)=\alpha(t,x)+i(\beta(t,x)+\gamma(t,z)),
$$
\noindent
where the function $\gamma:E_1(t_0) \to [0,\infty)$ is defined as follows.
By $\lambda_{i(1)}(t_0), \, \lambda_{i(2)}(t_0),...,\lambda_{i(k)}(t_0)$ we denote a k-tuple of consecutive edges (see the definition above), starting
with the edge $\lambda_{i(1)}$. Let $z \in E_1(t_0)$. Then $y-u(t_0,x) \ge 0$.
If $y-u(t_0,x) \ge (1-x)$ then  let $k(z) \ge 2$ be the smallest integer so that 
$$
y-u(t_0,x) \le (1-x)+  \sum\limits_{j=2}^{j=k(z)} e^{t_0(\shc_{i(2)}+...+\shc_{i(j)})}.
$$
\noindent
In this case let $p(z) \in [0,1]$ be determined by the identity
\begin{equation}\label{y-condition}
y-u(t_0,x)=(1-x)+  \sum\limits_{j=2}^{j=k(z)-1} {e^{t_0(\shc_{i(2)}+...+\shc_{i(j)})}} + p(z)e^{t_0(\shc_{i(2)}+...+\shc_{i(k(z))})}.
\end{equation}
\noindent
If $0<y<1-x$ then $p(z) \in [0,1]$ is given by $y-u(t_0,x)=p(z)(1-x)$. 
\vskip .1cm
The function  $p(z)$ is defined for $y\ge u(t_0,x)$ and $0 \le x \le 1$  and it is continuous there, while it is differentiable outside a discrete 
set of smooth curves. From (\ref{y-condition}) for $y-u(t_0,x) \ge (1-x)$ we have 
\begin{equation}\label{p-derivative}
p_x(z)=(1-u_x(t_0,x))e^{-t_0(\shc_{i(2)}+...+\shc_{i(k(z))})}, \: p_y(z)=e^{-t_0(\shc_{i(2)}+...+\shc_{i(k(z))})},
\end{equation}
\noindent
where defined. For $y-u(t_0,x) \ge (1-x)$ set
$$
\gamma(t,z)=(1-x)+  \sum\limits_{j=2}^{j=k(z)-1} {e^{t(\shc_{i(2)}+...+\shc_{i(j)})}} + p(z)e^{t(\shc_{i(2)}+...+\shc_{i(k(z))})}.
$$
\noindent
For $y-u(t_0,x)<1-x$ set 

$$
\gamma(t,z)=p(z)(1-x)=y-u(t_0,x).
$$
\vskip .1cm
For $y-u(t_0,x) \ge (1-x)$ it follows from (\ref{p-derivative}) that
\begin{equation}\label{gamma-derivative}
\gamma_x(t,z)=(1-u_x(t_0,x))e^{(t-t_0)(\shc_{i(2)}+...+\shc_{i(k(z))})}-1, \, \, \gamma_y(t,z)=e^{(t-t_0)(\shc_{i(2)}+...+\shc_{i(k(z))})}.
\end{equation}
\noindent
For $y-u(t_0,x)<1-x$ we have   
\begin{equation}\label{gamma-derivative-1}
\gamma_x(t,z)=-u_x(t_0,x), \, \, \gamma_y(t,z)=1.
\end{equation}
\vskip .1cm
As we stated above we define $f_t(z)$ as
$$
f_t(z)=\alpha(t,x)+i(\beta(t,x)+\eta(t,z)).
$$
\noindent
Clearly $f_t:E_1(t_0) \to E_1(t)$ is a homeomorphism. It is differentiable outside a discrete set of smooth curves. By repeating the same process
we define $f_t$ on every triangle. Every triangles can mapped by a M\"obius transformation to the triangle with vertices $0,1,\infty$ and this how we defined the map $f_t$ on every triangle. It is directly seen that the definition of $f_t$ on an edge $\lambda_i(t_0)$ does not depend on the choice of the triangle that has $\lambda_i(t_0)$ in its  boundary. For example consider the triangle $T_1(t_0)$ that is adjacent to $T(t_0)$ (recall that $T(t)$ has the vertices at $0, \, 1, \, \infty$ and $T_1(t)$ has the vertices $-e^{-t\shc_{i(1)}}, \,  0, \, \infty$). Let $A_t(w)=e^{t\shc_{i(1)}}+1$. Then $A_{t}(T_1(t))=T(t)$. We define $g_t:T_1(t_0) \to T_1(t)$ by $ (A_t)^{-1} \circ f_t \circ A_{t_{0}}$ where $f_t:T(t_0) \to T(t)=T(t_0)$. 
It is elementary to see that $f_t=g_t$ on $\lambda_{i(1)}(t_0)$. We define $f_t$ on $T_1(t_0)$ by setting $f_t=g_t$.
\vskip .1cm
The Beltrami dilatation of $f_t$ is given by
$$
({{\partial{f_t}}\over{\partial{\bar{z}}}}/{{\partial{f_t}}\over{\partial{z}}}) (t,z)=
{{(\alpha_x(t,x)-\gamma_y(t,z))+i(\beta_x(t,x)+\gamma_x(t,z))}\over { 
(\alpha_x(t,x)+\gamma_y(t,z))+i(\beta_x(t,x)+\gamma_x(t,z))}}.  
$$
\noindent
Set
$$
\mu(z)={{ \partial{({{\partial{f_t}}\over{\partial{\bar{z}}}}/{{\partial{f_t}}\over{\partial{z}}})}}\over{\partial{t}}}(t_0,z).
$$
\noindent
If $y- u(t_0,x) \ge 1-x$ from (\ref{alpha-beta-1}) and (\ref{gamma-derivative})  we compute
$$
\mu(z)={{1}\over{2}} \big( \varphi_1(x)- (\shc_{i(2)}+...+\shc_{i(k(z))}) \big)+
$$
\begin{equation}\label{mu}
\end{equation}
$$
+{{i}\over{2}} \big(\varphi_2(x)+(1-u_x(t_0,x))(\shc_{i(2)}+...+\shc_{i(k(z))} ) \big).
$$
\noindent
If $y-u(t_0,x)<1-x$ then 
\begin{equation}\label{mu-1}
\mu(z)={{1}\over{2}}( \varphi_1(x)+i\varphi_2(x)).
\end{equation}
\noindent
Going back to our path $\psi:[0,1] \to \Tgn$ we have that $\mu \in L^{\infty}(S_{t_0})$ represents the tangent vector
${{\partial{\psi}}\over{\partial{t}}}(t_0)$ in the tangent space of $\Tgn$ at the point $F_{\tau}(t_{0}\shc)$.
\vskip .1cm
Next, we derive two estimates needed in the proofs of Theorem 2.1 and Theorem 2.2 below. From (\ref{u-estimate}), (\ref{alpha-beta-2}), (\ref{mu}), (\ref{mu-1}) we have

$$
|\mu(z)|\le C_2(|\shc_{i(1)}|+|\shc_{i(2)}|)+{{2+C_1}\over{2}}O_{\tau}(\shc).
$$
\noindent
which shows that

$$
||\mu||_{\infty} \le 2||\shc||_{\infty}C_2+{{2+C_1}\over{2}}O_{\tau}(\shc).
$$
\noindent
Set $C_3=2C_2+{{2+C_1}\over{2}}$. Since $||\shc||_{\infty} \le O_{\tau}(\shc)$ we get

\begin{equation}\label{mu-supremum-estimate}
||\mu||_{\infty} \le C_3O_{\tau}(\shc).
\end{equation}

\begin{remark} It follows from (\ref{mu-supremum-estimate}) that the Teichm\"uller norm of the tangent vector  
${{\partial{\psi}}\over{\partial{t}}}(t_0)$  is less or equal to $C_3O_{\tau}(\shc)$ for every $0<t_0<1$. 
Let $S_m$, $m \in \N$, be a sequence of surfaces equipped with  
triangulations $\tau_m$ and let $\shc_m \in X(\tau_m)$ be a sequence of vectors so that $O_{\tau_m}(\shc_m) \to 0$, $m \to \infty$. 
Since $||\shc_m||_{\infty} \le O_{\tau_m}(\shc_m)$ the constant $C_3$ in (\ref{mu-supremum-estimate}) does not depend on $m$. This implies that
the Teichm\"uller distance between the points $F_{\tau_m}(0)$ and $F_{\tau_m}(\shc_m)$ tends to $0$ when $m \to \infty$.
\end{remark}

\vskip .1cm
Let $0<\delta<1$. Let $k \in \N$ such that $k \ge 2$ and such that $k\delta \le 1$.
Assume that for every $1\le j \le k$ we have $|\shc_{i(j)}| \le \delta$. 
For $z \in E_1(t_0)$ which satisfies that
$$
y-u(t_0,x) \le (k-1)e^{-k \delta t_0},
$$
\noindent
from (\ref{y-condition}) we get that $k(z) \le k$.
Then from (\ref{mu})  we have
\begin{equation}\label{mu-estimate}
|\mu(z)| \le {{1}\over{2}}(2\delta C_2+k \delta+2 \delta C_2+k\delta+C_1 k \delta) 
\le C_4 k\delta,
\end{equation}
\noindent
where $C_4$ is a constant that depends only on the upper bound for $|\shc_{i(1)}|+|\shc_{i(2)}| \le 2\delta <1$ that is
$C_4$ is a universal constant.
\vskip .1cm
Furthermore

$$
\int\limits_{E_1(t_0)}|\mu(z)|\rho^{2}(z)\, dxdy=\int\limits_{E_1(t_0)}|\mu(z)|y^{-2}\, dxdy=
$$

$$
=\int\limits_{0}^{1}\left(\int\limits_{u(t_0,x)}^{(k-1)e^{-k \delta t_0}}|\mu(z)|y^{-2}\, dy 
+ \int\limits_{(k-1)e^{-k \delta t_0}}^{\infty}|\mu(z)|y^{-2}\, dy \right) \, dx.
$$
\noindent
In the first integral above we estimate $|\mu(z)|$ by (\ref{mu-estimate}) and in the second integral we estimate $|\mu(z)|$ by (\ref{mu-supremum-estimate}). Note that $0\le t_0 \le 1$ and since $k \ge 2$ we have 
${{1}\over{k-1}} \le {{2}\over{k}}$. From the assumption $k\delta \le 1$ we have $e^{k \delta} \le e$, therefore
$$
\int\limits_{(k-1)e^{-k \delta t_0}}^{\infty}|\mu(z)|y^{-2}\, dy \le  10{{O_{\tau}(\shc)}\over{k}}.
$$
\noindent
We conclude 

\begin{equation}\label{mu-int-1}
\int\limits_{E_1(t_0)}|\mu(z)|\rho^{2}(z)\, dxdy \le C_5 \big( k \delta+  {{O_{\tau}(\shc)}\over{k}} \big),
\end{equation}
\noindent
where $C_5$ depends only on the upper bound for $|\shc_{i(1)}|+|\shc_{i(2)}| \le 2\delta <1$, that is $C_5$ is a universal constant.
\vskip .1cm

\begin{theorem} Let $S$ be a surface of type $(\gen,\numb)$ (where $n \ge 1$) and let $\tau$ be a triangulation on
$S$ (where $\lambda$ denotes the corresponding set of edges). Fix $\shc \in X(\tau)$. Let $0<\delta< 1$ and 
denote by $N_{\tau}(\delta)$ the number of edges for which the absolute value of the corresponding shear  coordinate is greater than $\delta$. Then there is a constant $C$ that depends only on $||\shc||_{\infty}$ so that for any $k \in \N$ such that $k\delta \le 1$ and $k \ge 2$ we have
\vskip .1cm
\begin{equation}\label{WP-main}
d_{WP}(F_{\tau}(0),F_{\tau}(\shc)) \le
C_{7}  \sqrt{  O_{\tau}(\shc) \left( k \delta+  {{O_{\tau}(\shc) }\over{k}} +
{{N_{\tau}(\delta)}\over{|\lambda|}} O_{\tau}(\shc)  \right) }.
\end{equation}

\end{theorem}

\begin{remark} Important point in this theorem is that the constant $C$ does not depend on choice of the triangulation $\tau$ or on the type $(\gen,\numb)$. It only depends on the upper bound on $||\shc||_{\infty}$.
\end{remark}
 
\begin{proof}
Let $\wt{\lambda}=\{\lambda_i \in \lambda: |\shc_i|>\delta\}$. Then $|\wt{\lambda}|=N_{\tau}(\delta)$.
We enlarge the set $\wt{\lambda}$ to the set $\wh{\lambda}$ as follows. 
An edge from $\lambda$ belongs to $\wh{\lambda}$ if it belongs to a consecutive k-tuple
of edges that starts at an edge from $\wt{\lambda}$ (here we take both the left and the right oriented 
consecutive k-tuples). Note that
$$
|\wh{\lambda}|\le 4k|\wt{\lambda}|.
$$

By $\wh{\tau}$ we denote the set of triangles that have at least one of its edges in $\wh{\lambda}$.
Note that

\begin{equation}\label{hat-tau-estimate}
|\wh{\tau}|\le 2 |\wh{\lambda}|\le  8k|\wt{\lambda}|.
\end{equation}

To estimate $d_{WP}(F_{\tau}(o),F_{\tau}(\shc))$ we need to estimate 
$||{{ \partial{\psi}}\over{\partial{t}}}(t_0)||_{WP}$ for every $t_0 \in [0,1]$. Here $\psi:[0,1] \to \Tgn$
is the path defined above and we have

\begin{equation}\label{WP-integral}
d_{WP}(F_{\tau}(0),F_{\tau}(\shc)) \le \int\limits_{0}^{1} ||{{ \partial{\psi}}\over{\partial{t}}}(t)||_{WP}dt. 
\end{equation}

\vskip .1cm
Assume that a triangle $T$ does not belong to $\wh{\tau}$ and let $\lambda_{i(1)}$ be its edge. 
As above, by $\lambda_{i(1)},...,\lambda_{i(k)}$,
we denote the consecutive k-tuple of edges. For each $1 \le j \le k$, we have
$|\shc_{i(j)}|\le \delta$. Applying (\ref{mu-int-1}) on $E_1(t_0)$ (and repeating the same on $E_2(t_0)$ and $E_3(t_0)$)
we get
$$
\int\limits_{T(t_0)}|\mu(z)| \rho^{2}(z)\, dxdy \le 3C_5 \big( k \delta+  {{O_{\tau}(\shc)}\over{k}} \big).
$$
\noindent
On the other hand, for a triangle $T \in \wh{\tau}$ from (\ref{mu-supremum-estimate})  we get the estimate 
$$
\int\limits_{T(t_0)}|\mu(z)| \rho^{2}(z)\, dxdy \le \pi C_3 O_{\tau}(\shc).
$$
\noindent
This implies that (we express the hyperbolic area of $S_{t}$ as $\pi|\tau|={{2\pi}\over{3}}|\lambda|$)

$$
\int\limits_{S_{t}} |\mu(z)| \rho^{2}(z)\, dxdy \le C_6 \big( k \delta+  {{O_{\tau}(\shc) }\over{k}} +
{{N_{\tau}(\delta)}\over{|\lambda|}} O_{\tau}(\shc)  \big),
$$
\noindent
where $C_6$ is a constant that only depends on $||\shc||_{\infty}$.
Applying Lemma 2.1 and the last  estimate   we conclude that there is a constant $C_7$ that depends only on $||\shc||_{\infty}$ such that 
$$
||{{ \partial{\psi} } \over{ \partial{t} }} (t_0)||^{2}_{WP} \le 
C^{2}_{7}   O_{\tau}(\shc) \left( k \delta+  {{O_{\tau}(\shc) }\over{k}} +
{{N_{\tau}(\delta)}\over{|\lambda|}} O_{\tau}(\shc)  \right).
$$

\noindent
Together with (\ref{WP-integral}) this proves the  estimate (\ref{WP-main}) and this theorem.  
\end{proof}

The following theorem is the only result regarding Weil-Petersson distances that will be used later in this paper.
We use the notation from the previous theorem.

\begin{theorem} Let $r_0>0$ and assume that for every $r>r_0$ we are given a finite type surface $S(r)$ (with at least one puncture)  with an ideal triangulation $\tau(r)$ (the corresponding set of edges is $\lambda(r)$) that determines a point $\shc(r) \in X(\tau(r))$  with the following properties

\begin{itemize} 
\item There exists a universal constant $C>0$ such that $\{||\shc_m||_{\infty} \} \le C$ for every $r>r_0$. 
\item Let $\delta(r)=re^{-r}$. Then  there exists a polynomial $P(r)$ such that
$$
{{N_{\tau(r)}(\delta(r))}\over{|\lambda(r)|}}\le P(r)e^{-r}.
$$
\item We have $O_{\tau(r)}(\shc(r)) \le 2r^{2}$.
\end{itemize}
Then for some polynomial $P_1(r)$ we have
$$
d_{WP}(F_{\tau(r)}(0),F_{\tau(r)}(\shc(r)) )\le P_1(r)e^{ -{{r}\over{4}} }.
$$
\noindent
In particular  the Weil-Petersson distance between the points $F_{\tau(r)}(0)$ and $F_{\tau(r)}(\shc(r))$ tends to zero when $r \to \infty$.
\end{theorem}

\begin{proof} Let $k$ be an integer that is given by 
$$
e^{{{r}\over{2}}} \le k < e^{{{r}\over{2}}}+1.
$$
\noindent
Then $k \ge 2$ and $k\delta(r) <1$ for $r$ large enough. The proof follows directly from 
(\ref{WP-main}) for this choice of $k$.

\end{proof}

\section{Constructing finite degree covers of a punctured Riemann surface}
\subsection{Admissible collections of triangles} The aim of this section is to describe how to construct finite degree covers of a given punctured Riemann 
surface of finite type. We also need to keep the track of the geometry of these covers which is the motivation for the construction below. 
From now until the end of the paper  $S$ denotes  a fixed Riemann surface of type $(\gen,\numb)$ (we assume that $\numb \ge 1$). The corresponding set of punctures in the boundary of $S$ (or simply cusps) is denoted by $\Cus(S)$.  We number the cusps in $\Cus(S)$ as  $\Cus(S)=\{c_1(S),...,c_{\numb}(S)\}$. We also fix an ideal geodesic triangulation $\tau(S)$ on $S$ where $\lambda(S)$ denotes the corresponding set of the  edges. 
Most estimates we obtain will depend on the choice of $S$ and $\tau(S)$.
\vskip .1cm
Let $\Gamma(S)$ denote the set of all immersed ideal geodesics on $S$ (all geodesics are taken with respect to the underlining hyperbolic metric). 
A geodesic $\gamma$ on $S$ is said to be ideal if both of its endpoints are in the set $\Cus(S)$ (note that these geodesics can have self intersections).
We will say that a triangle $T$ on $S$ is an immersed ideal triangle if it lifts to an ideal triangle in the universal cover. The set of all
immersed ideal triangles is denoted by $\Tr(S)$. If $T \in \Tr(S)$ then its edges belong to $\Gamma(S)$. 
\vskip .1cm
A geodesic  $\gamma \in \Gamma(S)$ does not carry a natural orientation. That is, we have a choice of two orientations on each such $\gamma$. By $\Gamma^{*}(S)$
we denote the corresponding set of oriented geodesics. In particular, the set of oriented geodesics from $\lambda(S)$ will be called $\lambda^{*}(S)$.
For $\gamma^{*} \in \Gamma^{*}(S)$  the corresponding unoriented geodesic is typically denoted by $\gamma \in \Gamma(S)$.
\vskip .1cm 
Let $\N\Tr(S)$ denote the space of all formal sums of triangles from $\Tr(S)$ over non-negative integers. Similarly, by $\Z\Gamma^{*}(S)$ we denote the space of all formal sums of oriented geodesics over all integers subject to the following rule. If $\gamma_1^{*}$ and $\gamma_2^{*}$ represent the same geodesic with opposite orientations then we have the identity $\gamma_1^{*}=-\gamma_2^{*}$.
\vskip .1cm
We define the boundary operator $\partial:\N\Tr(S) \to \Z\Gamma^{*}(S)$ as follows. Let $T \in \Tr(S)$ and let $\gamma_i$, $i=1,2,3$, denote its edges.
Then $\partial{T}=\gamma_1^{*}+\gamma_2^{*}+\gamma_3^{*}$ where $\gamma_i^{*}$ are the corresponding oriented geodesics so that the triangle $T$
is to the left of each $\gamma_i^{*}$. In other words, the orientation of $\gamma_i^{*}$ agrees with the orientation that  $\gamma_i$ 
inherits as the part of the boundary of $T$ (such $\gamma^{*}_i$ is called an oriented edge of the triangle $T$). 
The operator $\partial$ is extended to $\N\Tr$ by linearity.
\vskip .1cm
Note that the space $\N\Tr(S)$ naturally embeds into the corresponding group of 2-chains on $S$. 
Also, $\Z\Gamma^{*}(S)$ naturally embeds into the corresponding group of 1-chains on $S$. The operator $\partial$ we defined agrees with the restriction
of the standard boundary operator on $\N\Tr(S)$. Let $\Hom$ denote the first homology of the surface $\overline{S}$ relative to the boundary 
of $S$ which is the set of cusps $\Cus(S)$. The class of each element of $\Z\Gamma^{*}(S)$ represents an element in $\Hom$. In particular,
for every $R \in \N\Tr(S)$ we have that $\partial{R}$ is equal to zero in $\Hom$ 
(although $\partial{R}$ is not necessarily equal to zero in $\Gamma^{*}(S)$).
\vskip .1cm
By $G$ we always denote a Fuchsian group so that $S$ is isomorphic to $\Ha/G$. The lift of $\tau(S)$, $\lambda(S)$, $\lambda^{*}(S)$, $\Tr(S)$, $\Gamma(S)$ and $\Gamma^{*}(S)$ under this uniformisation will be denoted by $\tau(G)$, $\lambda(G)$, $\lambda^{*}(G)$, $\Tr(G)$, $\Gamma(G)$ and $\Gamma^{*}(G)$. Also, by $\Cus(G)$ we denote the set of cusps in the boundary of $\Ha$ (we identify $\partial{\Ha}$ with $\overline{\R}$), that is, $\Cus(G)$ is the set of fixed points of all parabolic elements in $G$. We can identify the quotient $\Cus(G)/G$ with the set of punctures in the boundary of $S$ which we denoted by $\Cus(S)$. If $T \in \Tr(G)$ by $[T]_G$ we denote the orbit of $T$ under $G$. We identify $[T]_G$ with the corresponding
element of $\Tr(S)$. 
\vskip .1cm
The set $\Tr^{*}(G)$ is defined as the set of pairs $(T,\gamma^{*})$ where $T \in \Tr(G)$ and $\gamma^{*} \in \Gamma^{*}(G)$ 
is an oriented edge of $T$. If $g \in G$ then $(g(T),g(\gamma^{*})) \in \Tr^{*}(G)$ as well. Set $\Tr^{*}(S)=\Tr^{*}(G)/G$ (this definition does not depend on the choice of the group $G$). The equivalence class of $(T,\gamma^{*})$ is denoted by $[(T,\gamma^{*})]_G \in \Tr^{*}(S)$.
Set $\Proj_S([(T,\gamma^{*})]_G)=[T]_G \in \Tr(S)$. Then $\Proj_S:\Tr^{*}(S) \to \Tr(S)$ is  well defined.

\begin{remark} Let $T \in \Tr(S)$ and let $\gamma^{*}_1,\gamma^{*}_2,\gamma^{*}_3$ denote its oriented edges. 
It can happen that $\gamma^{*}_1=\gamma^{*}_2$. This is why we can not define  $\Tr^{*}(S)$ as the set of pairs $(T,\gamma^{*})$ where $T \in \Tr(S)$ and $\gamma^{*} \in \Gamma^{*}(S)$ is an oriented edge of $T$. 
\end{remark}

If $(T_1,\gamma^{*}),(T_2,-\gamma^{*}) \in \Tr^{*}(G)$ then   $[(T_1,\gamma^{*})]_G \ne [(T_2,-\gamma^{*})]$. To prove this, assume that 
$[(T_1,\gamma^{*})]_G = [(T_2,-\gamma^{*})]$. Then there is $g \in G$ so that $g(\gamma^{*})=-\gamma^{*}$. Also, such $g$ is not the identity map.
But then $g^{2}(\gamma^{*})=\gamma^{*}$. This implies that $g$ fixes the endpoints $c_1,c_2 \in \Cus(G)$ of $\gamma^{*}$. Since  $g^{2}(c_1)=c_1$ we have that $g$ is a parabolic element of $G$ and therefore $c_1$ is the unique fixed point of $g^{2}$. This contradicts the fact that $g^{2}(c_2)=c_2$. 
\vskip .1cm
Fix  $T \in \Tr(G)$ and let $\Rot_T:\Ha \to \Ha$ be the standard rotation around the centre of $T$ for the angle $2\pi/3$ (this rotation is of the order three).  For  $(T,\gamma^{*}) \in \Tr^{*}(G)$ set $\rot_G(T,\gamma^{*})=(T,\Rot_T(\gamma^{*})) \in \Tr^{*}(G)$. 
The map $\rot_G:\Tr^{*}(G) \to \Tr^{*}(G)$ is a bijection and $\rot^{3}_G$ is the identity map. It is obvious that $\rot_G$ and $\rot^{2}_G$ have no fixed points. If $g \in G$ then $g \circ \Rot_T=\Rot_{g(T)} \circ g$. This shows that the map $\rot_S:\Tr^{*}(S) \to \Tr^{*}(S)$ given by
$\rot_S([(T,\gamma^{*})]_G)=[(T,\Rot_T(\gamma^{*}))]_G$ is well defined. Assume that $\rot_S([(T,\gamma^{*})]_G)=[(T,\gamma^{*})]_G$. Then there is
$g \in G$ so that $g(T)=T$ and so that $g(\gamma^{*})=\Rot_T(\gamma^{*})$. This implies that $g$ and $\Rot_T$ agree at the vertices of $T$ which shows that $g=\Rot_T$. This is not possible since $G$ is torsion free. The conclusion is that the map $\rot_S$ has no fixed points. Since $\rot^{3}_S$ is the identity map, we conclude that $\rot^{2}_S$ has no fixed points either. Note that $\Proj_S \circ \rot_S=\Proj_S$.
\vskip .1cm
Let $\Lab_{\Col}$ be a finite set of labels, and let $\lab_{\Col}:\Lab_{\Col}\to \Tr^{*}(S)$ be a labelling map. We say that the pair $\Col=(\Lab_{\Col},\lab_{\Col})$ is a labelled collection of triangles if the following holds:
\begin{itemize}
\item  There exists a bijection $\rot_{\Col}:\Lab_{\Col} \to \Lab_{\Col}$ so that $\rot^{3}_{\Col}$ is the identity map.
\item We have $\lab_{\Col} \circ \rot_{\Col}=\rot_S \circ \lab_{\Col}$.
\end{itemize}
Since $\rot_S$ and $\rot^{2}_S$ have no fixed points, we see that $\rot_{\Col}$ and $\rot^{2}_{\Col}$ have no fixed points either.
\vskip .1cm
Choose $\gamma^{*} \in \Gamma^{*}(S)$ and let $\gamma^{*}_1$ be its lift to $\Ha$. We say that $a \in \Lab_{\Col,{\gamma}^{*}} \subset \Lab_{\Col}$
if $\lab_{\Col}(a)=[(T,\gamma^{*}_1)]_G$ for some $T \in \Tr(G)$ so that $(T,\gamma^{*}_1) \in \Tr^{*}(G)$. Then $\sigma_{\Col}(\Lab_{\Col,{\gamma}^{*}})
=\Lab_{\Col,-{\gamma}^{*}}$. It follows from the above discussion that the sets $\Lab_{\Col,{\gamma}^{*}}$ and $\Lab_{\Col,-{\gamma}^{*}}$ are disjoint (if the were not there would be an order two element in $G$ as we showed above). Let $\gamma_i \in \Gamma(S)$, $i=1,...,k$, be the set of different edges of triangles from  $\Proj_S(\lab_{\Col})(\Lab_{\Col})$. Choose an orientation for $\gamma^{*}_i$ for each $\gamma_i$. Then $\Lab_{\Col}$ is the disjoint union 
$\Lab_{\Col}=  \Lab_{\Col,{\gamma}^{*}_{1}} \cup \Lab_{\Col,-{\gamma}^{*}_{1}} \cup...\cup \Lab_{\Col,{\gamma}^{*}_{k}} \cup \Lab_{\Col,-{\gamma}^{*}_{k}}$.
\vskip .1cm
Set $(\Proj_S \circ \lab_{\Col})(\Lab_{\Col})=\{ T_1,T_2,...,T_m\} \subset \Tr(S)$. 
Let $k_i=|(\Proj_S \circ \lab_{\Col})^{-1}(T_i)|$,  where  $|(\Proj_S \circ \lab_{\Col})^{-1}(T_i)|$ is the number of elements in the preimage 
$(\Proj_S \circ \lab_{\Col})^{-1}(T_i)$.  Since $\rot_{\Col}$ has no fixed points, and since $\Proj_S \circ \rot_S=\Proj_S$ 
we have that $k_i=3l_i$, for some integer $l_i$. Then each  $\Col$ induces an element in $\N\Tr(S)$ given by $k_1T_1+kT_2+...+k_mT_m=3l_1T_1+3l_2T_2+...+3l_mT_m$.
If there is no confusion we will denote this element of $\N\Tr(S)$ by $\Col$ as well. We define  $\partial\Col$ as the boundary of the corresponding element of $\N\Tr(S)$.

\begin{remark} Every $R \in \N\Tr(S)$ where $R=l_1T_1+l_2T_2+...+l_mT_m$ and $l_i$ are positive integers,
induces a labelled collection of triangles as follows (this will be used in the proof of Theorem 1.1 in the next section). 
The corresponding set of labels is $\Lab_{\Col}=\{(i,i'): i=1,2,...,(l_1+l_2+...+l_m);i'=1,2,3 \}$
(note that the set $\Lab_{\Col}$ has $3(l_1+...+l_m)$ elements). 
Let $T'_j \in \Tr(G)$ be a lift of $T_j$. Let $\gamma^{*}_{i'}$, $i'=1,2,3$, be its oriented edges, so that $\Rot_T(\gamma^{*}_1)=\gamma^{*}_2$.
The corresponding labelling map  $\lab_{\Col}$ is given by $\lab_{\Col}(i,i')=[(T'_{j},\gamma^{*}_{i'})]_G$, for $l_1+...+l_{j-1}<i \le l_1+...+l_{j}$.
The required bijection $\rot_{\Col}:\Lab_{\Col} \to \Lab_{\Col}$ is given by $\rot_{\Col}(i,j)=(i,j+1 \mod 3)$.
\end{remark}

\begin{definition} Let $\Col$ be a labelled collection of triangles. Let $\sigma_{\Col}:\Lab_{\Col} \to \Lab_{\Col}$ be a bijection.
We say that the pair $(\Col,\sigma_{\Col})$ is an $admissible$ $pair$ if the following holds:
\begin{itemize}
\item The map $\sigma_{\Col}$ is an involution, that is, $\sigma^{2}_{\Col}$ is the identity map, and $\sigma_{\Col}$ has no fixed points.
\item If $a,b \in \Lab_{\Col}$ and $\sigma_{\Col}(a)=b$ then there are triangles $T_1,T_2 \in \Tr(G)$ and $\gamma^{*} \in \Gamma^{*}(G)$ such that
$\lab_{\Col}(a)=[(T_1,\gamma^{*})]_G$ and $\lab_{\Col}(b)=[(T_2,-\gamma^{*})]_G$ 
(recall that $-\gamma^{*}$ denotes the opposite orientation of $\gamma^{*}$).
\end{itemize}
\end{definition}
If $(\Col,\sigma_{\Col})$ is an admissible pair, we have two bijections $\sigma_{\Col},\rot_{\Col}:\Lab_{\Col} \to \Lab_{\Col}$. The group generated by these two bijections is denoted by $\left< \sigma_{\Col},\rot_{\Col} \right>$.
\vskip .1cm
One can construct examples of admissible pairs as follows. Let $S_1 \to S$ be a cover, where $S_1$ is a finite type surface (recall that in this paper all the covers, except the universal cover, are assumed to be finite degree, regular and holomorphic). Then there is a Fuchsian group $G_1<G$
so that $S_1$ is isomorphic to $\Ha/G_1$. Let $\tau_1(S_1)$ be a geodesic triangulation of $S_1$ and let $\tau_1(G_1)$ denote its lift to $\Ha$. 
By $\tau^{*}_1(G_1)$ we denote the set of pairs $(T,\gamma^{*})$ where $T \in \tau_1(G_1)$ and $\gamma^{*}$ is an oriented edge of $T$. 
Same as above, the group $G_1$ acts on $\tau^{*}_1(G_1)$. By $[(T,\gamma^{*})]_{G_{1}}$ we denote the corresponding orbit, and by $\tau^{*}_1(S_1)=\tau^{*}_1(G_1)/G_1$ we denote the corresponding set of orbits. Same as above, we define the map $\rot_{S_{1}}:\tau^{*}_1(S_1)\to \tau^{*}_1(S_1)$. Then $\rot_{S_{1}}$ is a bijection of order three. 
\vskip .1cm
Set $\Lab_{\Col}=\tau^{*}_1(S_1)$ and let $\rot_{\Col}=\rot_{S_{1}}$. 
For $[(T,\gamma^{*})]_{G_{1}} \in \tau^{*}_1(S_1)$ set $\lab_{\Col}([(T,\gamma^{*})]_{G_{1}})=[(T,\gamma^{*})]_G \in \Tr^{*}(S)$. 
One can verify that $\lab_{\Col} \circ \rot_{\Col}=\rot_S \circ \lab_{\Col}$. Then $\Col=(\Lab_{\Col},\lab_{\Col})$ is a labelled collection of triangles. We now define the involution $\sigma_{\Col}$. 
\vskip .1cm
Let $a \in \Lab_{\Col}$. We want to define $\sigma_{\Col}(a) \in \Lab_{\Col}$. Let $T \in \tau_1(G_1)$ and $\gamma^{*} \in \Gamma^{*}(G)$ be its oriented edge,
so that $a=[(T,\gamma^{*})]_{G_{1}}$. Let $T_1$ be the unique triangle in $\tau_1(G_1)$ that is adjacent to $T$ along $\gamma^{*}$.  Set
$\sigma_{\Col}(a)=[(T_1,-\gamma^{*})]_{G_{1}}$. Clearly $(T_1,-\gamma^{*}) \in \tau^{*}_1(G_1)$ so $\sigma_{\Col}(a)$ is well defined. Moreover, $\sigma_{\Col}$ is of order two.
If $\sigma_{\Col}$ had a fixed point, then there would exist an element of order two in $G$ which is not possible. We have that $(\Col,\sigma_{\Col})$ is an admissible pair.
\begin{definition} The above constructed pair $(\Col,\sigma_{\Col})$ is called a $virtual$ $triangulation$ $pair$. 
\end{definition}
If $(\Col(i),\sigma_{\Col(i)})$, $i=1,...,m$, are  virtual triangulations pairs, then we can construct a new admissible pair $(\Col,\sigma_{\Col})$
as follows. Set $\Lab_{\Col}=\Lab_{\Col(1)} \cup...\cup \Lab_{\Col(m)}$ (here we assume that the sets of labels  $\Lab_{\Col(i)}$ are mutually disjoint). 
On each $\Col(i)$ the map $\rot_{\Col}$ agrees with $\rot_{\Col(i)}$. Also, on each $\Lab_{\Col(i)}$ the map $\lab_{\Col}$ agrees with the map
$\lab_{\Col(i)}$. Then $\Col$ is a labelled collection of triangles. We define $\sigma_{\Col}$ to agree with $\sigma_{\Col(i)}$ on each $\Lab_{\Col(i)}$.
We have that $(\Col,\sigma_{\Col})$ is an admissible pair.

\begin{definition}  Let $(\Col(i),\sigma_{\Col(i)})$, $i=1,...,m$, be admissible pairs. We say that an admissible pair
$(\Col,\sigma_{\Col})$ is the union of  $(\Col(i),\sigma_{\Col(i)})$, $i=1,...,m$, if the following holds. 
\begin{itemize}
\item There exist injections $\phi_i:\Lab_{\Col(i)} \to \Lab_{\Col}$ so that the set $\Lab_{\Col}$ is the disjoint union of 
the sets $\phi_i(\Lab_{\Col(i)})$.
\item For each $\phi_i$ we have $\lab_{\Col} \circ \phi_i=\lab_{\Col(i)}$.
\item We have $\phi \circ \rot_{\Col(i)} =\rot_{\Col} \circ \phi$ and $\phi \circ \sigma_{\Col(i)} =\sigma_{\Col} \circ \phi$.
\end{itemize}
\end{definition}

We have the following lemma.

\begin{lemma} For every admissible triangulation pair $(\Col,\sigma_{\Col})$ there exist  virtual triangulation pairs $(\Col(i),\sigma_{\Col(i)})$, $i=1,...,l$, so that  $(\Col,\sigma_{\Col})$ is the union of $(\Col(i),\sigma_{\Col(i)})$.
\end{lemma}

\begin{proof}  We divide $\Lab_{\Col}$ into the orbits of the group $\left< \sigma_{\Col},\rot_{\Col} \right>$. Denote these orbits by $\Lab_{\Col(i)}$ where $i$ goes through some finite set of labels. Let $\lab_{\Col(i)}$, $\sigma_{\Col(i)}$, and $\rot_{\Col(i)}$ denote the restrictions of the the maps $\lab_{\Col}$, $\sigma_{\Col}$, and $\rot_{\Col}$ on the set $\Lab_{\Col(i)}$. Then each pair $\Col(i)=(\Lab_{\Col(i)},\lab_{\Col(i)})$ is an admissible pair, and $(\Col,\sigma_{\Col})$ is the union of $(\Col(i),\sigma_{\Col(i)})$. Moreover, the group $\left< \sigma_{\Col(i)},\rot_{\Col(i)} \right>$ acts transitively on $\Lab_{\Col(i)}$.  Therefore, in order to prove the lemma, it suffices to prove that every admissible pair $(\Col,\sigma_{\Col})$ for which the group  $\left< \sigma_{\Col},\rot_{\Col} \right>$ acts transitively on $\Lab_{\Col}$ is in fact a virtual triangulation pair.
\vskip .1cm
Assume that  $(\Col,\sigma_{\Col})$ is an admissible pair, so that the group  $\left< \sigma_{\Col},\rot_{\Col} \right>$ acts transitively on $\Lab_{\Col}$. Let $\Tr^{*}_{\Col}(G)$ be the space of triples $(T,\gamma^{*},a)$ where  $(T,\gamma^{*}) \in \Tr^{*}(G)$ and where $\lab_{\Col}(a)=[(T,\gamma^{*})]$.
There is a unique involution $\sigma_{\Col}(G):\Tr^{*}_{\Col}(G) \to \Tr^{*}_{\Col}(G)$ such that $\sigma_{\Col}(G)((T_1,\gamma^{*}_1,a_1))=(T_2,\gamma^{*}_2,a_2)$ if and only if $-\gamma^{*}_1=\gamma^{*}_2$ and $a_1=\sigma_{\Col}(a_2)$. Define
$\rot_{\Col}(G):\Tr^{*}_{\Col}(G) \to \Tr^{*}_{\Col}(G)$ by $\rot_{\Col}(G)((T,\gamma^{*},a))=(T,\Rot_T(\gamma^{*}),\rot_{\Col}(a))$. 
It is easy to verify that the group $\left< \sigma_{\Col}(G),\rot_{\Col}(G) \right>$ is isomorphic to $\Z_2 \star \Z_3$ and that it acts freely on 
$\Tr^{*}_{\Col}(G)$. Moreover, the collection of triangles from $\Tr(G)$ that appears in a given orbit under this action, is an ideal triangulation of $\Ha$.
\vskip .1cm
The group $G$ naturally acts on $\Tr^{*}_{\Col}(G)$ by $g(T,\gamma^{*},a)=(g(T),g(\gamma^{*}),a) \in \Tr^{*}_{\Col}(G)$ (note that $[(T,\gamma^{*})]=[(g(T),g(\gamma^{*}))]$). If  $(T_1,\gamma^{*}_1,a),(T_2,\gamma^{*}_2,a)  \in \Tr^{*}_{\Col}(G)$ then by definition we have
$[(T_1,\gamma^{*}_1)]=[(T_2,\gamma^{*}_2)]$. That is, there is $g \in G$ so that $g(T_1,\gamma^{*}_1,a)=(T_1,\gamma^{*}_1,a)$.
In particular, this shows that the group generated by $G$ and $\left< \sigma_{\Col}(G),\rot_{\Col}(G) \right>$ acts transitively on $\Tr^{*}_{\Col}(G)$.
Also, every element from $G$ commutes with every element from $\left< \sigma_{\Col}(G),\rot_{\Col}(G) \right>$.
\vskip .1cm 
Fix $(T,\gamma^{*},a) \in \Tr^{*}_{\Col}(G)$ and consider the orbit  $O=\left< \sigma_{\Col}(G),\rot_{\Col}(G) \right>(T,\gamma^{*},a)$. 
Let $G_1$ be the subgroup of $G$ so that $g \in G_1$ if $g$ preserves the orbit $O$. If  $(T_1,\gamma^{*}_1,a),(T_2,\gamma^{*}_2,a)  \in O$
then there is $g \in G$ so that $g(T_1,\gamma^{*}_1,a)=(T_1,\gamma^{*}_1,a)$. 
Moreover, since $g$ commutes with elements from  $\left< \sigma_{\Col}(G),\rot_{\Col}(G) \right>$ we conclude that $g$  preserves the entire orbit $O$. This shows that the map $\psi:O/G_1 \to \Lab_{\Col}$ given by $\phi(T,\gamma^{*},a)=a$ is an injection. Since the group $\left< \sigma_{\Col},\rot_{\Col} \right>$ acts transitively on $\Lab_{\Col}$ we find that $\psi$ is a bijection. 
\vskip .1cm
Let $S_1$ be the Riemann surface isomorphic to $\Ha/G_1$. Set $\tau^{*}(S_1)=O/G_1$ and denote by $\tau(S_1)$ the  union of triangles that appear in $O/G_1$. Then $\tau(S_1)$  is an ideal triangulation of the surface $S_1$. Since $\tau(S_1)$ contains  finitely many such triangles (the map $\psi:\tau^{*}(S_1) \to \Lab_{\Col}$ is a bijection onto a finite set $\Lab_{\Col}$), we see that $S_1$ is a finite type surface, and the group $G_1$ has finite index in $G$. Now it follows that the admissible pair $(\Col,\sigma_{\Col})$ is a virtual triangulation pair, that via the map $\psi$ corresponds to the triangulation of $S_1$.

\end{proof}

\vskip .1cm
Consider a labelled collection of triangles $\Col$ as an element of $\N\Tr(S)$. Assume that $(\Col,\sigma_{\Col})$ is an admissible pair.
We saw above that  $\Lab_{\Col}$ is the disjoint union  $\Lab_{\Col}=  \Lab_{\Col,{\gamma}^{*}_{1}} \cup \Lab_{\Col,-{\gamma}^{*}_{1}} \cup...\cup \Lab_{\Col,{\gamma}^{*}_{k}} \cup \Lab_{\Col,-{\gamma}^{*}_{k}}$ where $\gamma_i \in \Gamma^{*}(S)$ are the different edges of triangles
from $\Proj_S(\lab_{\Col}(\Lab_{\Col}))$. Then $\sigma_{\Col}(\Lab_{\Col,{\gamma}^{*}_{i}})
=\Lab_{\Col,-{\gamma}^{*}_{i}}$. 
We conclude that the sets $\Lab_{\Col,{\gamma}^{*}_{i}}$ and $\Lab_{\Col,-{\gamma}^{*}_{i}}$ have the same number of elements, which implies that
$\gamma^{*}_{i}$ does not figure in $\partial{\Col}$. We have that  $\partial{\Col}$ is equal to zero in the space $\Z\Gamma^{*}(S)$. 
\vskip .1cm
On the other hand, if $\Col$ is a labelled collection of triangles so that $\partial{\Col}$ is zero in
$\Z\Gamma^{*}(S)$ then  the corresponding  sets $\Lab_{\Col,{\gamma}^{*}}$  and $\Lab_{\Col,-{\gamma}^{*}}$
have the same number of elements. Again, from the fact that  $\Lab_{\Col}$ is the disjoint union  $\Lab_{\Col}=  \Lab_{\Col,{\gamma}^{*}_{1}} \cup \Lab_{\Col,-{\gamma}^{*}_{1}} \cup...\cup \Lab_{\Col,{\gamma}^{*}_{k}} \cup \Lab_{\Col,-{\gamma}^{*}_{k}}$ we can construct
an appropriate involution $\sigma_{\Col}$. However, there are many ways in which we can do this (unless each $\Lab_{\Col,{\gamma}^{*}_{i}}$ has one element). This is where the geometric considerations start. Our aim will be to construct these involutions so that the pair of triangles that are related by  the involution satisfies that the hyperbolic distance between the orthogonal projections of the centres of these triangles to their common edge,
is as small as possible.

\subsection{The definitions of height, $0$-horoball and combinatorial length}
As we said, $G$  denotes a Fuchsian group so that $S$ is isomorphic to $\Ha/G$. 
For every  $c \in \Cus(G)$ by $G(c)$ we denote the orbit of $c$ under the group $G$. Each such orbit corresponds to a cusp $c_i(S) \in \Cus(S)$.
Fix $c_i(S)$ that is we fix the orbit of some $c \in \Cus(G)$.  Let $f:\Ha \to \Ha$  be a M\"obius transformation, so that $f(c)=\infty$ and so that 
the parabolic element in in $f \circ G \circ f^{-1}$ that generates the corresponding cyclic subgroup of $f \circ G \circ f^{-1}$ , is the translation $g_{\infty}(z)=z+1$ (there are many such M\"obius transformations $f$ but we choose one for each orbit). The group  $G_{c_{i}}=G_c=f \circ G \circ f^{-1}$ is called the normalised group with respect to the cusp $c$. If $c' \in G(c)$ then $G_{c_{i}(S)}=G_c=G_{c'}$. There are exactly $\numb$ normalised groups $G_c$ (where $\numb$ is the number of punctures in the boundary of $S$). From now on $G$ is  one of the normalised groups (if we do not specify which one, then it can be any one). The constants we introduce below  may depend on  the choice of the group $G$. However, since we consider only finitely many such groups, these constants  depend only on  $S$.
 
\begin{definition} Let $z \in \Ha$ $c \in \Cus(G)$. Let $f$ be the corresponding M\"obius transformation that conjugates $G$ onto $G_c$.
We have:
\begin{itemize}
\item Set $\h_c(z)=\log(\IM(f(z)))$ where $\IM(f(z))$ is the imaginary part of $f(z) \in \Ha$. We say that $\h_c(z)$ is the height
of the point $z$ with respect to the cusp $c \in \Cus(G)$. 
\item Let $T \in \Tr(G)$ and denote by $c_1,c_2,c_3 \in \Cus(G)$ its vertices. 
Set $\h(T)=\max \{|\h_{c_{1}}(\ct(T))|, |\h_{c_{2}}(\ct(T))|,|\h_{c_{3}}(\ct(T))|\}$. Note that $h(T) \ge 0$.
\item Set $\h(z)=\max_{c \in \Cus(G)}\h_c(z)$. 
\item Let $c \in \Cus(G)$ and let $\gamma \in \Gamma(G)$ be such so that $c$ is not one of its endpoints. Denote by $\z_{\max}(\gamma,c)$ the point on $\gamma$ so that $\h_c(\z_{\max}(\gamma,c))\ge \h_c(z)$ for any other point $z \in \gamma$. 
\item Let $t \in \R$. Set $\Hor_c(t)=\{z \in \Ha:\h_c(z) \ge t \}$. We say that $\Hor_c(t)$ is the $t$-horoball at $c$. 
\item Let $t\in \R$. We set $\Thick_G(t)=\{z \in \Ha:\h(z)\le t\}$ and $\Thin_G(t)=\Ha \setminus \Thick_G(t)$.
\end{itemize}
\end{definition}
\begin{remark} We establish the following related definitions. If $T_1,T_2$ are two triangles so that $f(T_1)=T_2$ for some $f \in G$ then $\h(T_1)=T_2$. This shows that $\h(T)$ is well defined for every $T \in \Tr(S)$. 
If $c_1,c_2$ are equivalent under $G$ then the projection of the horoballs $\Hor_{c_{1}}(t)$ and $\Hor_{c_{2}}(t)$ to $S$ agree. Therefore we can define the $t$-horoball $\Hor_c(t) \subset S$ for every  $c \in \Cus(S)$. It is well known that any two $0$-horoballs on $S$ either coincide or  are disjoint (note that the closures of two zero horoballs for $\PSL$ may touch, but if a Fuchsian group $G$ is a covering group of a Riemann surface, then this can not happen) . Therefore, for any $z \in \Ha$ there can exist at most one cusp $c \in \Cus(G)$ so that $\h_c(z)>0$. Moreover, given $z \in \Ha$ and any interval $[a,\infty) \subset \R$ there are at most finitely many cusps $c \in \Cus(G)$ so that $\h_c(z) \in [a,b]$.
This shows that $\h(z)$  is well defined. If $g \in G$ then  $\h(z)=\h(g(z))$. This shows that for $p \in S$ the value $\h(p)$ is well defined. Also, by $\Thick_S(t)=\{p \in S:\h(p)\le t\}$ we denote the projection of the set $\Thick_G(t)=\{z \in \Ha:\h(z)\le t\}$ to $S$. We have $\Thin_S(t)=S \setminus \Thick_S(t)$.
\end{remark}

\vskip .1cm
In the following proposition we introduce certain constants that will be used in the proof of the correction lemma below. The constant $N(S)$ will be used throughout the paper. Since $S$ is of finite type the following numbers are well defined.
\begin{proposition} There exist a large enough number $t_1(S)>0$ a sufficiently negative number $t_2(S)<0$ and an integer $N(S) \ge 1$ with the following properties:
\begin{itemize}
\item Let $c \in \Cus(G)$ and let $N_{c}(S)$ denote the number of non-equivalent geodesics from $\lambda(G)$ that end at $c$. Then the integer 
$N(S)=\max_{c \in \Gamma(G)}N_c(S)$ is well defined.
\item Let $c \in \Cus(G)$ and let $z \in \Hor_c(-\log(8N(S)))$. Then $t_1(S)$ is large enough so that $z$ does not belong to any $t_1(S)$-horoball,
except possibly the one at $c$. 
\item We have $e^{t_1(S)}>4N(S)$ and $e^{t_1(S)}>4+2N(S)+1$.
\item Let $\gamma \in \lambda(G)$ and let $c_1,c_2 \in \Cus(G)$ be its endpoints. Let $T \in \tau(G)$ be one of the two triangles that have $\gamma$ in their boundary. Then the number $t_2(S)$ is sufficiently negative, so that for every point $z \in \Thick_G(t_1(S))$ and $z \in \gamma$ we have  $z \in \Hor_{c_{1}}(t_2(S))$ and $z \in \Hor_{c_{2}}(t_2(S))$.
\end{itemize}
\end{proposition}
\begin{remark} If  $t>t'$ then for every $c \in \Cus(G)$ we have $\Hor_{c}(t) \subset \Hor_{c}(t')$.
\end{remark}
\begin{proof} There are only $|\lambda(S)|$ non-equivalent geodesics in $\lambda(G)$. Also, the set $\gamma \cap \Thick_G$ is compact in $\Ha$. The propositions follows from the basic compactness argument.
\end{proof}

\begin{figure}
  {
    \def\H{{\mathcal H}}
    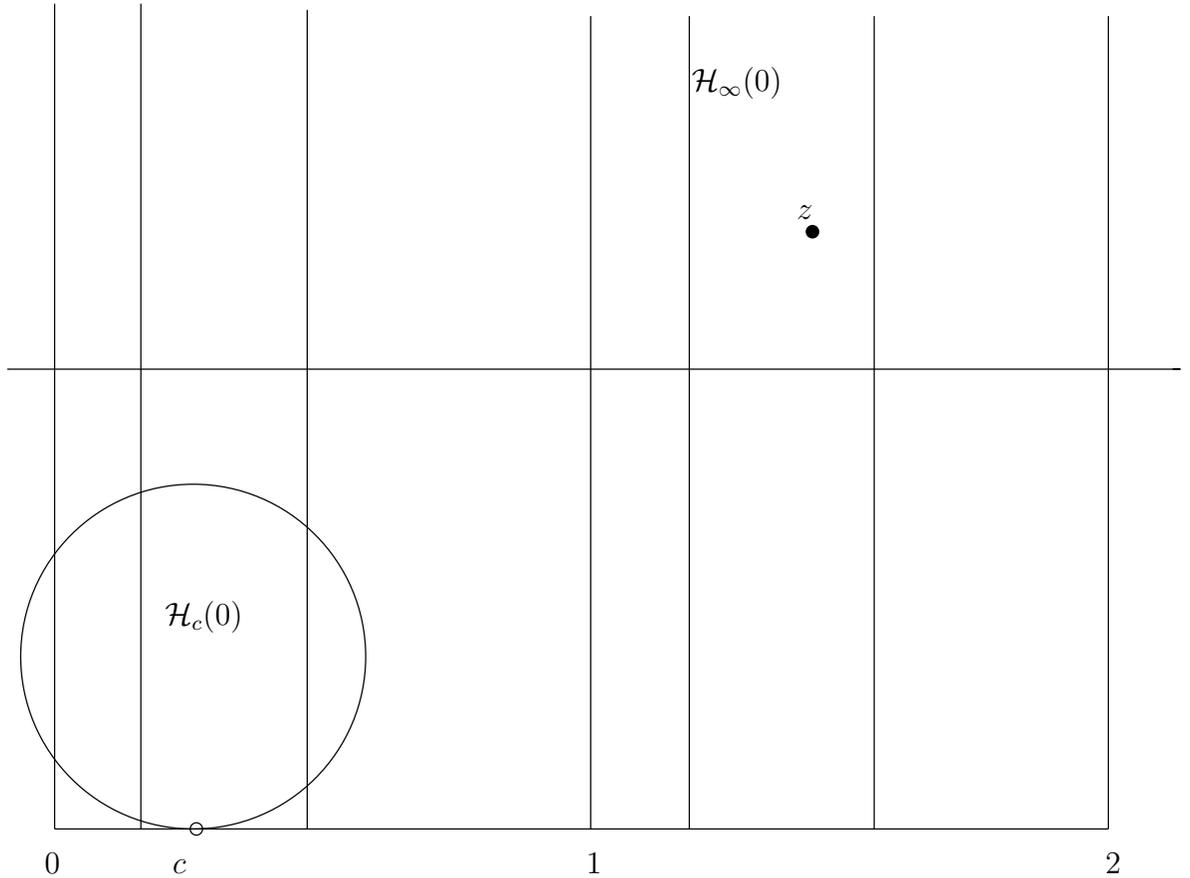
  }
  \caption{The case $N(S) = 3$; $h_\infty(z) = h(z) > 0$}
  \label{fig:horoball}
\end{figure}

If $\gamma \subset \Ha$ is an arbitrary geodesic segment (or finite or infinite length), by $\iota(\gamma,\tau(G))$ we denote the number of (transverse) intersections between $\gamma$ and edges from $\lambda(G)$. In particular, if $\gamma \in \Gamma(G)$ the intersection number   $\iota(\gamma,\tau(G))$
is finite. For a geodesic segment $\gamma \subset S$ by $\iota(\gamma,\tau(S))$ we denote the number of (transverse) intersections between $\gamma$ and edges from $\lambda(S)$.

\begin{definition}  Let $\gamma \in \Gamma(G)$ and let $z \in \gamma$. Let $c_1,c_2 \in \Cus(G)$ be the endpoints of $\gamma$. We define the $combinatorial$ $length$ $\cl(\gamma,z)$ as
$$
\cl(\gamma,z)=\iota(\gamma,\tau(G))+ \psi(c_1,c_2,z),
$$
\noindent
where $\psi(c_1,c_2,z)=0$ if $\max\{\h_{c_{1}(z)},\h_{c_{2}(z)} \} \le t_{1}(S)$ and 
$$
\psi(c_1,c_2,z)=\max\{[e^{\h_{c_{1}(z)}}],[e^{\h_{c_{2}(z)}}] \},
$$ 
\noindent
if $\max\{\h_{c_{1}(z)},\h_{c_{2}(z)} \} > t_{1}(S)$. Here $[e^{\h_{c_{1}(z)}}]$ and $[e^{\h_{c_{2}(z)}}]$ denote the integer parts of
$e^{\h_{c_{1}(z)}}$ and $e^{\h_{c_{2}(z)}}$ respectively. Note that if $z \in \Thick_G(t_1(S))$ then $\cl(\gamma,z)=\iota(\gamma,\tau(G))$.
\end{definition}
Since the 0-horoballs are disjoint, only one of the numbers $\h_{c_{1}(z)},\h_{c_{2}(z)}$ can be non-negative for any $z \in \Ha$. 
If $\gamma^{*} \in \Gamma^{*}(G)$ represents $\gamma$ together with a choice of orientation on $\gamma$ then we set $\cl(\gamma^{*},z)=\cl(\gamma,z)$. 
\vskip .1cm
The following lemma will be used in the proof of Theorem 1.1 (in the next section). 
\begin{lemma} Let $R$ be a finite type Riemann surface that covers $S$ and let $\tau(R)$ be a geodesic triangulation on $R$ (by $\lambda(R)$ we denote the corresponding set of edges). Let $\shc \in X(\tau(R))$ be the vector so that $R$ is the underlying Riemann surface for the point $F_{\tau(R)}(\shc)$
(see Section 2.). Let $A=\max_{T \in \tau(R)}|\h(T)|$. Then $O_{\tau(R)}(\shc) \le 2A$.
\end{lemma}
\begin{proof} Let $G_1$ be a finite index subgroup of $G$ so that $R$ is isomorphic to $\Ha/G_1$. By $\tau(G_1)$ and $\lambda(G_1)$ we denote the corresponding lifts of $\tau(R)$ and $\lambda(R)$. We have $\tau(G_1) \subset \Tr(G)$ so we can define the height $\h(T)$ for every $T \in \tau(G_1)$
(and by projecting $T$ to $R$ we define $\h(T)$ for $T \in \tau(R)$). 
\vskip .1cm 
The sets of cusps for $S$ and $R$ agree. Choose $c \in \Cus(G)$ and let $G$ be normalised, and $c=\infty$. Let $\lambda_1,...,\lambda_k \in \lambda(G_1)$ be a k-tuple of consecutive edges that all end at $\infty$. Let $\shc_i \in \R$ denote the corresponding shear coordinates on $\lambda_i$
$i=1,...,k$. Let $T_i \in \tau(G_1)$ be the triangle that has $\lambda_i$ as its edge, and so that $T_i$ is to the left of $\lambda^{*}_i$ where the orientation $\lambda^{*}_i$ is chosen so that $+\infty$ is to the right of $\lambda^{*}_i$. Then
$$
\IM(\ct(T_k))=e^{\shc_1+...+\shc_{k-1}}\IM(\ct(T_1)),
$$
\noindent
and
$$
|\shc_1+...+\shc_{k-1}|=|\log(\IM(\ct(T_k)))-\log(\IM(\ct(T_1)))|\le |\log(\IM(\ct(T_k)))|+|\log(\IM(\ct(T_1)))|=
$$
$$
=|\h_{\infty}(\ct(T_k))|+|\h_{\infty}(\ct(T_1))| \le 2A. 
$$
\noindent
Since $O_{\tau(R)}(\shc)$ is the maximum of all the sums of the type $|\shc_1+...+\shc_{k-1}|$ the lemma follows.
\end{proof}
Next, we determine a set of independent generators for the homology group $\Hom$. Let $\FD$ be an ideal fundamental polygon for $G$ so that
the edges of $\FD$ are geodesics from $\lambda(G)$ (every such $\FD$ has $2(2g+n-1)$ edges). Choose an orientation on every such edge, so that the identification (given by $G$) respects the orientation. Denote by $\lambda^{*}_{\Gen}(S) \subset \lambda^{*}(S)$ the set of oriented geodesics, whose lifts to the universal cover $\Ha$ are the oriented geodesics from the boundary of $\FD$ (if we forget about the orientation, the corresponding set of  geodesic is called $\lambda_{\Gen}(S)$). The set  $\lambda^{*}_{\Gen}(S)$ contains $2g+n-1$ different elements. 
Note that the complement of the set $\lambda^{*}_{\Gen}(S)$ in $S$ is simply connected. This shows that the geodesics from $\lambda^{*}_{\Gen}(S)$ generate the group $\Hom$. On the other hand, the dimension of the group is $\Hom$ is $2g+n-1$. This shows that the set $\lambda^{*}_{\Gen}(S)$ is a set of independent generators. The following proposition is elementary.
\begin{proposition} Let $R \in \N\Tr(S)$. Assume that 

$$
\partial{R}=\sum_{i=1}^{i=m}k_i\gamma^{*}_{k_{i}},
$$
\noindent
where $\gamma^{*}_{k_{i}} \in \lambda^{*}_{\Gen}(S)$ and $k_i \in \Z$. Then $k_i=0$ for every $i$ that is $\partial{R}=0$ in $\Z\Gamma^{*}(S)$.
\end{proposition}
\begin{proof} We have already observed that the homology class of $\partial{R}$ is equal to zero in $\Hom$. Since $\lambda^{*}_{\Gen}(S)$ are independent generators the proposition follows.
\end{proof}

\subsection{The Correction lemma} Given a labelled collection of triangles $\Col$ our aim is to equip $\Col$ with the corresponding involution  
to produce an admissible pair. This of course not always possible. As we saw above, only if $\partial{\Col}=0$ in $\Gamma^{*}(S)$ one can do this. Moreover, we want to be able to construct this involution so that the centres of the triangles that are paired are as close as possible. Given an arbitrary $\Col$ we can add more labels to the set  $\Lab_{\Col}$ and expand the domain and the range of the map $\lab_{\Col}$ so that we are able to construct the corresponding involution. The purpose of this subsection is to prove the lemma which tells us how many labels we need to add.
\vskip .1cm 
We first prove  a few propositions. All the distances mentioned below are considered to in the hyperbolic metric (unless stated otherwise).
\begin{proposition} There exists a constant $d_1(S)>0$ which depends only on $S$ so that the following holds. Let $c \in \Cus$ and let $G=G_c$ $c=\infty$ be normalised. Let $\gamma_0 \in \Gamma(G)$  with the endpoints $c_1,c_2 \in \Cus(G)$, $c_1,c_2 \in \R$, $c_1<c_2$, and let $z_0 \in \gamma$. Assume that for $\infty \in \Cus(G)$ the geodesic $\gamma_0$ intersects at least $100N(S)$ geodesics from $\lambda(S)$ that end at $\infty$. 
Also, assume that $z_0 \in \Hor_{\infty}(t_2(S))$. Let $\gamma'$ be the geodesic that is orthogonal to $\gamma_0$ and that contains $z_0$. Let  $c_1<x<c_2$ and $y \in \R$ be the endpoints of $\gamma'$. We have:
\begin{itemize}
\item  There exists $c_3 \in \Cus(G)$  so that the geodesic that connects $c_3$ and $\infty$ belongs to $\lambda(G)$ and so that
$c_3$ is the closest point to $x$ subject to the condition $c_1+2 <c_3 <c_2-2$.
Let $\gamma_i$, $i=1,2$, be the geodesic that connects $c_i$ with $c_3$ and let $T$ be the triangle bounded by $\gamma_0$, $\gamma_1$ and $\gamma_2$. 
Then the distance between the centre $\ct(T)$ and the point $z_0$ is less than $d_1(S)$.
\item Assume $y>c_2$ (the analogous statement holds for $y<c_1$). There exists $c_3 \in \Cus(G)$  so that the geodesic that connects $c_3$ and $\infty$ belongs to $\lambda(G)$ and so that $c_3$ is the closest point to $y$ subject to the condition $c_2+2 <c_3$. Let $\gamma_i$, $i=1,2$, be the geodesic that connects $c_i$ with $c_3$ and let $T$ be the triangle bounded by $\gamma_0$, $\gamma_1$ and $\gamma_2$. 
Then the distance between the centre $\ct(T)$, and the point $z_0$ is less than $d_1(S)$.
\end{itemize}
\end{proposition}

\begin{proof} The existence of such $c_3$ in both cases follows from the fact that $|c_1-c_2| \ge 100$.
We prove the first statement. Assume that $|c_3-c_2| \le |c_3-c_1|$ 
(the other case is done in the same way). Clearly 
\begin{equation}\label{dis-0}
|x-c_3|<3.
\end{equation}
\noindent
Since  $z \in \Hor_{\infty}(t_2(S))$ it is elementary to see that 
\begin{equation}\label{dis-1}
c_2-x>{{e^{t_2(S)}}\over{2}}. 
\end{equation}

\begin{remark} Let $z_0=p+iq$.  If $c_2-x<2$ then $c_2-x<x-c_1$. This implies $x < p <c_2$. 
Then one shows that $\max\{(c_2-p),(p-x) \} \ge {{q}\over{2}} \ge {{e^{t_2(S)}}\over{2}}$.
\end{remark}

We prove the existence of the constant $d_1(S)$ by contradiction. Assume that there is no such $d_1(S)$. 
Then there exists the corresponding sequence of geodesic $\gamma_0(m)$ and points $z_0(m)$, $m \in \N$, so that $\dis(\ct(T(m)),z_0(m)) \to \infty$ when $m \to \infty$ (here $T(m)$ is the corresponding sequence of triangles). Let 
$$
f_m(z)={{(z-c_2(m))}\over{|c_3(m)-c_2(m)|}},
$$
\noindent
where $c_1(m)$, $c_2(m)$, $c_3(m)$ and $x(m)$, are the corresponding sequences of points. Since
$|c_3(m)-c_2(m)| \le |c_3(m)-c_1(m)|$, we have that the sequence of triangles $f_m(T(m))$ after passing onto a subsequence if necessary, converges to a non-degenerate ideal triangle $T(\infty)$ in $\Ha$.  From the choice of $f_m$ we have that $f_{m}(c_2(m))=0$ for every $m \in \N$. 
\vskip .1cm
If $|c_3(m)-c_2(m)| \to \infty$ then from (\ref{dis-0}) we conclude that the sequence $f_m(x(m))$ converges to $x(\infty)<0$.
If $|c_3(m)-c_2(m)|$ remains bounded then from (\ref{dis-1}), and from the fact that $|c_2-c_3| \ge 2$ (which means that this sequence is bounded from below as well),  we conclude that the sequence $f_m(x(m))$ converges to $x(\infty)<0$. Also, $f_m(c_1(m)) \to c_1(\infty)$ and $c_1(\infty)<x(\infty)$.
This shows that the sequence of geodesic $f_m(\gamma_0(m))$ converges to a proper geodesic in $\Ha$. Moreover, $f_m(z_0(m))$ converges to a point $z_0(\infty) \in \Ha$.  Since $\dis(\ct(T(m)),z_0(m)) \to \dis(\ct(T(\infty)),z_0(\infty)) $, $m \to \infty$, we obtain a contradiction. 
The second statement is proved in a similar way.
\end{proof}

\begin{proposition} There exist constants $d_2(S),d_3(S),\theta(S)>0$ so that the following holds. Let $\gamma_0 \in \Gamma(G)$ and suppose $\iota(\gamma_0,\tau(G))>0$. Let $z_0 \in \gamma_0 \cap \Thick_G(t_1(S))$. Let $\gamma_1 \in \lambda(G)$ be the geodesic, so that letting $w=\gamma_0 \cap \gamma_1$ the distance $\dis(z_0,w)$ is the smallest among all the distances between $z_0$ and the intersection points between $\gamma_0$ and  geodesics from
$\lambda(G)$ (if there are two such closest points $w$ we choose either one of them). Assume in addition that if for some $c \in \Cus(G)$ the geodesic
$\gamma_0$ intersects at least $100N(S)$ geodesics from $\lambda(G)$ that all have $c$ as their endpoint, then $z_0$ does not belong to the horoball
$\Hor_c(t_2(S)-1)$.  Then,
\begin{enumerate}
\item We have $\dis(z_0,w) < d_2(S)$ and the smaller angle between $\gamma_0$ and $\gamma_1$ is greater than $\theta(S)$.
\item Let $c_1 \in \Cus(G)$ be an endpoint of $\gamma_1$. Let $\gamma'_0, \gamma''_0 \in \Gamma(G)$ be the geodesics that connect 
$c_1$ with the two endpoints of $\gamma_0$ respectively. Let $T \in \Tr(G)$ be the triangle bounded by $\gamma_0$ $\gamma'_0$ and $\gamma''_0$.
Then $\iota(\gamma'_0,\tau(G))+\iota(\gamma''_0,\tau(G)) \le \iota(\gamma_0,\tau(G))-1$. Moreover, there are points 
$z'_0 \in \gamma'_0 \cap \Thick_G(t_1(S))$ and $z''_0 \in \gamma''_0 \cap \Thick_G(t_1(S))$ so that
the distance between the centre of the triangle $T$ and any of the points $z_0,z'_0,z''_0$ is less than $d_3(S)$.
\end{enumerate}
\end{proposition}

\begin{proof} We first prove (1). Let $\gamma_0(m)$, $m \in \N$, be  a sequence of  geodesics, and $z_0(m) \in \gamma_0(m)$ a sequence of points,
so that $\gamma_0(m)$ and $z_0(m)$ share the above stated properties of $\gamma_0$ and $z_0$. Let $\gamma_1(m)\in \tau(G)$ and $w(m) \in \gamma_1(m)$
be the corresponding geodesics and points. Assume that either $\dis(z_0(m),w(m)) \to \infty$ or that the angle between $\gamma_0(m)$ and $\gamma_1(m)$
tends to zero, when $m\to \infty$. Since $z_0(m) \in \Thick_G(t_1(S))$ we can choose $f_m \in G$ so  that after passing to a subsequence if necessary, we have $f_m(z_0(m)) \to z_0(\infty) \in \Ha$.
Therefore, we may assume that $z_0(m) \to z_0(\infty) \in \Thick_G(t_1(S))$. Let $\gamma_0(\infty)$ be a geodesic in $\Ha$ where  $\gamma_0(m) \to \gamma_0(\infty)$. If $\dis(z_0(m),w(m)) \to \infty$ then the geodesic
$\gamma_0(\infty)$ does not intersect any geodesics from $\lambda(G)$. We conclude that $\gamma_0(\infty) \in \lambda(G)$. If the sequence
$\{\dis(z_0(m),w(m))\}$ is bounded, then the sequence of geodesics $\gamma_1(m)$ tends to a geodesic $\gamma_1(\infty) \in \lambda(G)$. If the angle between $\gamma_0(m)$ and $\gamma_1(m)$ tends to zero, when $m\to \infty$ then $\gamma_0(\infty)=\gamma_1(\infty)$ and we again conclude that  $\gamma_0(\infty) \in \lambda(G)$. 
\vskip .1cm
Let $c_1(\infty),c_2(\infty) \in \Cus(G)$ denote the endpoints of $\gamma_0(\infty) \in \lambda(G)$. Since $z_0(\infty) \in \Thick_G(t_1(S))$ we conclude from Proposition 3.1 that $z_0(\infty)$ belongs to both horoballs 
$\Hor_{c_{1}(\infty)}(t_2(S))$ and $\Hor_{c_{2}(\infty)}(t_2(S))$. Therefore, for $m_0$ large enough, and for every $m \ge m_0$  we have
$$
z_0(m) \in \Hor_{c_{1}(\infty)}(t_2(S)-{{1}\over{2}}),
$$
\begin{equation}\label{horo-1}
\end{equation}
$$
z_0(m) \in \Hor_{c_{2}(\infty)}(t_2(S)-{{1}\over{2}}).
$$
\noindent
Note that $\Hor_{c_{1}(\infty)}(t_2(S)-{{1}\over{2}})$ is contained in the horoball $\Hor_{c_{1}(\infty)}(t_2(S)-1)$ (the same is true for $c_2(\infty)$).
On the other hand, since $\gamma_0(m) \to \gamma_0(\infty)$ and $\gamma_0(m) \ne \gamma_0(\infty)$ (since by the assumption of the lemma we have that $\gamma_0(m)$ does not belong to $\lambda(G)$), we conclude that the number of geodesics from $\lambda(G)$ that have either $c_1(\infty)$ or $c_2(\infty)$ as their endpoints, and that are intersected by $\gamma_0(m)$ tends to $\infty$. Let $m_1>m_0$ be large enough, so that $\gamma_0(m_1)$ intersects
at least $100N(S)+1$ geodesics from $\lambda(G)$ that all have $c_1(\infty)$ as their endpoint (the case when these geodesics end at $c_2(\infty)$ is similar). From this and from (\ref{horo-1}) we obtain a contradiction with the assumption  that $z_0(m_1)$ does not belong to a $t_2(S)-1$-horoball
of any cusp $c\in \Cus(G)$ when $\gamma_0(m_1)$ intersects at least $100N(S)$ geodesics from $\lambda(G)$ that end at $c$.
\vskip .1cm
Next we prove (2). Every geodesic from $\lambda(G)$ that intersects $\gamma_0$ either intersects exactly one of the  geodesics $\gamma'_0$ or.  $\gamma''_0$ or this geodesic ends at $c_1$. Since $\gamma_1 \in \lambda(G)$ has $c_1$ as its endpoint, we see that
$\iota(\gamma'_0,\tau(G))+\iota(\gamma''_0,\tau(G))\le \iota(\gamma_0,\tau(G))-1$. 
\vskip .1cm
We prove the last part of (2) by contradiction. The argument is very similar as in the proof of (1). Let $\gamma_0(m) \in \Gamma(G)$, $m \in \N$, 
be a sequence of geodesics, and $z_0(m) \in \gamma_0(m)\cap \Thick_G(t_1(S))$, $m \in \N$, a sequence of points. Let 
$\gamma'_0(m)$ and   $\gamma''_0(m)$ be the corresponding sequences of geodesics, and let $T(m)$ be the corresponding sequence of triangles, where
$\ct(T(m))$ denotes the centres of $T(m)$. Assume that at least one of the distances $\dis(\ct(T(m)),z_0(m))$ $\dis(\ct(T(m)),\Thick_G(t_1(S))\cap\gamma'_0(m))$ or $\dis(\ct(T(m)),\Thick_G(t_1(S))\cap\gamma''_0(m))$ tends to $\infty$ when $m \to \infty$.
Same as above, we can assume that $z_0(m) \to z_0(\infty) \in \Ha$ and $\gamma_0(m) \to \gamma_0(\infty)$. From the conclusion (1) of this lemma that was proved above, 
we conclude that $T(m)$ converges to an ideal triangle $T(\infty)$. 
In particular,  $\gamma'_0(m) \to \gamma'_0(\infty)$ and 
$\gamma''_0(m) \to \gamma''_0(\infty)$. Also, $\ct(T(m)) \to \ct(T(\infty))$ where $\ct(T(\infty))$ is the centre of $T(\infty)$.
Since $\dis(\ct(T(m)),z_0(m))\to \dis(\ct(T(\infty)),z_0(\infty))$ we see that $\dis(\ct(T(m)),z_0(m))$ does not tend to $\infty$.
This means that one of the distances $\dis(\ct(T(m)),\Thick_G(t_1(S))\cap\gamma'_0(m))$ or $\dis(\ct(T(m)),\Thick_G(t_1(S))\cap\gamma''_0(m))$ tends to $\infty$ when $m \to \infty$. This implies that at least one of the geodesics $\gamma'_0(\infty)$ or $\gamma''_0(\infty)$ is contained in
$\Ha \setminus \Thick_G(t_1(S))$. Since $t_1(S)>0$ we have that the set $\Ha \setminus \Thick_G(t_1(S))$ is a union of disjoint horoballs (the base points of these horoballs are in $\Cus(G)$) in $\Ha$. Since every geodesic in $\Ha$ is a connected set, we conclude that at least one of the geodesics $\gamma'_0(\infty)$ or $\gamma''_0(\infty)$ is contained in a horoball in $\Ha$. This is a contradiction.
\end{proof}

We will use the following definition:
\begin{definition}
Let $\gamma \in \lambda(G)$ be a geodesic. Let $T_1,T_2 \in \tau(G)$ be the two triangles that are adjacent along $\gamma$. 
By $\midp(\gamma)$ we denote the intersection point between $\gamma$ and the geodesic segment that connects the centres of the triangles $T_1$ and $T_2$.
\end{definition}

We now prove the Correction Lemma.

\begin{lemma} Let $\gamma_0 \in \Gamma(G)$ be a geodesic, and let $\gamma_0^{*} \in \Gamma^{*}(G)$ be the geodesic $\gamma_0$ with a chosen orientation. 
Let $z_0 \in \gamma_0$ be a point. Then there exists an ideal polygon $\Pol$ with a triangulation $\tau(\Pol)$ that has the following properties:

\begin{enumerate}
\item The geodesic $\gamma_0$ is an edge of $\Pol$ and all other edges of $\Pol$ are geodesics from $\lambda_{\Gen}(G)$ where
$\lambda_{\Gen}(G)$ is the lift of $\lambda_{\Gen}(S)$ to $\Ha$.
\item The polygon $\Pol$ is to the right of $\gamma_0^{*}$.
\item There exists a constant $D>0$ which depends only on $S$ so that the centres of any two adjacent triangles from $\Pol$ are within the $D$ hyperbolic distance.  If $T \in \tau(\Pol)$ is the triangle adjacent to $\gamma_0$ then the distance between the centre of $T$ and the point $z_0$ is less than ${{D}\over{2}}$.  If  $T \in \tau(\Pol)$ is a triangle adjacent to another edge $\gamma$ of $\Pol$ then the  hyperbolic distance between the centre of $T$ and the point $\midp (\gamma) \in \gamma$ is less than ${{D}\over{2}}$. 
\item There exist constants $C,K>0$ which depend only on $S$ so that  $|\tau(\Pol)| \le C (\cl(\gamma_0,z_0))+K$. 
Here $|\tau(\Pol)|$ stands for the total number of triangles in $\Pol$.
\end{enumerate}

\end{lemma}

\begin{proof} First consider the case  when $\cl(\gamma_0,z_0)=0$. We allow any orientation on $\gamma^{*}_0$. 
Then we have $\gamma_0 \in \lambda(G)$. Moreover, let $c_1,c_2 \in \Cus(G)$ be the two endpoint of $\gamma_0$. In this case we have that $\h_{c_{1}}(z_0),\h_{c_{2}}(z_0) \le t_1(S)$.
We build the corresponding polygon as follows. Let $T$ be the triangle from $\tau(G)$ that is to the right of $\gamma^{*}_0$. If an edge of $T$ belongs to the set $\lambda_{\Gen}(G)$ we do not add any more triangles along this edge. If an edge of $T$ does not belong to $\lambda_{\Gen}(G)$ we glue the corresponding triangle from $\tau(G)$ to the right of this edge. We repeat this process until all the edges of the polygon $\Pol$ we obtained are all from $\lambda_{\Gen}(G)$ (except the edge $\gamma_0$). This process has to end because the geodesics from $\lambda_{\Gen}(G)$ are the edges of an ideal fundamental polygon for $G$. This way we have also constructed the corresponding triangulation $\tau(\Pol)$ of the polygon $\Pol$.
\vskip .1cm
We repeat this construction for every $\gamma_0 \in \lambda(G)$. If  $f \in G$ and $f(\gamma_0)=\gamma'_0$ then the corresponding polygon for 
$\gamma'_0$ is the image of the corresponding polygon for $\gamma_0$ under the map $f$. So, we need to consider only $|\lambda(S)|$ different geodesics
from $\lambda(G)$. Since $\gamma_0 \cap \Thick_G(t_1(S))$ is a compact set in $\Ha$ by the compactness argument we find that there exist constants $D_1>0$ and $K>0$ so that for every such $\gamma_0\in \lambda(G)$ the corresponding polygon $\Pol$ satisfies that $|\tau(\Pol)| \le K$.
Also, the distances between the centres of adjacent triangles in $\Pol$ are  bounded by $D_1$. The distances between the centres of boundary triangles, and the corresponding points on the edges of $\Pol$ are bounded by ${{D_1}\over{2}}$.
\vskip .1cm
We now prove the general case by induction. We first determine the constants $D$, $C$ from the statement of the lemma (the constant $K$ has already been defined). Set
$$
D=2\big(D_1+2d_3(S)+(64+\log32)+2d_1(S)+2\log(4N(S)+1)\big)
$$
\noindent
Let
$$
C_1={{K+1}\over{ {{[e^{t_{1}(S)}]}\over{2}} -2-N(S)}}.
$$
\noindent
Recall that we chose $t_1(S)$ large enough so that the denominator is a positive constant. Set
$$
C=(2K+1)+10K+C_1.
$$
\noindent
The numbers that appear in the definition of $C$ and $D$, will appear in the arguments below. 
\vskip .1cm 
The proof is by induction on $m=\cl(\gamma_0,z_0)$. The case $0=m=\cl(\gamma_0,z_0)$ has been done above. Fix $m \in \N$ and assume that the statement of the lemma is true for every non-negative integer that is less than $m$. We now prove that the statement for $m$. There are four cases to consider.
\vskip .1cm
\noindent
${\bf {Case}}$ $1.$ In this case we assume that $z_0 \in \Thick_G(t_1(S))$. Moreover, we assume that if for some $c \in \Cus(G)$
$\gamma_0$ intersects at least $100N(S)$ geodesics from $\lambda(G)$ that all end at $c$ then
the point $z_0$ does not belong to $\Hor_c(t_2(S))$ (then $z_0$ does not belong to $\Hor_c(t_2(S)-1)$ either, so we may apply Proposition 3.4). 
We allow any orientation $\gamma^{*}_0$. We have
$\cl(\gamma_0,z_0)=\iota(\gamma_0,\tau(G))=m$. 
\vskip .1cm
We apply  Proposition 3.4. That is, there are geodesics $\gamma'_0, \gamma''_0 \in \Gamma(G)$ so that there exists a triangle $T \in \Tr(G)$ that is  bounded by $\gamma_0$, $\gamma'_0$ and $\gamma''_0$. Moreover, there are points  $z'_0 \in \gamma'_0 \cap \Thick_G(t_1(S))$ and $z''_0 \in \gamma''_0 \cap \Thick_G(t_1(S))$ so that the distance between the centre of the triangle $T$ and any of the points $z_0,z'_0,z''_0$ is less than $d_3(S)$.
Also, we have that $\iota(\gamma'_0,\tau(G))+\iota(\gamma''_0,\tau(G)) \le \iota(\gamma_0,\tau(G))-1=m-1$. We choose the orientations of ${\gamma_0'}^{*}$ 
and  ${\gamma_0''}^{*}$ so that $\gamma_0$ is to the left of both of them. 
\vskip .1cm
Since  $z'_0,z''_0 \in \Thick_G(t_1(S))$ we have $\cl(\gamma'_0,z_0)=\iota(\gamma'_0,\tau(G))$ and $\cl(\gamma''_0,z_0)=\iota(\gamma''_0,\tau(G))$. By the induction hypothesis, the statement is true for both pairs $({\gamma_0'}^{*},z'_0)$ and  $({\gamma_0''}^{*},z''_0)$. We glue the corresponding triangulated polygons to the right of ${\gamma_0'}^{*}$ and ${\gamma_0''}^{*}$ respectively. Together with the triangle $T$ this gives the needed polygon $\Pol$. We have (using $C>K+1$)
$$
|\tau(\Pol)| \le 1+ C\cl(\gamma'_0,z'_0)+C\cl(\gamma''_0,z''_0)+2K=(2K+1)+C(\cl(\gamma_0,z_0)-1) \le C\cl(\gamma_0,z_0)+K.
$$
\noindent
By the induction hypothesis, the distances between the centres of adjacent triangles in the two polygons we have glued along
$\gamma'_0$ and $\gamma''_0$ are less than $D$. Also, by the induction hypothesis, the distances between the points $z'_0$ and $z''_0$ and the
centres of the corresponding triangles in those glued polygons, that contain the points $z'_0$ and $z''_0$ in the their boundaries respectively, are less than ${{D}\over{2}}$. Since the distance between the centre of the triangle $T$ and any of the points $z_0,z'_0,z''_0$ is less than $d_3(S)<{{D}\over{2}}$ we have that the distance between the centre of $T$ and either one of the two adjacent triangles to $T$ in $\Pol$ is less than $D$. This shows that all the distances between the corresponding points in $\Pol$ are within $D$ or ${{D}\over{2}}$ as required.
This settles the first case.
\vskip .1cm
\noindent
${\bf {Case}}$ $2.$  Let $c_1,c_2 \in \Cus(G)$ be the endpoints of $\gamma_0$. In this case we assume that 
$\max\{\h_{c_{1}}(z_0),\h_{c_{2}}(z_0) \} > t_{1}(S)$. For the sake of the argument assume $\h_{c_{1}}(z_0) > t_{1}(S)$. We allow any orientation on $\gamma_0$.  Normalise $G$ that is
$G=G_{c_{1}}$ and $c_1=\infty$. Also, assume that the orientation of $\gamma^{*}_0$ is such that the pair $(c_2,\infty)$ has the positive orientation
on $\gamma^{*}_0$ (the other case is done analogously). Let $k=\iota(\gamma_0,z_0)$. Then $[e^{\h_{\infty}(z_0)}]=m-k$.  
Let $c_3 \in \Cus(G)$ be the point on $\R$ defined as follows. We require that the geodesic that connects $c_3$ and $\infty$ belongs to  $\lambda(G)$. Then,  $c_3$ is the smallest such point subject to the condition 
$$
c_2+{{e^{\h_{\infty}(z_0)}}\over{4N(S)}} \le c_3. 
$$
\noindent
Let $\gamma_1 \in \Gamma(G)$ be the geodesic that connects $c_2$ and $c_3$ (we fix the orientation of $\gamma^{*}_1$ so that $\infty$ is to the left of 
$\gamma^{*}_1$). Since $G$ is normalised, we have 
$$
{{e^{\h_{\infty}(z_0)}}\over{4N(S)}}<c_3-c_2<{{e^{\h_{\infty}(z_0)}}\over{4N(S)}}+1,
$$
\noindent
and this yields
$$
\iota(\gamma_1,\tau(G)) \le k+\left( {{e^{\h_{\infty}(z_0)}}\over{4N(S)}}+1 \right) N_{\infty}(S) \le k+  {{[e^{\h_{\infty}(z_0)}]}\over{4}}+1+N(S).
$$
\noindent
Let $z_1=\z_{\max}(\gamma_1,\infty)$. Since $(c_3-c_2) \ge {{1}\over{4N(S)}}$ we have that $z_1 \in \Hor_{\infty}(-\log(8N(S)))$ and therefore
by the choice of $t_1(S)$ in Proposition 3.1 the point $z_1$ does not belong to any $t_1(S)$-horoball (except possibly the one at $\infty$ which is not an endpoint of $\gamma_1$). This implies that 
$$
\cl(\gamma_1,z_1)=\iota(\gamma_1,\tau(G)). 
$$
\vskip .1cm
Let $\gamma_2$ be the geodesic that connects $c_3$ and $\infty$ (we fix the orientation of $\gamma^{*}_2$ so that $c_2$ is to the left of $\gamma^{*}_2$). By the choice of $c_3$ we have that $\gamma_2 \in \lambda(G)$. Let $z_2 \in \gamma_2$
be the point whose imaginary part is ${{e^{\h_{\infty}(z_0)}}\over{4}}$. Since
$$
{{e^{\h_{\infty}(z_0)}}\over{4}}>{{e^{t_{1}(S)}}\over{4}}>1,
$$
\noindent
we have that $z_2$ does not belong to any $t_1(S)$-horoball except possibly the one at
$\infty$.  We have 
$$
\cl(\gamma_2,z_2)\le [e^{\h_{\infty}(z_2)}] \le {{e^{\h_{\infty}(z_0)}}\over{4}} \le {{[e^{\h_{\infty}(z_0)}]}\over{4}}+1.
$$
\noindent
By the induction hypothesis, we have that the statement of the lemma is true for both pairs $(\gamma^{*}_1,z_1)$ and $(\gamma^{*}_2,z_2)$. We construct the polygon $\Pol$ by gluing the corresponding polygon for $\gamma_1$ to the right of $\gamma^{*}_1$ and the corresponding polygon for $\gamma_2$ to the right of $\gamma^{*}_2$. We also add the triangle $T \in \Tr(G)$ bounded by $\gamma_0$ $\gamma_1$ and $\gamma_2$ to $\Pol$. Combining this with the above estimates (and the definition of $C$) we have 
$$
|\Pol| \le 1+ C(k+{{[e^{\h_{\infty}(z_0)}]}\over{4}}+1+N(S))+C({{[e^{\h_{\infty}(z_0)}]}\over{4}}+1)+2K \le
$$

$$
\le 1+C(k+{{[e^{\h_{\infty}(z_0)}]}\over{2}}+2+N(S))+2K \le
$$
$$
\le C(k+[e^{\h_{\infty}(z_0)}])+K +\big(K+1- C({{[e^{t_{1}(S)}]}\over{2}}-2-N(S))\big) \le C\cl(\gamma_0,z_0)+K
$$
\noindent
Note that the distance between the centre of $T$ and any of the points $z_0,z_1,z_2$  is less than $\log(4N(S)+1)$.
One now shows that the distances between the corresponding points in $\Pol$ are within $D$ or ${{D}\over{2}}$ in much the same way as in the Case 1.
\vskip .1cm
\noindent
${\bf {Case}}$ $3.$ The only case left to consider is when  $z_0 \in \Thick_G(t_1(S))$ but when there exists $c \in \Cus(G)$ so that $z_0$  belongs 
to $\Hor_c(t_2(S))$ and where $\gamma_0$ intersects at least $100N(S)$ geodesics from $\lambda(G)$ that end at $c$. In this case $\gamma_0$ can not have $c$ as its endpoint. Moreover, $\cl(\gamma_0,z_0)=\iota(\gamma_0,\tau(G))=m$.
Assume that $G$ is normalised, $G=G_c$. Let $c_1,c_2 \in \Cus(G)$, $c_1<c_2$, be the endpoints of $\gamma_0$.  We first consider the case when the orientation of $\gamma^{*}_0$ is such that the cusp $c$ is to the left of $\gamma_0$.
\vskip .1cm
Let $\gamma'_0$ be the geodesic that is orthogonal to $\gamma_0$ and that contains $z_0$. Let $c_1<x<c_2$ be the corresponding endpoint of $\gamma'_0$. 
Since $\gamma_0$ intersects at least $100N(S)$ geodesics from $\lambda(G)$ that end at $\infty$ we can choose a point  
$c_3 \in \Cus(G)$ so that the geodesic that connects $c_3$ and $\infty$ belongs to $\lambda(G)$ and so that $c_3$ is the closest such point 
to the point $x$ subject to the condition $c_1+2<c_3<c_2-2$ (see Proposition 3.3). Let $\gamma_i \in \Gamma(G)$
$i=1,2$, be the geodesic that connects $c_i$ with $c_3$. Choose the orientation of $\gamma^{*}_i$ so that $\infty$ is to the left of  $\gamma^{*}_i$.
Let $T \in \Tr(G)$ be the triangle bounded by $\gamma_0$, $\gamma_1$ and $\gamma_2$.
Let $z_i$, $i=1,2$, be the orthogonal projections of the centre $\ct(T)$ to the geodesic $\gamma_i$. Since $|c_3-c_i|>2$, $i=1,2$, it is easily seen 
that $\h_{\infty}(z_i)>0$. We conclude that the point $z_i$ does not belong to the $t_1(S)$-horoballs at $c_i$ or $c_3$. Therefore, 

$$
\cl(\gamma_i,z_i)=\iota(\gamma_i,\tau(G)), i=1,2.
$$
\vskip .1cm  
From Proposition 3.3 we have that $\iota(\gamma_1,\tau(G))+\iota(\gamma_2,\tau(G))=\iota(\gamma_0,\tau(G))-1=\cl(\gamma_0,z_0)-1$. By the induction hypothesis, the statement of the lemma is true for the pair $(\gamma_i,z_i)$. We glue the corresponding polygons to the right of $\gamma^{*}_i$. Let $T \in \Tr(G)$ be the triangle bounded by $\gamma_0$, $\gamma_1$ and $\gamma_2$. We add this triangle to obtain the required polygon $\Pol$. We have 
$$
|\tau(\Pol)| \le 1+C\cl(\gamma_1,z_1)+K+C\cl(\gamma_2,z_2)+K \le C\cl(\gamma_0,z_0)+(2K+1)-C \le C\cl(\gamma_0,z_0)+K.
$$
\noindent
By Proposition 3.3 the distance between the centre of $T$ and the point $z_0$ is less than $d_1(S)$. By the definition of
$z_1$ and $z_2$  the distance between the centre of $T$ and the point $z_i$ is less than $1$ (the distance between the centre of an ideal triangle,
and its orthogonal projections to one of its sides is less than $1$).
In much the same way as before one shows that the distances between the corresponding points in $\Pol$ are within the required bounds.
\vskip .1cm

\begin{figure}
  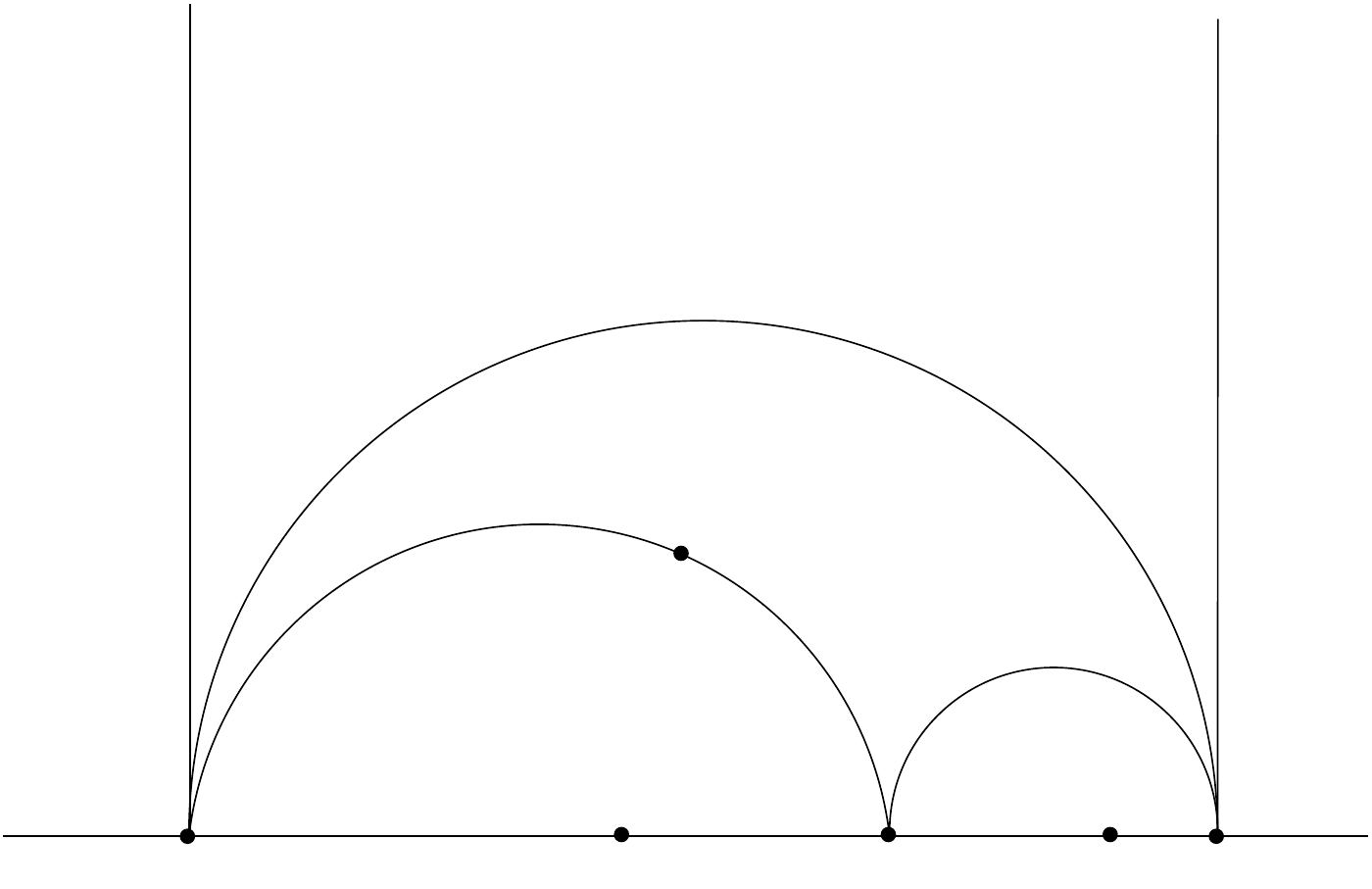
\label{three}
\caption{}
\end{figure}

\noindent
${\bf {Case}}$ $4.$ The assumptions are the same as in Case 3. except that the orientation of $\gamma^{*}_0$ is such that
$\infty$ is to the right of $\gamma^{*}_0$.  Let $y$  denote the second endpoint of $\gamma'_0$ and assume that $y>c_2$ (the case $y<c_1$ is treated similarly).  Recall $\cl(\gamma_0,z_0)=\iota(\gamma_0,\tau(G))=m$. Let $m_3$ denote the number of geodesics from $\lambda(G)$ that all end at $\infty$ and that are intersected by $\gamma_0$. Let $m_1$ denote the number number of geodesics from $\lambda(G)$ that are intersected by $\gamma_0$ and that are to the left of those geodesics we counted for $m_3$. Let $m_2$ denote the  number of geodesics from $\lambda(G)$ that are intersected by $\gamma_0$ and that are to the right of those counted for $m_3$. Then $m_1+m_2+m_3=m$. 
\vskip .1cm
Choose a point   $c_3 \in \Cus(G)$ (see Figure \ref{three}) so that the geodesic that connects $c_3$ and $\infty$ belongs to $\lambda(G)$ and so that $c_3$ is the closest such point  to the point $y$ subject to the condition $c_2+2<c_3$ (see Proposition 3.3).  Let $k$ be the integer so that
\begin{equation}\label{k-correction}
\log_{2}[c_2-c_1]-\log_{2}[c_3-c_2] < k \le \log_{2}[c_2-c_1]-\log_{2}[c_3-c_2]+1,
\end{equation}
\noindent
where $[c_2-c_1]$ and $[c_3-c_2]$ denote the corresponding integer parts. Note that $m_3 \ge [c_2-c_1]N_{\infty}(S)-1$. This gives
$$
\cl(\gamma_0,z_0) \ge  [c_2-c_1]N_{\infty}(S).
$$
\vskip .1cm
First consider the case  $k>3$. Let $c_{i+3}$, $i=1,...,k-3$, be given by
$$
c_{i+3}=c_3+\sum_{j=i}^{j=i}2^{j}[c_3-c_2].
$$
\noindent
Let $\gamma_{i}\in \Gamma(G)$, $i=0,...,k-2$, denote the geodesic that connects $c_{1}$ and $c_{i+2}$. By $\wt{\gamma_i}$ denote the geodesic that connects
$c_{i+1}$ and $c_{i+2}$, $i=0,...,k-2$ (with this definition, the geodesics $\gamma_0$ and $\wt{\gamma}_0$ agree). 
Let $\wh{\gamma}_1$ be the geodesic that connects $c_1$ and $\infty$, and $\wh{\gamma}_k$ the one that connects $c_k$ and $\infty$. Let $T_{i+1}$ be the triangle bounded by $\gamma_i$, $\gamma_{i+1}$ and $\wt{\gamma}_{i+1}$. By $\wh{T}$ denote the triangle bounded by  $\wh{\gamma}_1$, $\wh{\gamma}_k$
and $\gamma_{k-2}$. 
\vskip .1cm
Note that by the construction, the distance between the centres of adjacent triangles $T_{i}$ and $T_{i+1}$ is less than $\log2<1$. It is directly seen
that the distance between $\ct(T_{k-2})$ and $\ct(\wh{T})$ is less than $\log32+64$. By Proposition 3.3
the distance between $z_0$ and the centre of the triangle $T_1$ is less than $d_1(S)$.
\vskip .1cm
Let $\wt{z}_i \in \wt{\gamma}_i$, $i=1,...,k-2$, be the points that are the orthogonal projections of the centres of  triangles $T_i$ to 
$\wt{\gamma}_i$. Since $(c_3-c_2)>2$ we have that the centre of $T_i$ is high enough with respect to $\infty$ so that $\h_{\infty}(\wt{z}_i)>0$.
Therefore, $\cl(\wt{\gamma}_i,\wt{z}_i)=\iota(\wt{\gamma}_i,\tau(G))$. Since for both endpoints of $\gamma_i$, $i \ge 2$, we have that there are geodesics
from $\lambda(G)$ that connect those points to $\infty$ we have that 

$$
\cl(\wt{\gamma}_i,\wt{z}_i)\le 2^{i}[c_3-c_2]N_{\infty}(S)-1+m_1+m_2, 
$$
\noindent
for $i \ge 2$. For the pair 
$(\wt{\gamma}_1,\wt{z}_1)$ we have 
$$
\cl(\wt{\gamma}_1,\wt{z}_1) \le m_2+2[c_3-c_2]N_{\infty}(S).
$$
\vskip .1cm
On the other hand,
let $\wh{z}_1 \in \wh{\gamma}_1$ and $\wh{z}_{k} \in \wh{\gamma}_{k}$ be the points whose imaginary parts are equal to ${{[c_2-c_1]}\over{8}}$.
Then 
$$
\cl(\wh{\gamma}_1,\wh{z}_1)\le m_1+{{[c_2-c_1]}\over{8}},
$$
\noindent
and
$$
\cl(\wh{\gamma}_{k},\wh{z}_{k})\le {{[c_2-c_1]}\over{8}}.
$$
\noindent
Also, it is easily seen that the distances between the centre of the triangle $\wh{T}$ and the points $\wh{z}_1$ and $\wh{z}_{k}$ are less than $\log64<8$.
\vskip .1cm 
We now apply the induction hypothesis on the pairs $(\wt{\gamma}_i,\wt{z}_i)$, $i=1,...,k-2$, $(\wh{\gamma}_1,\wh{z}_1)$, $(\wh{\gamma}_{k-2},\wh{z}_{k-2})$. We put the corresponding orientations on these geodesics, so that $\gamma_0$ is to the left
of them. After gluing the corresponding polygons to the right
of the corresponding geodesics, and adding the triangles $T_i$ and $\wh{T}$ we obtain the polygon $\Pol$. We have 

$$
|\tau(\Pol)| \le (k-2)+1+C\sum_{i=1}^{i=k-2}2^{i}[c_3-c_2]N_{\infty}(S)+Cm_2+C(m_1+{{[c_2-c_1]}\over{8}})+C{{[c_2-c_1]}\over{8}}+kK\le
$$

$$
\le (k+1)K-1+C \big(m_1+m_2+{{[c_2-c_1]}\over{4}}+2^{k-1}[c_3-c_2]N_{\infty}(S)\big) \le 
$$
$$
\le (k+1)K-1+C\big(m_1+m_2+{{[c_2-c_1]}\over{4}}+
{{[c_2-c_1]}\over{2}}N_{\infty}(S)+{{1}\over{2}}\big) \le
$$
$$
\le C(m_1+m_2+[c_2-c_1]N_{\infty}(S)-1)+K+\left({{3}\over{2}}C-C{{[c_2-c_1]}\over{4}}+kK \right) \le
C\cl(\gamma_0,z_0)+K.
$$
\noindent
The last part of the above estimate follows from (\ref{k-correction}), the fact that $c_2-c_1>100$ and the property
$C>10K$.
\vskip .1cm
Let $0 \le k \le 3$ in the above construction. Let $\wh{\gamma_1}$ and $\wh{\gamma}_2$ be the geodesics that connect
$c_1$ and $c_2$ with $\infty$, respectively. Let $\wh{T}$ be the triangle bounded by $\gamma_0$, $\wh{\gamma_1}$ and $\wh{\gamma}_2$.
Let $\wh{z}_1 \in \wh{\gamma}_1$ and  $\wh{z}_2 \in \wh{\gamma}_2$ be the corresponding points (chosen in the same way as above).
We now glue the corresponding polygons (whose existence follows from the induction hypothesis), to the appropriate sides of $\wh{\gamma_1}$ and $\wh{\gamma}_2$. We also add the triangle $\wh{T}$ to obtain the polygon $\Pol$. The rest is proved in the same way as in the case $k>3$.
\end{proof}

\section{Measures on triangles and the proof of Theorem 1.1}
\subsection{Measures on triangles and the $\ph$ operator} 
Let $\NBG(S)$ and  $\NBG(G)$ denote the unit normal bundles of $\Gamma(S)$ and $\Gamma(G)$ respectively. Here $\NBG(S)$ is the union of the normal bundles
$\NB\gamma$, where $\gamma \in \Gamma(S)$. Every point in $\NB\gamma$ is a pair $(z,\vec{n})$ where $z \in \gamma$ and $\vec{n}$ is the unit vector in the tangent space $\TB{S}$ at $z$ that is orthogonal to $\gamma$. Every  point $(z,\vec{n}) \in \NB\gamma$ 
corresponds to a unique pair $(\gamma^{*},z)$, where $\gamma^{*} \in \Gamma^{*}$ and $z \in \gamma$, such  that the vector $\vec{n}$ points to the left of $\gamma^{*}$.  
\vskip .1cm
Connected components of $\NBG(S)$ are the normal bundles $\NB\gamma^{*}$, where $\gamma^{*} \in \Gamma^{*}(S)$, and $\NBG(S)$ is the disjoint union $\NBG(S)=\cup_{\gamma^{*} \in \Gamma^{*}(S)}\NB\gamma^{*}$. The normal bundle $\NB\gamma^{*}$ is the subset of the normal bundle $\NB\gamma$ that contains the unit vectors that are pointing to the left of $\gamma^{*}$. Note that for a fixed geodesic $\gamma$ and a chosen orientation $\gamma^{*}$ the space $\NB\gamma^{*}$ is a connected 1-manifold. This manifold  is identified with $\gamma$ by the obvious projection map 
$\NB\gamma^{*} \to \gamma$. Therefore,  the  set $\NBG(S)$ is a 1-manifold with countably many components. Note that the group $G$ acts on $\NBG(G)$ 
and  $\NBG(S)=\NBG(G)/G$.

\begin{remark} Since every point  in $\NBG(G)$ corresponds to a unique pair  $(\gamma^{*},z)$. Therefore,  we  define the combinatorial length mapping
$\cl:\NBG(G) \to \N$ as $\cl(z,\vec{n})=\cl(\gamma^{*},z)$, where $\vec{n}$ points to the left of $\gamma^{*}$. 
The induced map $\cl:\NBG(S) \to \N$ is also denoted by $\cl$.
\end{remark}

\begin{definition} Let $\gamma$ be a geodesic in $\Ha$. We define the $foot$ $projection$ $\ft_{\gamma}:\Ha \setminus \gamma \to \NB\gamma$ as follows.
Let $w \in \Ha \setminus \gamma$ and let $z \in \gamma$ denote the orthogonal projection of $w$ to $\gamma$. Then $\ft_{\gamma}(w) \in \NB\gamma$ 
is the point $(z,\vec{n})$ such that the geodesic ray that starts at $(z,\vec{n})$ contains $w$.
\end{definition}
Unless otherwise stated, all measures in this paper are positive, Borel measures.
\begin{definition} If $X$ is a topological space, then $\Mes(X)$ denotes the space of  positive, Borel measures on $X$ (necessarily finite). 
If $X$ is a countable set,
we equip $X$ with the discrete topology. In particular, by $\Mes(\Tr(S))$ and $\Mes(\NBG(S))$ we denote the spaces of positive, Borel measures 
on the set of triangles $\Tr(S)$ and the manifold $\NBG(S)$. The corresponding spaces of measures on $\Tr(G)$ and $\NBG(G)$ are denoted by  
$\Mes(\Tr(G))$ and $\Mes(\NBG(G))$. 
\end{definition}
\begin{remark} Note that $\N\Tr(S)$ can be seen as the subset of $\Mes(\Tr(S))$. Each $R \in \N\Tr(S)$ induces a measure in  $\Mes(\Tr(S))$
in the obvious way (if $R=k_1T_1+...+k_mT_m$ then the corresponding measure $\mu \in \Mes(\Tr(S))$ satisfies that $\mu(T_i)=k_i$).
\end{remark}
We define the $\ph:\Mes(\Tr(S))\to \Mes(\NBG(S))$ operator as follows. The set $\Tr(S)$ is a countable set, so every measure
from $\Mes(\Tr(S))$ is determined by its value on every triangle in $\Tr(S)$. Let $T \in \Tr(S)$ and let $\gamma_i \in \Gamma(S)$, $i=1,2,3$, denote its edges. Set $V_i=\ft_{\gamma_{i}}(\ct(T))$ where $\ct(T)$ is the centre of $T$. Choose $\mu \in \Mes(\Tr(S))$. Let $\alpha^{T} \in \Mes(\NBG(S))$ be the atomic measure, supported on $V_1$, $V_2$ and $V_3$ such that $\alpha^{T}(V_i)=\mu(T)$. We define $\ph\mu=\alpha \in \Mes(\NBG(S))$ as
$$
\alpha=\sum_{T \in \Tr(S)} \alpha^{T}.
$$
\noindent
If $\mu \in \Mes(\Tr(S))$ is a finite measure, then $\ph\mu$ is a finite measure as well.
Moreover, the total measure of $\ph\mu$ is three times the total measure of $\mu$.

\subsection{Transport of measure} In this subsection we define the notion of equivalent measures. The following is the standard result in measure theory.

\begin{proposition} Let $\mu_i \in \Mes(\R)$, $i=1,2$, be two finite measures on the real line $\R$. Let $K>0$. The following conditions are equivalent:

\begin{enumerate}
\item For every $a \in \R$ the following inequalities hold
\begin{equation}\label{transport-condition}
\mu_1(-\infty,a] \le \mu_2(-\infty,a+K], \, \mu_2(-\infty,a] \le \mu_1(-\infty,a+K].
\end{equation}
\noindent
\item The total $\mu_1$ and $\mu_2$ measures of $\R$ coincide, that is $\mu_1(\R)=\mu_2(\R)$. There exist mappings $\psi_i:(0,\mu_1(\R)] \to \R$ so that
$(\psi_i)_{*}(\nu)=\mu_i$ and $||\psi_1-\psi_2||_{\infty} \le K$. Here $\nu$ denotes the Lebesgue measure on the interval $(0,\mu_1(\R)]$ and $(\psi_i)_{*}(\nu)$ denotes the push-forward of the measure $\nu$ under $\psi_i$. Moreover, the mapping $\psi_i$ is non-decreasing and left continuous.
\item There exists a topological space $X$ with a Borel measure $\eta \in \Mes(X)$ so that the following holds. There exist mappings $\psi_i:X \to \R$ so that
$(\psi_i)_{*}(\eta)=\mu_i$ and $||\psi_1-\psi_2||_{\infty} \le K$. 
\end{enumerate}
If either of these three conditions is satisfied we say that $\mu_1$ and $\mu_2$ are $K$-equivalent measures.
\end{proposition}
\begin{remark} If $g(x)=x+c$ is a translation, then the measures $\mu_1$ and $\mu_2$ are $K$-equivalent, if and only if
the measures $g_{*}\mu_1$ and $g_{*}\mu_2$ are. Here $g_{*}\mu_i$ denotes the push-forward of the measure $\mu_i$ by $g$. Since $g$ is a homeomorphism
the similar statement is true for the pull-backs of $\mu_1$ and $\mu_2$. Also, note that the relation of being $K$-equivalent is symmetric. 
Moreover, if  a pair of measures $\mu_1,\mu_2$  is $K_1$ equivalent,  and a pair $\nu_1,\nu_2$ is $K_2$-equivalent,  
then the measures $\mu_1+\nu_1$ and $\mu_2+\nu_2$ are $\max \{K_1,K_2 \}$-equivalent.

\end{remark}
\begin{proof} The implications $(2) \to (1)$ and $(3) \to (1)$ are elementary. We prove  $(1) \to (2)$. This also proves $(1) \to (3)$ since one can take $X=(0,\mu_1(\R)]$. 
\vskip .1cm
Note that (\ref{transport-condition}) implies that the total $\mu_1$
and $\mu_2$ measures of $\R$ coincide, that is $\mu_1(\R)=\mu_2(\R)$. Set $I=(0, \mu_1(\R)]$. Define
$\psi_i$ as follows. Set 
$$
E_i(x)=\{y \in \R: \mu_i(-\infty,y] \ge x\}, 
$$
\noindent
for $i=1,2$. Let $\psi_i(x)=\inf{E_i(x)}=\min{E_i(x)}$. The fact the infimum of the set $E_i(x)$ belongs to this set follows from the countable additivity property of measures. In particular,
$$
\int\limits_{\infty}^{\psi_{i}(x)}d\mu_i \ge x.
$$
\noindent
Note that $\psi_i$ is  non-decreasing. It is elementary to check that $(\psi_i)_{*}(\nu)=\mu_i$.
We now show that $||\psi_1-\psi_2||_{\infty}\le K$. Let $x \in I$. It follows from (\ref{transport-condition})
that
$$
x \le \int\limits_{-\infty}^{\psi_{1}(x)}d\mu_1 \le \int\limits_{-\infty}^{\psi_1(x)+K}d\mu_2.
$$
\noindent
This shows that $(\psi_1(x)+K) \in E_2(x)$. We conclude $\psi_2(x) \le \psi_1(x)+K$. Similarly one shows $\psi_1(x) \le \psi_2(x)+K$. This proves the proposition.
\end{proof}

For $\gamma \in \Gamma(S)$ let $\phi:\R \to \gamma$ be an isometric parametrisation (from the Euclidean metric on $\R$ to the hyperbolic
on $\gamma$). Every other isometric parametrisation is obtained by pre-composing $\phi$ by a translation on $\R$. 
Let $\alpha \in \Mes(\NBG(S))$ and let $\gamma^{*} \in \Gamma^{*}(S)$. The restriction of the measure $\alpha$ to $\NB\gamma^{*}$ is denoted by $\alpha_{\gamma^{*}}$. Same as before, we identify (in the obvious way), the manifold $\NB\gamma^{*}$ and $\gamma$. Therefore, the
pull-back measure $\phi^{*}\alpha_{\gamma^{*}}$ is well defined.

\begin{definition} Let $\alpha, \alpha' \in \NBG(S)$ and let $K>0$. 
For  $\gamma \in \Gamma(S)$ let $\gamma^{*}(1), \gamma^{*}(2) \in \Gamma^{*}(S)$ denote the two orientations on $\gamma$. 
Let $\phi:\R \to \gamma$ be any isometric parametrisation. We have
\begin{itemize}
\item If for every $\gamma \in \Gamma(S)$ the measures $\phi^{*}\alpha_{\gamma^{*}(i)}$ and  $\phi^{*}\alpha'_{\gamma^{*}(i)}$, $i=1,2$, are $K$-equivalent, then we say that the measures $\alpha$ and $\alpha'$ are $K$-equivalent.
\item If for some $\gamma \in \Gamma(S)$ the measures $\phi^{*}\alpha_{\gamma^{*}(1)}$ and  $\phi^{*}\alpha_{\gamma^{*}(2)}$ are $K$-equivalent, we say that the measure $\alpha$ is $K$-symmetric on $\gamma$. If $\alpha$ is $K$-symmetric on every $\gamma \in \Gamma(S)$ then we say that $\alpha$ is $K$-symmetric.
\end{itemize}
\end{definition}
The above definition does not depend on the choice of the parametrisation $\phi$ (see the above remark).
\vskip .1cm
The following propositions are needed in the proof of Theorem 1.1 below.
\begin{proposition} Let $0<\epsilon<1$. Let $\alpha_1,\alpha_2 \in \Mes(\R)$ be discrete measures with finitely many non-trivial atoms, and suppose that
$\alpha_1$ and $\alpha_2$ are $K$-equivalent. Then there are measures $\alpha^{\rat}_i,\alpha'_i \in \Mes(\R)$, $i=1,2$,
so that $\alpha^{\rat}_i+\alpha'_i=\alpha_i$ and $\alpha^{\rat}_1$ and $\alpha^{\rat}_2$ are $K$ equivalent. 
Also,  $\alpha^{\rat}_i$ has atoms of rational weights,  and the weight of any atom of $\alpha'_i$ is at most $\epsilon$.
\end{proposition}
\begin{proof}
Let $a_i \in \R$ and $b_j \in \R$ denote respectively the points where $\alpha_1$ and $\alpha_2$ have non-trivial atoms.
Set $x_i=\alpha_1(a_i)$ and $y_j=\alpha_2(b_j)$. Let $m_1$ be the total number of atoms $a_i$ and let $m_2$ be the total number of atoms $b_j$.
Set $m=m_1+m_2$. Also, let $A$ be the minimum of all non-zero weights of atoms of both measures, and $B$ the maximum of all non-zero weights.
\vskip .1cm
Since $\alpha_1$ and $\alpha_2$ are $K$-equivalent, we have that (\ref{transport-condition}) holds.
Since $\alpha_1$ and $\alpha_2$ have finitely many atoms, the condition (\ref{transport-condition}) becomes a finite systems of linear inequalities with integer coefficients, in $x_i$ and $y_j$. Each such inequality has the form
$$
\sum_{i}\sigma_1(i)x_{i}\le \sum_{j}\sigma_2(j)y_{j},
$$
or
$$
\sum_{j}\sigma_2(j)y_{j} \le \sum_{i}\sigma_1(i)x_{i},
$$
\noindent
where every $\sigma_1(i)$, $\sigma_2(j)$ is either $1$ or $0$.
If we treat $x_i$ and $y_j$ as real variables, then we conclude that this system of linear inequalities has a non-trivial solution. 
In fact, the set  of solutions of each above inequality is a half-space in $\R^{m}$ (each half space contains the origin in $\R^{m}$). The set of the solutions of the entire system is the intersection of all these half-spaces in $\R^{m}$. We denote this set by $\Sol$.  Let $N \in \N$ and let $\Sol_N=\Sol \cap \{|x_i|,|y_j| \le N \}$. We have that $\Sol_N$ is a convex polyhedron (possibly degenerate) in $\R^{m}$.  By the Krein-Milman theorem, $\Sol_N$ is the closure of the convex combinations of the extreme points on $\Sol_N$ (there are finitely many extreme points). Each extreme point is the unique solution of a certain system of equations with integer coefficients. Therefore, every extreme point is a rational point in $\R^{m}$. We conclude that the rational points in $\R^{m}$ are dense in
every $\Sol_N$ and therefore the rational points are dense in $\Sol=\cup_{N \in \N} \Sol_N$. 
\vskip .1cm
If $x_i$ and $y_j$ is the only solution to this system, then these numbers have to be rational. If this is not the only solution, then we can choose  rational solutions $x^{\rat}_i$ and $y^{\rat}_j$ to be as close to $x_i$ and $y_j$ as we want. 
Fix any $\epsilon_1>0$ and let $x^{\rat}_i$ and $y^{\rat}_j$ be rational numbers that satisfy all the above inequalities, and  such that $A\epsilon_1>|x_i-x^{\rat}_i|$ and $A\epsilon_1>|y_j-y^{\rat}_j |$. Let $t$ be a rational number so that
$$
1-2{{\epsilon_1}\over{A}}<t<1-{{\epsilon_1}\over{A}}.
$$
\noindent
Then $tx^{\rat}_i<x_i$ and $ty^{\rat}_j<y_j$ and also $tx^{\rat}_i$ and $ty^{\rat}_j$ satisfy all the above inequalities. Moreover, the following inequalities hold
$$
|x_i-x^{\rat}_i|<\epsilon_1\big( {{2B}\over{A}}+A\epsilon_1+2{\epsilon_1}^{2}\big),
$$
\noindent
and 
$$
|y_j-y^{\rat}_j|<\epsilon_1\big( {{2B}\over{A}}+A\epsilon_1+2{\epsilon_1}^{2}\big).
$$
\noindent
Choose $\epsilon_1$ small enough so that $\epsilon_1\big( {{2B}\over{A}}+A\epsilon_1+2{\epsilon_1}^{2}\big)<\epsilon$.
\vskip .1cm
Let $\alpha^{\rat}_1$ be the measure with the same non-trivial atoms as $\alpha_1$ and $\alpha^{\rat}_1(a_i)=tx^{\rat}_i$. Similarly define  $\alpha^{\rat}_2(b_j)=ty^{\rat}_j$. Set $\alpha'_1=\alpha_1-\alpha^{\rat}_1$ and
$\alpha'_2=\alpha_2-\alpha^{\rat}_2$. The measures $\alpha'_1$ and $\alpha'_2$ are non-negative, and the weight of any atom under either of these two measures is less than $\epsilon$. Since $tx^{\rat}_i$ and $ty^{\rat}_j$ satisfy the same inequalities as $x_i$ and $y_j$ we conclude that
$\alpha^{\rat}_1$ and $\alpha^{\rat}_2$ are $K$-equivalent.
\end{proof}

Let L be a finite set. By $\vartheta_L \in \Mes(L)$ we denote the counting measure, that is if $L_1 \subset L$ then $\vartheta_L(L_1)$ is equal to the number of elements in $L_1$. On the other hand, we say that a Borel measure is integral if the measure of every set is an integer. The following proposition is left to the reader.

\begin{proposition} Let $L$ be a finite set, and let $T:L \to A$ be a function. Suppose that $T_{*}(\vartheta_L)=\alpha_1+\alpha_2$ such that $\alpha_1,\alpha_2 \in \Mes(A)$ are integral measures. Then we can write $L=L_1 \cup L_2$ where $L_1$ and $L_2$ are disjoint, such that $T_{*}(\vartheta_{L_{1}})=\alpha_1$ and $T_{*}(\vartheta_{L_{2}})=\alpha_2$.
\end{proposition}

\begin{proposition} Let $\Lab_A$ and $\Lab_B$ be finite sets of labels, and let 
$\lab_A:\Lab_A \to \R$ and $\lab_B:\Lab_B \to \R$ be labelling maps. Suppose that the measures $\alpha=(\lab_A)_{*}(\vartheta_{\Lab_{A}})$ and 
$\beta=(\lab_B)_{*}(\vartheta_{\Lab_{B}})$ are $K$-equivalent. Then we can find a bijection $\sigma_{A,B}:\Lab_A \to \Lab_B$ such that $||\lab_A-\lab_B \circ \sigma_{A,B}||_{\infty} \le K$.
\end{proposition}

\begin{proof} Since $\alpha$ and $\beta$ are $K$-equivalent, we conclude that the total mass of $\alpha$ is equal to the total mass of  $\beta$.
Set $m=\alpha(\R)=\beta(\R)$. By Proposition 4.1, we can find non-decreasing and left continuous functions $\psi_A,\psi_B:(0,m] \to \R$ such that
$||\psi_A-\psi_B||_{\infty} \le K$ where $(\psi_l)_{*}(\nu)=(\lab_l)_{*}(\vartheta_{\Lab_{l}})$ for $l=A,B$ (here $\nu$ is the Lebesgue measure on $(0,m]$). 
Since $(\lab_l)_{*}(\vartheta_{\Lab_{l}})$ is an integral measure, we have that $\psi_l(t)=\psi_l(t^*)$ where $t^*$ is the least integer greater or equal to $t$.
Therefore, we can find bijections $\phi_l \to \{1,2,...,m\}$, $l=A,B$, such that $\psi_l(\phi_l(x))=\lab_l(x)$ for every $x \in \Lab_l$. 
\vskip .1cm
For $a \in \Lab_A$  we let $\sigma_{A,B}(a)=(\phi_B)^{-1}(\phi_A(a))$. Then $|\lab_A(a)-\lab_B(\sigma_{A,B}(a))|=|\psi_A(\phi_A(a))-\psi_B(\phi_A(a))| \le K$ which proves the proposition.
\end{proof}

\begin{proposition} Let $A$ and $B$ be finite sets of real numbers. Let $\Lab_A$ and $\Lab_B$ be finite sets of labels, and
$\lab_A:\Lab_A \to A$ and $\lab_B:\Lab_B \to B$ labelling maps.
Let $\alpha, \beta \in \Mes(\R)$ be the atomic measures, that are supported
on $A$ and $B$ respectively, and $\alpha(a)=|\lab^{-1}_A(a)|$ for every $a \in A$ and  $\beta(b)=|\lab^{-1}_B(b)|$ for every $b \in B$
(here $|\lab^{-1}_A(a)|$, $|\lab^{-1}_B(b)|$ denote the numbers of the corresponding preimages in $\Lab_A$ and $\Lab_B$ respectively).
Suppose that $\alpha=\alpha_1+\alpha_2$ and $\beta=\beta_1+\beta_2$ where $\alpha_i,\beta_j \in \Mes(\R)$ are integral measures,  such that $\alpha_1$, $\beta_1$  are $K_1$-equivalent, and $\alpha_2$, $\beta_2$ are $K_2$-equivalent. Denote by $m_1$ the total mass of $\alpha_1$ and by $m_2$ the total mass of $\alpha_2$. Then, there is a bijection  $\sigma_{A,B}:\Lab_A \to \Lab_B$ so that for $m_1$ elements $a \in \Lab_A$ we have
$|\lab_B(\sigma_{A,B}(a))-\lab_A(a)|<K_1$ and for the remaining $m_2$ elements  $b \in \Lab_A$ we have $|\lab_B(\sigma_{A,B}(b))-\lab_A(b)|<K_2$.
\end{proposition}
\begin{proof} Note that $\alpha=(\lab_A)_{*}(\vartheta_{\Lab_{A}})$ and $\beta=(\lab_B)_{*}(\vartheta_{\Lab_{B}})$. The proof follows from Proposition 4.3 and Proposition 4.4.
\end{proof}

We end this subsection with the following two elementary propositions that will be used in Section 6.

\begin{proposition} Let $\alpha,\beta, \eta \in \Mes(\R)$ such that $\alpha$ and $\beta$ are $K$-equivalent, and such that $\beta$ and $\eta$ are $L$-equivalent. Then $\alpha$ and $\eta$ are $K+L$-equivalent.
\end{proposition}
\begin{proof} One directly verifies that the condition $(1)$ from Proposition 4.1 holds for the measures $\alpha$ and $\eta$. 
\end{proof}

\begin{proposition} Let $(X,\nu)$ be a measure space and let $\mu_i:X \to \Mes(\R)$, $i=1,2$,   be such that the mapping 
$x \to \mu_i(x)(-\infty,t]$ is measurable for every $t \in \R$. Suppose that $\mu_1(x)$ and $\mu_2(x)$ are $K$-equivalent for every $x \in X$. Then the measures
$$
\mu_1(X)=\int\limits_{X} \mu_1(x) \, d\nu(x),
$$
\noindent
and 
$$
\mu_2(X)=\int\limits_{X} \mu_2(x) \, d\nu(x),
$$
\noindent
are $K$-equivalent.
\end{proposition}
\begin{remark} If  $\mu:X \to \Mes(\NB\R)$ is a measurable mapping, and if $\mu(x)$ is $K$-symmetric for every $x$, then the measure
$$
\mu(X)=\int\limits_{X} \mu(x) \, d\nu(x),
$$
\noindent
is also  $K$-symmetric.
\end{remark} 

\begin{proof} One verifies directly that the condition $(1)$ from Proposition 4.1 holds for the measures $\mu_1(X)$ and $\mu_2(X)$.
\end{proof}

\subsection{Proof of Theorem 1.1}  The remainder of the paper after this subsection is devoted to proving the next theorem.
From now on $P(r)$ denotes a polynomial that only depends on $S$. In particular we have $P(r)+P(r)=P(r)$ and $rP(r)=P(r)$.

\begin{theorem} There exist a constant $r_0(S)=r_0$ that depends only on $S$, so that for every $r>r_0$  there exists a finite measure $\mu(r) \in \Mes(\Tr(S))$ so that the total measure $|\mu(r)|$ satisfies the inequality

$$
{{\LM(\TB{S})}\over{2}} < |\mu(r)|<{{3\LM(\TB{S})}\over{2}},
$$
\noindent
and  with the following properties. There exist measures $\alpha(r),\alpha_1(r),\beta(r) \in \Mes(\NBG(S))$ so that the measure $\ph\mu(r) \in \Mes(\NBG(S))$ can be written  as $\ph\mu(r)=\alpha(r)+\alpha_1(r)+\beta(r)$ and  the following holds
\begin{enumerate}
\item Let $\wh{\mu}(r)$ denote the restriction of the measure $\mu(r)$ to the set of triangles $T \in \Tr(S)$ for which $\h(T) \ge r^2$. 
Then
$$
\int\limits_{\NBG (S) } (\cl(\gamma^{*},z)+1) d \ph \wh{\mu}(r) \le P(r)e^{-r}.
$$
\item The measure $\alpha(r)$ is $re^{-r}$-symmetric, and the measure $\alpha_1(r)$ is $Q$-symmetric, for any   $Q >200$.
\item We have
$$
\int\limits_{\NBG(S)}(\cl(\gamma^{*},z)+1) d\beta(r) \le P(r)e^{-r}.
$$
\noindent
\item We have
$$
\int\limits_{\NBG(S)}d\alpha_1(r) \le e^{-r}.
$$
\noindent

\end{enumerate}

\end{theorem}

We will explicitly construct the required measure $\mu(r)$. In the remainder of this section we prove Theorem 1.1 assuming Theorem 4.1 and its notation.
First, we show that we may assume that the measure $\mu(r)$ from Theorem 4.1, has finite support.

\begin{proposition} Assume that Theorem 4.1 holds. Then there exist a constant $r_0(S)=r_0$ that depends only on $S$ so that for every $r>r_0$  there exists a finite measure $\wt{\mu}(r) \in \Mes(\Tr(S))$ with finite support, so that the total measure $|\wt{\mu}(r)|$ satisfies the inequality

$$
{{\LM(\TB{S})}\over{4}} < |\wt{\mu}(r)|<2\LM(\TB{S}),
$$
\noindent
and  with the following properties. There exist measures $\wt{\alpha}(r),\wt{\alpha}_1(r),\wt{\beta}(r) \in \Mes(\NBG(S))$ so that the measure $\ph \wt{\mu}(r) \in \Mes(\NBG(S))$ can be written  as $\ph \wt{\mu}(r)=\wt{\alpha}(r)+\wt{\alpha}_1(r)+\wt{\beta}(r)$ and  the following holds
\begin{enumerate}

\item The measure $\wt{\alpha}(r)$ is $re^{-r}$-symmetric, and the measure $\wt{\alpha}_1(r)$ is $Q$-symmetric, for any  $Q>200$.
\item We have
$$
\int\limits_{\NBG(S)}(\cl(\gamma^{*},z)+1) d \wt{\beta}(r) \le P(r)e^{-r}.
$$
\noindent
\item We have
$$
\int\limits_{\NBG(S)}d\wt{\alpha}_1(r) \le e^{-r}.
$$
\noindent

\end{enumerate}

\end{proposition}

\begin{proof} Let $\wt{\mu}(r) \in \Mes(\Tr(S))$ be such that $\mu(r)=\wt{\mu}(r)+\wh{\mu}(r)$ where $\wh{\mu}(r)$ was defined in the statement of Theorem 4.1. Then $\wt{\mu}(r)$ has finite support since it is supported only on triangles from $\Tr(S)$ which satisfy $\h(T)\le r^2$ (and there are only finitely many such triangles). By the assumption in Theorem 4.1 we have  $|\wh{\mu}(r)| \le P(r)e^{-r}$. This shows that 
$$
{{\LM(\TB{S})}\over{4}} < |\wt{\mu}(r)|<2\LM(\TB{S}),
$$
\noindent
for $r$ large enough. Note that the measure $\ph\mu(r)$ consists of countably many atoms, that are obtained as  the feet of the centres of the triangles from $\Tr(S)$. The measure $\wt{\mu}(r)$ consists of finitely many atoms. It remains to construct the decomposition $\ph \wt{\mu}(r)=\wt{\alpha}(r)+\wt{\alpha}_1(r)+\wt{\beta}(r)$
\vskip .1cm
For a discrete positive measure $\eta$ we let $\supp (\eta)$ be the set of points $t \in \R$ for which $\eta(\{t\}) >0$. 
We observe that $\wh{\mu}$ and $\wt{\mu}$ have disjoint support because if $T_1,T_2$ are triangles that share a geodesic $\gamma$, and such that $\ft_{\gamma}(\ct(T_1))=\ft_{\gamma}(\ct(T_2))$ then $T_1=T_2$ (recall that $\ft_{\gamma}$ is an element of $\NB\gamma$ so if 
$\ft_{\gamma}(\ct(T_1))=\ft_{\gamma}(\ct(T_2))$ then $T_1$ and $T_2$ are on the same side of $\gamma$). We let $\underline{\alpha}$ be the restriction of $\alpha$ to $\supp{\wt{\mu}}$ and $\wh{\alpha}$ the restriction of $\alpha$ to $\supp{\wh{\mu}}$. One defines $\underline{\alpha}_1$, $\wh{\alpha}_1$,
$\underline{\beta}$ and $\wh{\beta}$ likewise. Then 
$$
\wt{\mu}=\underline{\alpha}+\underline{\alpha}_1+\underline{\beta}, \quad \quad \wh{\mu}=\wh{\alpha}+\wh{\alpha}_1+\wh{\beta},
$$
\noindent
and $\underline{\alpha}+\wh{\alpha}=\alpha$,  $\underline{\alpha}_1+\wh{\alpha}_1=\alpha_1$ and $\underline{\beta}+\wh{\beta}=\beta$.
\vskip .1cm
We aim to ''symmetrise" $\underline{\alpha}$ so that it is $re^{-r}$-symmetric. To this end for each $\gamma \in \Gamma(S)$ we choose an orientation $\gamma^{*} \in \Gamma^{*}(S)$ for $\gamma$. Let $\alpha^{+}$ be the restriction of $\alpha$  to $\NB\gamma^{*}$ and $\alpha^{-}$ be the restriction of $\alpha$  to $\NB(-\gamma^{*})$, and think of $\alpha^{+}$ and $\alpha^{-}$ as measures on $\gamma$. We define $\underline{\alpha}^{+}$, $\underline{\alpha}^{-}$, $\wh{\alpha}^{+}$ and $\wh{\alpha}^{-}$ likewise. Since $\alpha$ is $re^{-r}$-symmetric by Proposition 4.1  we can write 
$$
\alpha^{+/-}=(\psi_{+/-})_{*}[0,\alpha^{+}(\gamma)),
$$
\noindent
such that $\psi_{+/-}: [0,\alpha^{+}(\gamma))\to \gamma$ satisfy that $\dis(\psi_{+}(t),\psi_{-}(t)) \le re^{-r}$ for all $t \in [0,\alpha^{+}(\gamma))$.
Then we let $A_{+/-} \subset [0,\alpha^{+}(\gamma))$ be defined by $A_{+/-}=\psi_{+/-}^{-1}(\supp(\alpha_{+/-}))$. Set $A=A_{+} \cap A_{-}$.
Also, let $\wt{\alpha}^{+/-}=(\psi_{+/-})_{*}A$. By construction $\wt{\alpha}^{+}$ and $\wt{\alpha}^{-}$ are $re^{-r}$-equivalent. 
Now consider $\wt{\alpha}^{+}$ and $\wt{\alpha}^{-}$ as measures on $\NB\gamma^{*}$ and $\NB(-\gamma^{*})$ respectively. Let $\wt{\alpha}=\wt{\alpha}^{+}+\wt{\alpha}^{-}$. Then $\wt{\alpha}$ is $re^{-r}$-symmetric. 
\vskip .1cm
We let $\eta=\underline{\alpha}-\wt{\alpha}$. The reader can verify that
there exists a measure $\eta' \le \wh{\alpha}$ such that $\eta'+\eta$ is $re^{-r}$-symmetric (the measure $\eta'$ is constructed in the obvious way using
the maps $\psi_{+/-}$). Note that if $(\gamma^{*},z), (\gamma^{*},z') \in \NB\gamma$ satisfy $\dis(z,z') \le K$ then $\cl(\gamma^{*},z) \le e^{K}\cl(-\gamma^{*},z_1)$. It follows that
$$
\int\limits_{\NBG(S)} (\cl(\gamma^{*},z)+1) \, d\eta \le e^{re^{-r}}\int\limits_{\NBG(S)} (\cl(\gamma^{*},z)+1) \, d\eta' \le
$$

$$
\le e^{re^{-r}}\int\limits_{\NBG(S)} (\cl(\gamma^{*},z)+1) \, d\wh{\alpha}  \le 3\int\limits_{\NBG(S)} (\cl(\gamma^{*},z)+1) \, d\wh{\alpha},
$$
\noindent
for $r$ large enough. 
\vskip .1cm
We likewise find $\wt{\alpha}_1 \le \underline{\alpha}_1$ such that $\wt{\alpha}_1$ is $Q$-symmetric and $\eta_1=\underline{\alpha}_1-\wt{\alpha}_1$ satisfy 
$$
\int\limits_{\NBG(S)} (\cl(\gamma^{*},z)+1) \, d\eta_1 \le e^{Q}\int\limits_{\NBG(S)} (\cl(\gamma^{*},z)+1) \, d\wh{\alpha}_1.
$$
\noindent
Also
$$
\int\limits_{\NBG(S)} \, d \wt{\alpha}_1 \le e^{-r}.
$$
\noindent
This shows that $\wt{\alpha}$ and $\wt{\alpha}_1$ satisfy the conditions $(1)$ and $(3)$ of this proposition. 
\vskip .1cm
We then let $\wt{\beta}=\underline{\beta}+\eta+\eta_1$. Then $\wt{\mu}=\wt{\alpha}+\wt{\alpha}_1+\wt{\beta}$. 
Moreover
$$
\int\limits_{\NBG(S)} (\cl(\gamma^{*},z)) \, d\wt{\beta}=\int\limits_{\NBG(S)} (\cl(\gamma^{*},z)) \, d(\underline{\beta}+\eta+\eta_1) \le
$$
$$
\le \int\limits_{\NBG(S)} (\cl(\gamma^{*},z)) \, d\beta +\le \int\limits_{\NBG(S)} (\cl(\gamma^{*},z)) \, d(\wh{\alpha}+\wh{\alpha}_1) \le P(r)e^{-r},
$$
\noindent 
since $\wh{\alpha}+\wh{\alpha}_1 \le \wh{\mu}$. This proves the proposition.

\end{proof}

The idea is to use the measures $\wt{\mu}(r)$ to construct certain admissible pairs $(\Col(r),\sigma_{\Col(r)})$. In turn, this will enable us to construct the corresponding covers of $S$ that are required by Theorem 1.1.
\vskip .1cm
From now on we suppress the dependence on $r$ that is we set  $\wt{\mu}(r)=\mu$, $\wt{\alpha}(r)=\alpha$, $\wt{\alpha}_1(r)=\alpha_1$ and $\wt{\beta}(r)=\beta$. Since $\mu$ is finitely supported, the measure $\mu$ is atomic and it has finitely many atoms. The measure $\ph\mu$ is atomic and it has finitely many atoms, as well. Let $m$ denote the total number of non-trivial atoms for $\ph\mu$. Let $L$ be the maximum of the combinatorial length function $\cl$ over the $m$  points where $\ph\mu$ is supported. Let 
$$
\epsilon={{1}\over{m(L+1)e^{r}}}.
$$ 
\vskip .1cm
For $T \in \Tr(S)$ choose a number $0 \le \mu'(T)\le \epsilon$ so  that $\mu'(T)+\mu(T)$
is a rational number, and set $\mu'(T)+\mu(T)=\mu^{\rat}(T)$. If $\mu(T)=0$ then we set $\mu'(T)=0$ as well. This is how we define two new 
measures $\mu',\mu^{\rat} \in \Mes(\Tr(S))$. Both these measures are atomic (finitely many atoms), and $\mu^{\rat}$ has the same set of atoms as $\mu$ (the weights of atoms of $\mu^{\rat}$ are rational numbers). In particular, $\mu'$ has at most $m$ atoms, and each of them has  the weight at most $\epsilon$. Note that the total measure of $\mu^{\rat}$ satisfies $\mu(\Tr(S)) \ge \mu^{\rat}(\Tr(S)) \ge \mu(\Tr(S))-m\epsilon>\mu(\Tr(S))-e^{-r}$
\vskip .1cm
The measure $\ph\mu \in \Mes(\NBG(S))$ is also atomic, with finitely many atoms as well. Let $\gamma \in \Gamma(S)$ and let
$\gamma^{*}_1, \gamma^{*}_2 \in \Gamma^{*}(G)$ denote the two orientations on $\gamma$.  Since $\alpha$ is $re^{-r}$-symmetric,
the two measures on $\R$ that arise as the restrictions of $\alpha$ on $\NB\gamma^{*}_1$ and $\NB\gamma^{*}_2$ respectively,
are atomic (finitely many atoms), $re^{-r}$-equivalent measures. Applying Proposition 4.2 on these two measures, and repeating the same for
each $\gamma \in \Gamma(S)$ we construct the measures $\alpha^{\rat},\alpha' \in \Mes(\NB\Gamma(S))$ with the  properties:
\begin{itemize}
\item $\alpha^{\rat}+\alpha'=\alpha$. 
\item $\alpha^{\rat}$ is $re^{-r}$-symmetric.
\item for any atom $V \in \NBG(S)$ of $\alpha'$ we have $\alpha'(V) \le \epsilon$. 
\end{itemize}
We repeat the same for $\alpha_1$. We construct the measures $\alpha_1^{\rat},\alpha'_1 \in \Mes(\NB\Gamma(S))$
with 
\begin{itemize}
\item $\alpha_1^{\rat}+\alpha'_1=\alpha_1$.
\item $\alpha^{\rat}$ is $Q$-symmetric.
\item for any atom $V \in \NBG(S)$ of $\alpha'_1$ we have $\alpha'_1(V) \le \epsilon$. 
\end{itemize}
We have
$$
\ph\mu^{\rat}=\ph\mu+\ph\mu'=\alpha^{\rat}+{\alpha_1}^{\rat}+(\beta+\ph\mu'+\alpha'+\alpha'_1).
$$
\noindent
Note that from the fourth inequality in Proposition 4.6, we have 
$$
\int\limits_{\NBG(S)}d\alpha^{\rat}_1 \le  \int\limits_{\NBG(S)}d\alpha_1 \le e^{-r}.
$$
\vskip .1cm
Set $\beta_1=\beta+\ph\mu'+\alpha'+\alpha'_1$. Since $\mu^{\rat}$ is atomic (finitely many atoms) with rational weights,
so is the measure $\ph\mu^{\rat}$. Since $\alpha^{\rat}$ and ${\alpha_1}^{\rat}$ are also atomic (finitely many atoms) with rational weights,
and since all the measure in question are positive, we conclude that $\beta_1$ is also atomic (finitely many atoms) and with rational weights.
Moreover, we have
$$
\int\limits_{\NBG(S)}(\cl(\gamma^{*},z)+1) d\beta_1 \le \int\limits_{\NBG(S)}(\cl(\gamma^{*},z)+1) d\beta+\int\limits_{\NBG(S)}(\cl(\gamma^{*},z) +1)
d\ph \mu'+
$$

$$
+\int\limits_{\NBG(S)}(\cl(\gamma^{*},z)+1) d\alpha'+\int\limits_{\NBG(S)}(\cl(\gamma^{*},z)+1) d\alpha'_1 
\le P(r)e^{-r}+3m(L+1)\epsilon,
$$
\noindent
that is, for $r$ large enough we have
$$
\int\limits_{\NBG(S)}(\cl(\gamma^{*},z)+1) d\beta_1 \le P(r)e^{-r}.
$$
\vskip .1cm
Let $n$ be a large enough integer so that the weights of atoms of the measures $\mu^{\rat}$, $\alpha^{\rat}$, ${\alpha_1}^{\rat}$ and
$\beta_1$ are all integers. We multiply the measures in question by $n$. Set 
$\mu^{\intg}=n\mu^{\rat}$, $\alpha^{\intg}=n\alpha^{\rat}$, $\alpha^{\intg}_1=n{\alpha_1}^{\rat}$ and
$\beta^{\intg}_1=n\beta_1$. We have that $\alpha^{\intg}$ is still  $re^{-r}$-symmetric, and  $\alpha^{\intg}_1$ is
$Q$-symmetric. Moreover, we have the following estimate for the total measure of $\alpha^{\intg}_1$
\begin{equation}\label{total-measure-2}
\int\limits_{\NBG(S)}d\alpha^{\intg}_1 \le n e^{-r}.
\end{equation}
\noindent
Also, the total measure of $\mu^{\intg}$ satisfies that $n \mu(\Tr(S))\ge \mu^{\intg}(\Tr(S))>n( \mu(\Tr(S) )-e^{-r})$.
\vskip .1cm
Now we apply the Correction lemma (Lemma 3.3). Let $(\gamma^{*},z) \in \NBG(S)$ be an atom of $\beta^{\intg}_1$ (note that the same point can be an atom of the measure $\alpha^{\intg}$ or $\alpha^{\intg}_1$).  Choose a lift of $(\gamma^{*},z)$ to $\NBG(G)$ and denote it also by $(\gamma^{*},z)$. We apply Lemma 3.3
to this pair, to obtain the corresponding polygon $\Pol$ in $\Ha$ that is to the right of $\gamma^{*}$. Project the triangles from the triangulation
$\tau(\Pol)$ to $S$. Let $T' \in \tau(\Pol)$ and let $[T']_G=T \in \Tr(S)$ be its projection. Let $\nu^{T} \in \Mes(\Tr(S))$ be the  measure that is supported on $T$ and so that $\nu^{T}(T)=\beta^{\intg}_1(\gamma^{*},z)$. Set  

$$
\nu^{(\gamma^{*},z)}=\sum_{T}\nu^{T},
$$ 
\noindent
where we sum over all such triangles $T$. From Lemma 3.3 we have the bound on the number of triangles in $\tau(\Pol)$ and thus we obtain the following estimate of the total measure
$$
\nu^{(\gamma^{*},z)}(\Tr(S)) \le \big(C\cl(\gamma^{*},z)+K\big)\beta^{\intg}_1(\gamma^{*},z).
$$
\noindent
Here $C$ and $K$ are the constants from Lemma 3.3.
Note that each atom of $\nu^{(\gamma^{*},z)}$ has an integer weight. 
\vskip .1cm
Let $D$ be the constant from Lemma 3.3, and let $\gamma_0 \in \Gamma(S) \setminus \lambda_{\Gen}(S)$ be a geodesic that lifts to  an edge of a triangle from $\tau(\Pol)$. Denote by $\beta^{\intg}_1|_{(\gamma^{*},z)}$ the restriction of
$\beta^{\intg}_1$ to the point $(\gamma^{*},z) \in \NB\Gamma(S)$.
Then the two measures on $\gamma_0$ that are the restrictions of the measure $\ph\nu^{(\gamma^{*},z)}+\beta^{\intg}_1|_{(\gamma^{*},z)}$ on $\NB\gamma_0$ are $D$-equivalent. This follows from Lemma 3.3, that is,
the two atoms of $\ph\nu^{(\gamma^{*},z)}+\beta^{\intg}_1|_{(\gamma^{*},z)}$ in $\NB\gamma_0$ (one on each side of $\gamma_0$), are within the hyperbolic distance $D$ (these two atoms have the same weight by the definition of $\nu^{(\gamma^{*},z)}$). 
\vskip .1cm
Repeat this process for every non-trivial atom of $\beta^{\intg}_1$ and set
$$
\nu=\sum_{(\gamma^{*},z)} \nu^{(\gamma^{*},z)},
$$
\noindent
where we sum over all non-trivial atoms for $\beta^{\intg}_1$. We have $\nu \in \Mes(\Tr(S))$ and
$$
\nu(\Tr(S)) \le \int\limits_{\NBG(S)}\big(C\cl(\gamma^{*},z)+ K\big) d\beta^{\intg}_1 \le nP(r)e^{-r}.
$$
\vskip .1cm

Let $\gamma_0 \in \Gamma(S) \setminus \lambda_{\Gen}(S)$. Then the two measures on $\gamma_0$ that are the  restrictions of the measure $\ph\nu+\beta^{\intg}_1$ to $\NB\gamma_0$ are $D$-equivalent. 
Let $\gamma_0 \in \lambda_{\Gen}(S)$. Then by Lemma 3.3 all the atoms of  $\ph\nu$ on $\gamma_0$ are within the $D/2$ hyperbolic distance  from the point $\midp(\gamma_0) \in \gamma_0$. If we can show that the total measures are equal, that is of $\ph\nu(\NB\gamma^{*}_0)=\ph\nu(\NB(-\gamma^{*}_0))$ that would show that the measure $\ph\nu+\beta^{\intg}_1$ is $D$-symmetric.
\vskip .1cm
Set $\mu^{\intg}_1=\mu^{\intg}+\nu$. Again, $\mu^{\intg}_1$ has finitely many atoms, and all the weights are integers. 
The above estimate for $\nu(\Tr(S))$ implies that, for $r$ large enough, the total measure of $\mu^{\intg}_1$ satisfies the following inequalities
\begin{equation}\label{number-triangles}
{{n\mu(\Tr(S))}\over{2}}<n(\mu(\Tr(S))-e^{-r})<\mu^{\intg}_1(\Tr(S))<n\mu(\Tr(S))+n\mu(\Tr(S))P(r)e^{-r}<2n\mu(\Tr(S)).
\end{equation}
\vskip .1cm
Set $\beta^{\intg}_2=\beta^{\intg}_1+\ph\nu$. Then the following equality holds
$$
\ph\mu^{\intg}_1=\alpha^{\intg}+\alpha^{\intg}_1+\beta^{\intg}_2.
$$
\vskip .1cm
We have the following estimates on the total measures of $\beta^{\intg}_2$
\begin{equation}\label{total-measure-1}
\beta^{\intg}_2(\Tr(S)) \le \int\limits_{\NBG(S)}d\nu+\int\limits_{\NBG(S)}d\beta^{\intg}_1<n P(r)e^{-r}.
\end{equation}
\vskip .1cm
Since $\mu^{\intg}_1$ has finitely many atoms with integer weights, we can consider $\mu^{\intg}_1$ as an element of $\N\Tr(S)$. We construct the labelled collection of triangles $\Col$ that corresponds to  $\mu^{\intg}_1 \in \N\Tr(S)$ (see remark before Definition 3.1) . Let $M$ denote the total number of elements of $\Lab_{\Col}$ (clearly $M$ is equal to the total measure of $3\mu^{\intg}_1$). 
From (\ref{number-triangles}) we have ${{3n\mu(\Tr(S))}\over{2}}<M<6n\mu(\Tr(S))$. 
\vskip .1cm
Fix $\gamma \in \Gamma(S)\setminus \lambda_{\Gen}(S)$ and choose an orientation $\gamma^{*}$ on $\gamma$. We have already seen that the restriction of the measure $\mu^{\intg}_1$ on $\NB\gamma$ can be written as the sum of three measures from $\Mes(\NB\gamma)$ where each of these three measures produces a pair of measures on $\gamma$ that are equivalent (for some constant). This implies that the sets
$\Lab_{\Col,{\gamma}^{*}},\Lab_{\Col,-{\gamma}^{*}} \subset \Lab_{\Col}$ have the same number of elements. This shows that such $\gamma^{*}$
does not figure in the formal sum $\partial\Col \in \Z\Gamma^{*}(S)$. So
$$
\partial{\Col}=\sum_{i=1}^{l}k_i\gamma^{*}_{k_{i}},
$$
\noindent
where $\gamma^{*}_{k_{i}} \in \lambda^{*}_{\Gen}(S)$ and $k_i \in \Z$. By Proposition 3.2 we have  that every $k_i=0$  
that is $\partial{\Col}=0$ in $\Z\Gamma^{*}(S)$. Thus for every $\gamma \in \Gamma(S)$ the total measures of $\ph\mu^{\intg}_1$ on $\NB\gamma^{*}$ and on $\NB(-\gamma^{*})$ are the same. As indicated above, this proves that the measure $\beta^{\intg}_2=\beta^{\intg}_1+\ph\nu$
is $D$-symmetric.
\vskip .1cm
Set $\alpha^{\intg}_2=\alpha^{\intg}_1+\beta^{\intg}_2$. Then $\alpha^{\intg}_2$ is $\max\{Q,D\}$-symmetric, and from
(\ref{total-measure-2}), (\ref{total-measure-1}), for $r$ large enough, we have
\begin{equation}\label{total-measure-3}
\alpha^{\intg}_2(\Tr(S)) \le nP(r)e^{-r}+n P(r)e^{-r} \le nP(r)e^{-r}.
\end{equation}
\vskip .1cm
Let $\gamma \in \Gamma(S)$ and consider the restriction of the measure $\ph \mu^{\intg}_1$ to $\NB\gamma$. This produces a pair of measures in $\Mes(\R)$. Apply Proposition 4.5 to this pair of measures. Choose an orientation $\gamma^{*}$ on $\gamma$. Recall the notation from Proposition 4.5.
We say that  $z \in A \subset \gamma$ if $(\gamma^{*},z)$ is a non-trivial atom  of the measure $\ph\mu^{\intg}_1$. We say 
that $z \in B \subset \gamma$ if $(-\gamma^{*},z)$ is a non-trivial atom  of the measure $\ph\mu^{\intg}_1$.
We identify $\Lab_A$ with $\Lab_{\Col,\gamma^{*}}$ and $\Lab_B$ with $\Lab_{\Col,-\gamma^{*}}$. We define the involution $\sigma_{\Col}: \Lab_{\Col,\gamma^{*}} \to \Lab_{\Col,-\gamma^{*}}$ by $\sigma_{\Col}=\sigma_{A,B}$ where $\sigma_{A,B}:A \to B$ is the bijection from Proposition 4.5.  This is how we construct the admissible pair $(\Col,\sigma_{\Col})$. 
\vskip .1cm
Fix $a \in \Lab_{\Col}$. Let $(T_1,\gamma^{*}_1), (T'_1,-\gamma^{*}_1) \in \Tr^{*}(G)$ such that $\lab_{\Col}(a)=[(T_1,\gamma^{*}_1)]$
and $\lab_{\Col}(\sigma_{\Col}(a))=[(T'_1,-\gamma^{*}_1)]$. Let $[T_1]_G=\Proj_S(\lab_{\Col}(a))=T \in \Tr(S)$ and  
$[T'_1]_G=\Proj_S(\lab_{\Col}(\sigma_{\Col}(a)))=T' \in \Tr(S)$. Also, let $\gamma^{*} \in \Gamma^{*}(S)$ be the projection of $\gamma^{*}_1$ to 
$S$ and finally let $\gamma \in \Gamma(S)$ be the corresponding unoriented geodesic.
We have $\dis(\ft_{\gamma}(\ct(T)),\ft_{\gamma}(\ct(T'))) \le \max\{ Q,D \}$ where $\lab_{\Col}(\sigma_{\Col}(a))=T'$. 
Let $N_{\Col}(re^{-r})$ denote the number of elements $a \in \Lab_{\Col}$ so that for the corresponding  edge $\gamma$ of $T=\Proj_S(\lab_{\Col}(a))$ we have $\dis(\ft_{\gamma}(\ct(T)),\ft_{\gamma}(\ct(T'))) > re^{-r}$.
From Proposition 4.5 and from  (\ref{total-measure-3}), we conclude  that $N_{\Col}(r e^{-r}) \le nP(r)e^{-r}$. 
\vskip .1cm
By Lemma 3.1, there exists finitely many virtual triangulation pairs  $(\Col_i,\sigma_i)$ so that $(\Col,\sigma)$ is their union. 
For each $i$ there exists a finite cover $S_i$ of $S$ so that $\Lab_{\Col_i}=\tau^{*}(S_i)$ is a triangulation of $S_i$ (by $\lambda(S_i)$ we denote the corresponding set of edges). Also, by $\shc(i) \in X(\tau(S_i))$ we denote the corresponding shear coordinates. Note that
the Riemann surface $S_i$ corresponds to the point $F_{\tau(S_i)}(\shc(i))$ in the corresponding Teichm\"uller space.
For each $S_i$ we have the following: 
\begin{itemize}
\item Since $\h(T)\le r^2$ for every $T \in \tau(S_i)$ from Lemma 3.2  we have \newline
$O_{\tau(S_{i})}(\shc(i)) \le 2r^2$.
\item We have $||\shc(i)||_{\infty}<\max \{Q,D \}$.
\end{itemize}
One can verify the following elementary proposition.
\begin{proposition} Let $x_i,y_i>0$, $i=1,...,k$, and set $x=x_1+...+x_k$, $y=y_1+...+y_k$. Then
$$
\min_{1 \le i \le k}{{x_i}\over{y_i}} \le {{x}\over{y}}.
$$
\end{proposition}

Let $N_{\tau(S_{i})}(re^{-r})$ denote the number  of edges from $\lambda(S_i)$ for which the corresponding shear coordinate is greater  
than  $re^{-r}$.  We have 
$$
\sum_{i}N_{\tau(S_{i})}(re^{-r})=N_{\Col}(re^{-r}) \le nP(r)e^{-r}.
$$
\noindent
Also, 
$$
\sum_{i}|\lambda(S_i)|=M \ge {{3n\mu(\Tr(S))}\over{2}}.
$$
\noindent
From the above proposition we have that  for at least one surface $S_i$, say for $S_1$, we have that 
$$
{{N_{\tau(S_{1})}(re^{-r})}\over{|\lambda(S_1)|}} \le P(r)e^{-r}.
$$
\noindent
where $|\lambda(S_1)|$ is the total number of edges in $\lambda(S_1)$. 

\vskip .1cm
Let $r \to \infty$. From Theorem 2.2 we have that the Weil-Petersson distance between $F_{\tau(S_i)}(\shc(i))$ and $F_{\tau(S_i)}(0)$, tends to $0$ when $r \to \infty$. Note that by Proposition 2.2, the Riemann surface that corresponds to $F_{\tau(S_i)}(0)$ is isomorphic to the quotient of $\Ha$ by a finite index subgroup of $\PSL$.
\vskip .1cm
Theorem 1.1 states that for any two punctured Riemann surfaces $S$ and $R$ of finite type, and for every $\epsilon>0$ we can find finite covers  
$S_{\epsilon}$ and $R_{\epsilon}$ of $S$ and $R$ respectively, so that the Weil-Petersson distance between them is less than $\epsilon$. 
We first find $S'_{\epsilon}$ and $R'_{\epsilon}$ finite covers of $S$ and $R$ respectively, so that  $S'_{\epsilon}$ and $R'_{\epsilon}$ are
${{\epsilon}\over{2}}$-close (in the Weil-Petersson sense) to two Riemann surfaces $S''_{\epsilon}$ and $R''_{\epsilon}$ where 
$S''_{\epsilon}$ and $R''_{\epsilon}$ are isomorphic to $\Ha/G_1$ and $\Ha/G_2$ respectively, where $G_1,G_2$ are finite index subgroups of $\PSL$.
Set $G_3=G_1 \cap G_2$ and let $M_{\epsilon}$ be the corresponding Riemann surface. Then there are covers $S_{\epsilon}$ and $R_{\epsilon}$ of $S$ and $R$ respectively, so that $S_{\epsilon}$ and $R_{\epsilon}$ are ${{\epsilon}\over{2}}$-close (in the Weil-Petersson sense) to $M$. Moreover,
the Weil-Petersson distance between $S_{\epsilon}$ and $R_{\epsilon}$ is at most $\epsilon$. This proves the theorem.

\section{The absorption maps and the proof of Theorem 4.1}

\subsection{Preliminary results from hyperbolic geometry and four operations on the unit tangent bundle}
Let $\Mob$ denote the group of orientation preserving  M\"obius transformations of $\Ha$. By $\Mobi$ we denote the subgroup of $\Mob$ whose elements preserve  $\infty$. By $\Gamma(\Ha)$ we denote the set of all geodesics in $\Ha$. Also, let $\Tr(\Ha)$ denote the set of all ideal triangles in $\Ha$.
\vskip .1cm
Denote by $\TB\Ha$ the unit tangent bundle of $\Ha$. Elements of $\TB\Ha$ are pairs $(z,v)$ where $z \in \Ha$ and $v$ is a unit vector at $z$.
If $f \in \Mob$ then $f(z,v) \in \TB\Ha$ is a well defined element. The quotient $\TB\Ha/G$ is isomorphic to the unit tangent bundle $\TB{S}$.
Also, if $E \subset \Ha$ or $E \subset S$ then $\TB{E}$ denotes the restriction of the corresponding unit tangent bundle over $E$.
\vskip .1cm
We parametrise $\TB\Ha$ by $(z,v)=(x,y,\theta)$. Here  $z=x+iy$, and $\theta \in [0,2\pi)$ is the positively oriented angle that $v$ makes with the positive part of the $x$-axis. The Liouville volume form on $\TB\Ha$ is given by $d\LM=y^{-2}\, dx \wedge dy \wedge d\theta$. The corresponding measure on $\TB\Ha$ is called the Liouville measure. The total measure $\LM(\TB{S})$ is finite.
\vskip .1cm
Let $\omega=e^{2\pi i/3}$. By $\omega:\TB\Ha \to \TB\Ha$ we also denote the map $\omega(z,v)=(z,\omega v)$. Clearly the map $\omega$ is a diffeomorphism
of $\TB\Ha$ of order three (this is the first operation on $\TB\Ha$ we define in this subsection). Also, the map $\omega$ commutes with every element from $\Mob$ (and in particular $\omega$ commutes with every element of the group $G$). We have that $\omega:\TB{S} \to \TB{S}$ is a well defined diffeomorphism. Moreover,  $\omega$ is measure preserving (it preserves the Liouville measure on $\TB\Ha$), that is $\Jac(\omega)=1$ where $\Jac(\omega)$ denotes  the Jacobian. 
\vskip .1cm
Fix $(z,v) \in \TB\Ha$. Let $\gamma_{(z,v)}:[0,\infty) \to \Ha$ be the natural parametrisation of the geodesic ray that starts at $z$ and that is tangent to the vector $v$ at $z$ that is $\gamma'_{(z,v)}(0)=v$. We use the same notation for the induced map $\gamma_{(z,v)}:\R^{+}\cup\{0\} \to S$.
By $\gamma_{(z,v)}$ we also denote the corresponding geodesic ray $\gamma_{(z,v)}(\R^{+}\cup \{0\})$. 
\vskip .1cm
As usual, for $t \in \R^{+} \cup \{0\}$ the geodesic flow $\flow_t:\TB\Ha \to \TB\Ha$ is given by $\flow_t(z,v)=(\gamma_{(z,v)}(t),\gamma'_{(z,v)}(t))$. 
Moreover, for $f \in \Mob$ we have $\flow_t \circ f=f \circ \flow_t$. Therefore, the flow is well defined on $\TB\Ha/G$. We use the same notation for the induced flow $\flow_t:\TB{S} \to \TB{S}$. The geodesic flow $\flow_t$ is  measure preserving.

\begin{definition} Let $\p:\TB\Ha \to \partial{\Ha}$ denote the map such that $\p(z,v) \in \partial{\Ha}$ is the end point of the geodesic ray $\gamma_{(z,v)}$. Let $\p^{1}:\TB\Ha \to \Gamma(\Ha)$ denote the map such that  $\p^{1}(z,v) \in \Gamma(\Ha)$ is the geodesic that connects the points $\p(z,v)$ and  $\p(z,\omega v)$. Let $\p^{2}:\TB\Ha \to \Tr(\Ha)$ denote the map such that  $\p^{2}(z,v) \in \Tr(\Ha)$ is the triangle with the vertices  $\p(z,v)$ $\p(\omega(\z,v))$ and  $\p(\omega^{2}(\z,v))$.
\end{definition}

Since the angle between $v$ and $\omega v$ is $2\pi/3$ we find that the point $z$ is the centre of the triangle $\p^{2}(z,v)=\p^{2}(\omega(z,v))=\p^{2}(\omega^{2}(z,v))$. 
The triangle $\p^{2}(z,v)$  is bounded by the geodesics $\p^{1}(z,v)$, $\p^{1}(\omega(z,v))$ and  $\p^{1}(\omega^{2}(z,v))$. Moreover, we endow the geodesic $p^{1}(z,v)$ with the induced orientation, so that the endpoints $p(z,v)$ and $\p(z,\omega v)$ correspond to $-\infty$ and $+\infty$, respectively.
\vskip .1cm
We now define the other three operations on $\TB\Ha$. Let $(z,v) \in \TB\Ha$. Let $f_{(z,v)} \in \Mob$ be the rotation of order two about the point $\ft_{\p^{1}(z,v)}(z) \in \p^{1}(z,v)$. Set $\refl(z,v)=f(z,v)$. Clearly the map $\refl:\TB\Ha \to \TB\Ha$ is a diffeomorphism of order two.
Also the map $\refl$ commutes with every element from $\Mob$ and the induced diffeomorphism $\refl:\TB{S} \to \TB{S}$ is well defined. Note that $\p(z,v)=\p(\omega(\refl(z,v)))$ and $\p^{1}(z,v)=\p^{1}(\refl(z,v))$.

\begin{remark} The group $\left< \omega,\refl, \right>$ generated by $\omega,\refl:\TB{S} \to \TB{S}$ is isomorphic to $\Z_2 \star \Z_3$.
It is easy to see that there exists  a finite orbit under the action of this group if and only if the surface $S$ is modular, that is $S$ is isomorphic to
$\Ha/G$ where $G$ is a finite index subgroup of $\PSL$. This indicates the relevance of this group action to the Ehrenpreis conjecture.
\end{remark}

\begin{figure}
  \centering
  {
    \def\I{{\mathcal I}}
    \def\R{{\mathcal R}}
    \def\S{{\mathcal S}}
    \def\o{\omega}
    \def\p{{\mathbf p}}
    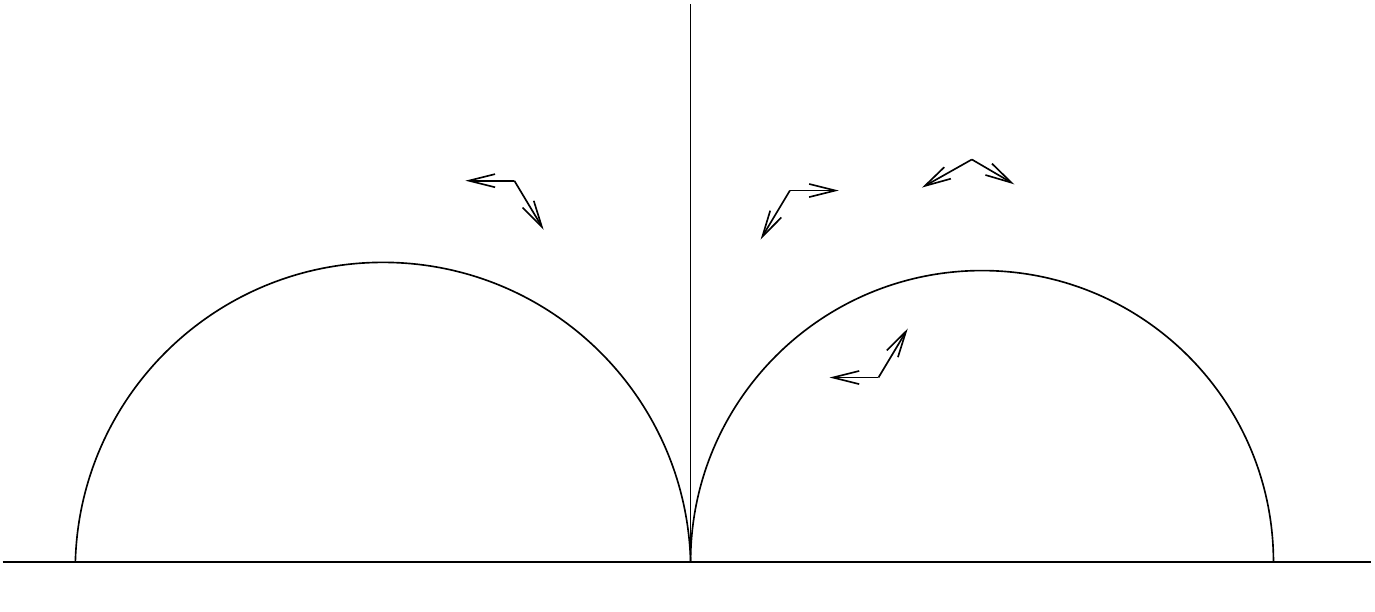
  }
\caption{The four operations on the $\TB\Ha$}
\label{operations}
  \end{figure}

Let $(z,v) \in \TB\Ha$, $w \in \p^{1}(z,v)$, and $t \in \R$. Let $f_{(z,v)}(t) \in \Mob$ be the hyperbolic transformation which fixes the points $\p(z,v)$ and $\p(z,\omega v)$, and so that the signed hyperbolic distance between the points $w$ and $f_{(z,v)}(t)(w)$, is equal to $t$. Here the signed distance between  $w$ and $f_{(z,v)}(t)(w)$ is positive if and only if the points $\p(z,v)$, $w$, $f_{(z,v)}(t)(w)$, and  $\p(z,\omega v)$, sit on the geodesic $\p^{1}(z,v)$, in this order, with respect to the induced orientation on $\p^{1}(z,v)$. The definition of the map  $f_{(z,v)}(t)$ does not depend on the choice of $w \in \p^{1}(z,v)$.  Clearly, the collection of transformations $f_{(z,v)}(t)$, $t \in \R$, is an one-parameter Abelian group.  Set $\tra_t(z,v)= f_{(z,v)}(t)(z,v)$. The collections of maps $\tra_t:\TB\Ha \to \TB\Ha$, $t \in \R$, is an one-parameter Abelian group of diffeomorphisms, and  $\tra_t$ commutes with every element from $\Mob$.
Note that $\p(z,v)=\p(\tra_t(z,v))$, and $\p^{1}(z,v)=\p^{1}(\tra_t(z,v))$.
\begin{remark} In fact, $\tra_t:\TB\Ha \to \TB\Ha$, is the ''equidistant" flow. This means that for $(z_t,v_t)=\tra_t(z,v)$, the points $z_t$ move  along the  line that is equidistant from the geodesic $\p^{1}(z,v)$ (the distance between $z_t$ and  $\p^{1}(z,v)$ is $\log \sqrt{3}$).
\end{remark}
\vskip .1cm 
Let $(z,v) \in \TB\Ha$ so that  $\p(z,v) \ne \infty$. 
Let $\gamma \in \Gamma(\Ha)$, be the geodesic that connects $\p(z,v)$ and $\infty$, and let $f_{\gamma}:\Ha \to \Ha$, be the reflection through $\gamma$. 
Set $\inv^{L}_1(z,v)=f_{\gamma}(z,v)$, and  $\inv^{L}=\omega^{2} \circ \inv^{L}_1$. The maps $\inv^{L}_1, \inv^{L}:\TB\Ha \to  \TB\Ha$ are defined almost everywhere on $\TB\Ha$, because the subset of $\TB\Ha$ on which  $\p(z,v) = \infty$ has zero measure in $\TB\Ha$. The maps $\inv^{L}$ and  $\inv^{L}_1$ commute with every element of $\Mobi$. 
Note that $\inv^{L}_1$ is of order two. Also, $\p(z,v)=\p(\omega(\inv^{L}(z,v)))$. We set $\inv^{R}=(\inv^{L})^{-1}$.

\begin{remark} The map $\inv^{L}:\TB\Ha \to \TB\Ha$  does not commute with the entire group $\Mob$ and the map $\inv^{L}$ does not give 
rise to a map on $\TB{S}$. Let $c \in \Cus(G)$, and   let $G=G_c$ be normalised. Then  $\inv^{L}$  commutes with the translation for $1$, so the maps
$\inv^{L}$ and $\inv^{L}_1$, are well defined almost everywhere  on $\TB\Hor_c(0)$. We have the induced maps $\inv^{L}, \inv^{L}_1:\TB\Ha \setminus \TB\Thick_G(0) \to \TB\Ha \setminus \TB\Thick_G(0)$, and $\inv^{L}, \inv^{L}_1:\TB{S} \setminus \TB\Thick_S(0) \to \TB{S} \setminus \TB\Thick_S(0)$, that are defined on each $0$-horoball in the normalised setting. Also, note that $\inv^{L} \ne \inv^{R}$. 
\end{remark}

\begin{proposition} We have $\Jac(\refl)=\Jac(\tra_t)=1$ everywhere on $\TB\Ha$ and $\Jac(\inv^{L})=1$ almost everywhere on $\TB\Ha$. Moreover we have
that the relations $\refl \circ \tra_t=\tra_{(-t)} \circ \refl$  and $\inv^{L} \circ \tra_t=\tra_{(-t)} \circ \inv^{L}$ 
hold for every $t \in \R$. Also $\refl^{2}$ is the identity mapping on $\TB\Ha$.
\end{proposition}

\begin{proof} We have already observed that $\refl$ is of order two, that is $\refl^{2}$ is the identity mapping on $\TB\Ha$. 
If $(z,v),(z_1,v_1) \in \TB\Ha$, and $(z_1,v_1)=\tra_t(z,v)$, for some $t \in \R$, then $\p^{1}(z,v)=\p^{1}(z_1,v_1)$, and $\p^{1}(\inv^{L}(z,v))=\p^{1}(\inv^{L}(z_1,v_1))$. This yields the relation $\inv^{L} \circ \tra_t=\tra_{(-t)} \circ \inv^{L}$. The relation $\refl \circ \tra_t=\tra_{(-t)} \circ \refl$, is proved similarly. 
\vskip .1cm
Since all four maps $\refl$, $\tra_t$, $\inv^{L}_1$, and $\inv^{L}$ commute with $\Mobi$, and  
$\Mobi$ acts transitively on $\Ha$, we have that the functions $\Jac(\refl)$, $\Jac(\tra_t)$, $\Jac(\inv^{L}_1)$, and $\Jac(\inv^{L})$, are constant functions almost everywhere on $\TB\Ha$. Since $\refl=\refl^{-1}$ we have that $\Jac(\refl)=1$ everywhere on $\TB\Ha$. It follows from $\refl \circ \tra_t=\tra_{(-t)} \circ \refl$, that $\Jac(\tra_t)=\Jac(\tra_{(-t)})$. Since $(\tra_t)^{-1}=\tra_{(-t)}$, we conclude that $\Jac(\tra_t)=1$ everywhere on $\TB\Ha$. From $(\inv^{L}_1)^{-1}=\inv^{L}_1$, we find that $\Jac(\inv^{L}_1)=1$. Since $\Jac(\omega)=1$ it follows from the definition of $\inv^{L}$  that $\Jac(\inv^{L})=1$.
\end{proof}

\begin{definition} Let $(z,v) \in \TB\Ha$, and let $p \in \partial{\Ha}$. 
\begin{itemize} 
\item  Denote by  $\ang_p(z,v)$  the unique number in $[-\pi,\pi)$ that is equal to the positively oriented angle between the vector $v$ and the geodesic ray $\gamma^{p}_{z}$, where $\gamma^{p}_{z}$  denotes  the geodesic ray that starts at $z$ and  ends at $p$. 
\item By $\Hor^{p}_z$ we denote the unique horoball that contains $z$ and  meets $\partial{\Ha}$ at $p$. 
\item Let $z,z' \in \Ha$. Denote by $\Delta_p[z,z']$ the signed hyperbolic distance between the horocircles $\partial{\Hor^{p}_{z}}$ and $\partial{\Hor^{p}_{z'}}$. That is, $\Delta_p[z,z']$ is non-negative if and only if $\Hor^{p}_{z}$ is contained in $\Hor^{p}_{z'}$ (in this case we say that $z$ is closer to $p$ than $z'$ is). 
\item Suppose that $p \ne \p(z,v)$. Then by $\z_{\max}(\gamma_{(z,v)},p)$ we denote the point on $\gamma_{(z,v)}$ that is the closest to $p$. 

\end{itemize}
\end{definition}

Also, if $p \ne \p(z,v)$, then there exists a unique point on $\gamma_{(z,v)}$ that is the closest to $p$. This shows that  $\z_{\max}(\gamma_{(z,v)},p)$ is well defined. If $p=\p(z,v)$, then such a point does not exist.
\vskip .1cm
Note that for $z,w \in \Ha$, we have $\Delta_p[z,w]=-\Delta_p[w,z]$, and $|\Delta_p[z,w]| \le \dis(z,w)$. For points $z,w_1,w_2 \in \Ha$, we have  $\Delta_p[z,w_2]-\Delta_p[z,w_1]=\Delta_p[w_1,w_2]$.

\begin{remark} Let $c \in \Cus(G)$, and $z,z' \in \Ha$. Then $\Delta_c[z,z']$ measures the difference in heights, that is
$$
\h_c(z)-\h_c(z')=\Delta_c[z,z'].
$$
\end{remark}

It follows from the definition that $\z_{\max}(\gamma_{(z,v)},p) \ne z$, if and only if $0<|\ang_p(z,v)|< {{\pi}\over{2}}$. If $\ang_p(z,v)=0$, then $\gamma^{p}_{z}=\gamma_{(z,v)}$. If $0<|\ang_p(z,v)|< {{\pi}\over{2}}$, then there exists $t_0>0$, so that $\gamma_{(z,v)}(t_0)=\z_{\max}(\gamma_{(z,v)},p)$. Most calculations in the remainder of this paper are based on the following two 
elementary identities

\begin{equation}\label{ang-estimate-1}
\Delta_p[\z_{\max}(\gamma_{(z,v)},p),z]=\log(\csc(|\ang_p(z,v)|)),
\end{equation}
\noindent
and
\begin{equation}\label{ang-estimate-2}
t_0=\log \big( \csc(|\ang_p(z,v)|)+\cot(|\ang_p(z,v)|) \big).
\end{equation}
\noindent
It follows that

\begin{equation}\label{ang-estimate-2'}
0<t_0-\Delta_p[\z_{\max}(\gamma_{(z,v)},p),z]<\log 2,
\end{equation}
\noindent
and when  $|\ang_p(z,v)|$ is small, we have 
\begin{equation}\label{ang-estimate-3}
0<t_0-\Delta_p[\z_{\max}(\gamma_{(z,v)},p),z]=\log 2- O(|\ang_p(z,v)|^{2}).
\end{equation}
\noindent

\begin{proposition} Let  $(z,v) \in \TB\Ha$, and $p \in \partial{\Ha}$,  such that  $0<|\ang_p(z,v)|< {{\pi}\over{2}}$. Let
$t_0>0$ be the number so that  $\z_{\max}(\gamma_{(z,v)},p)=\gamma_{(z,v)}(t_0)$. Then for every $0 \le t \le t_0$, we have 

$$
t-\log 2<\Delta_p[\gamma_{(z,v)}(t)),z]< t.
$$
\noindent
Also,
$$
e^{-t_{0}} <  |\ang_p(z,v)|<\pi e^{-t_{0}}.
$$

\end{proposition}

\begin{proof} The first inequality  follows from  (\ref{ang-estimate-2'}). The second inequality follows from (\ref{ang-estimate-1}), and the fact that for $|\theta| \le {{\pi}\over{2}}$, we have $|\sin \theta|\le |\theta| \le {{\pi}\over{2}}|\sin \theta|$.
\end{proof}

\begin{proposition} Let $t \in \R$, $p \in \partial{\Ha}$, and $(z,v) \in \TB\Ha$, such that $\ang_p(z,\omega^{2}v)=0$. Let $(z_t,v_t)=\tra_t(z,v)$. Then
$$
|\Delta_p[z,z_t]| \ge |t|-\log 6.
$$

\end{proposition} 

\begin{proof} Let $w=\ft_{\p^{1}(z,v)}(z)$, and $w_t=\ft_{\p^{1}(z,v)}(z_t)$. Since $\ang_p(z,\omega^{2}v)=0$, we have $w=\z_{\max}(\p^{1}(z,v),p)$.
It follows from the definition of $\tra_t(z,v)$, that the hyperbolic distance between the points $w$ and $w_t$, is equal to $|t|$. From the first inequality in the previous proposition we have
$$
|\Delta_p[w,w_t]| \ge |t|-\log 2.
$$
\noindent
Since $\dis(z,w)=\dis(z_t,w_t)=\log \sqrt{3}$, we have
$$
|\Delta_p[z,z_t]| \ge |\Delta_p[w,w_t]|- \log 3 \ge |t|-\log 6.
$$
\noindent
If $p$ is an endpoint of $\p^{1}(z,v)$, then $|\Delta_p[z,z_t]|=|t|>|t|-\log 6$.
\end{proof}

\begin{proposition} Let $(z,v) \in \TB\Ha$, and let $p_i \in \partial{\Ha}$, $i=0,1,2$, so that $|\ang_{p_{i}}(z,\omega^{i} v)|<{{\pi}\over{2}}\delta$, for some $0 \le \delta$. Let $T \in \Tr(\Ha)$ be the triangle with the vertices $p_i$. 
There exists a universal constant $C>0$, so that for $\delta$ small enough, we have $\dis(\ct(T),z) \le C\delta$.
\end{proposition}

\begin{proof} We consider the unit disc model $\D$ for $\Ha$. We may assume that $z=0$, and that $v={{ \partial}\over{\partial{x}}}(0)$. Then 
$\p(z, \omega^{i} v)=\omega^{i}=e^{2\pi i/3} \in \partial{\D}$, $i=0,1,2$, where $\partial{\D}$ is the unit circle. Since 
$|\ang_{p_{i}}(z,\omega^{i} v)|<{{\pi}\over{2}} \delta$, we have that $|p_i-\p(z,\omega^{i} v)| < {{\delta}\over{4}}$, where  $|p_i-\p(z,\omega^{i} v)|$ is the Euclidean distance.  Recall that $z$ is the centre of the triangle $\p^{2}(z,v)$. The centre of an ideal triangle, as the function of triples of points on $\partial{\D}$, is smooth in some   neighbourhood the triple $(1,e^{2\pi/3},e^{4\pi /3})$. Therefore, there exists a universal constant $C>0$, so that $\dis(\ct(T),z) \le C \delta$, for $\delta$ small.

\end{proof}

\begin{proposition} Let $(z,v) \in \TB\Ha$, and set $\refl(z,v)=(z_1,v_1)$. 
Let $p_i \in \partial{\Ha}$, $i=0,1$, so that $|\ang_{p_{i}}(z,\omega^{i} v)|<{{\pi}\over{2}}\delta$, for some $0 \le \delta$. By $\gamma \in \Gamma(\Ha)$, we denote the geodesic with the endpoints $p_0$ and $p_1$. There exists a universal constant $C>0$, so that 
for $\delta$ small enough we have $\dis(\ft_{\p^{1}(z,v) }(z),\ft_{\gamma}(z)) \le C\delta$,  $\dis(\ft_{\p^{1}(z,v)}(z_1),\ft_{\gamma}(z_1)) \le C\delta$, and $\dis(\ft_{\gamma}(z),\ft_{\gamma}(z_1)) \le C\delta$.

\end{proposition}

\begin{proof}  We consider the unit disc model $\D$ for $\Ha$. We may assume that the geodesic $\p^{1}(z,v)$ connects the points $-1$ and $1$, on $\partial{\D}$, and that $\p(z,v)=-1$. Also, we may assume that the $x$-coordinate of $z$ is equal to zero. Then 
$z=i{{\sqrt{3}-1}\over{\sqrt {3}+1}}$, and $z_1=-i{{\sqrt{3}-1}\over{\sqrt{3}+1}}$.
Since $|\ang_{p_{i}}(z,\omega^{i} v)|<{{\pi}\over{2}}\delta$, $i=0,1$,  we have that for some universal constant $D>0$, and for $\delta$ small enough, the inequalities $|p_0+1| < D\delta$, and $|p_1-1| < D\delta$, hold. Here   $|p_0+1|$, and $|p_1-1| $, are the Euclidean distances. The function
$\ft_{\gamma}(z)$, as a function of $p_0$, and $p_1$, is smooth in some neighbourhood of the pair $(1,-1)$. This shows that there is a universal constant $C>0$, so that for $\delta$ small enough, we have  $\dis(\ft_{\p^{1}(z,v) }(z),\ft_{\gamma}(z)),\dis(\ft_{\p^{1}(z,v)}(z_1),\ft_{\gamma}(z_1)), \dis(\ft_{\gamma}(z),\ft_{\gamma}(z_1)) \le C\delta$.
\end{proof}

\subsection{The Absorption maps}

First we give a short overview of the construction that follows. In order to construct the measures from the statement of Theorem 4.1, for every $r>2$ we  define a map (almost everywhere on $\TB\Ha$ with respect to the Liouville measure $\LM$ on $\TB\Ha$)
$$
\abs_r:\TB\Ha \to \Cus(G),
$$
\noindent
such that for almost every $(z,v) \in \TB\Ha$ we have 
\begin{itemize}
\item $\abs_r(g(z,v))=g(\abs_r(z,v))$, for every $g \in G$.
\item $\ang_{\abs_{r}(z,v)}(z,v) \le Ce^{-r}$.
\end{itemize}

Moreover we have the induced map $\abs^{1}_r:\TB\Ha \to \Gamma(G)$ where $\abs^{1}_r(z,v)$ is the geodesic with the endpoints $\abs_r(z,v)$ and $\abs_r(z,\omega v)$. Also we have the map  $\abs^{2}_r:\TB\Ha \to \Tr(G)$ where $\abs^{2}_r(z,v)$ is the triangle with vertices $\abs_r(z,\omega^{j} v$, $j=0,1,2$. The maps $\abs_r$, $\abs^{1}_r$ and $\abs^{2}_r$ are $G$ equivariant so we have the induced maps from $\TB{S}$ to  $\Cus(G)$, $\Gamma(S)$ and $\Tr(S)$, respectively. Note that if $\varphi:S \to \R$ is a non-negative integrable function on $S$, then $(\abs^{2}_r)_{*}(\varphi \, d\LM) \in \Mes(\Tr(S))$. The measure $\mu \in \Mes(\Tr(S))$ from Theorem 4.1 will be constructed in this way for a suitable choice of $\varphi$.
\vskip .1cm
Our goal is to let $r$ be large and show that for ''most" points $(z,v) \in \TB\Ha$ we have $\abs^{1}_r(z,v)=\abs^{1}_r(\refl(z,v))$. This would imply that $\ph (\abs^{2}_r)_{*}(d\LM)=\alpha+\beta$, where $\alpha$ is $Ce^{-r}$ symmetric and $\beta$ is small. This seems to be a good candidate for the choice of measure $\mu$  in Theorem 4.1. However with this choice we have
$$
\int\limits_{\NB\Gamma(S)} (\cl(\gamma^{*},z)+1) d \ph (\abs^{2}_r)_{*}(d\LM)=\infty,
$$
\noindent
because of the thin part of $S$. In order to overcome this problem we set $\varphi_r(z)=e^{ {{1}\over{2}}(2r-\h(z))}$, and consider the measure
$(\abs^{2}_r)_{*}(\varphi_r \, d\LM) \in \Mes(\Tr(S))$ (note that $\varphi_r(z)=1$ for $z \in \Thick_S(2r)$). But this introduces an imbalance, that is the measure $\ph (\abs^{2}_r)_{*}(\varphi_r \, d\LM)$ is no longer almost symmetric outside $\Thick_S(2r)$ because $\varphi_r(z,v) \ne \varphi_r(\refl(z,v))$ for $z$ that is is outside $\Thick_S(2r)$.  That is why we introduce the  map 
$$
\taso^{2}_r:\Col_r(S) \to \Tr(S), 
$$
\noindent
where $\Col_r(S) \subset (\TB{S} \setminus \Thick_S(2r))$ is the ''correctable" part. Then we let our measure $\mu \in \Mes(\Tr(S))$ be given by
$$
\mu= (\abs^{2}_r)_{*}(\varphi_r \, d\LM)+3(\taso^{2}_r)_{*}(\vartheta_r \, d\LM),
$$
\noindent
where $\vartheta_r(z,v)=\varphi_r(\refl(z,v))-\varphi_r(z,v)$, for $(z,v) \in \Col_r(S)$.  Then we show that the measure $\ph \mu$ has the desired decomposition. 
\vskip .3cm
We now return to the construction. Recall that $S$ is a fixed surface of type $(\gen,\numb)$, and $\Cus(S)=\{c_1(S),...,c_n(S)\}$.
As always, $G$ is one of the $n$ normalised Fuchsian  groups $G_{c_{i}(S)}$, $i=1,...,n$, such that $\Ha/G$ is isomorphic to $S$.
Recall that the interiors of the different $0$-horoballs on $S$ are disjoint. This implies that different $1$-horoball are disjoint on $S$
\vskip .1cm
Let $(z,v) \in \TB\Ha$.  For any $t \ge 0$, the point
$\gamma_{(z,v)}(t)$ is either in the interior of  $\Thick_G(1)$, or in one of the $1$-horoballs. Let $(c_1,c_2,...)$, be the ordered set of cusps from $\Cus(G)$, so that the ray $\gamma_{(z,v)}$ intersects  each horoball $\Hor_{c_{i}}(1)$, and which are ordered so that the ray $\gamma_{(z,v)}$ visits the horoball $\Hor_{c_{i}}(1)$ before it visits the horoball $\Hor_{c_{j}}(1)$ if and only if $i<j$. In other words, let $t_i \ge 0$, so that $\gamma_{(z,v)}(t_i)=\z_{\max}(\gamma_{(z,v)},c_i)$. Then $i <j$, if and only if $t_i<t_j$.
Since the $1$-horoballs are disjoint, and each horoball is a convex subset of $\Ha$, we have that $c_i \ne c_j$, for $i \ne j$. 

\begin{remark} For a point $(z,v) \in \TB\Ha$, the set $(c_1,c_2...)$ is either: empty, finite but non-empty, or infinite. It can be show that the set of points in $\TB\Ha$, for which the corresponding set $(c_1,c_2...)$ is infinite, has the full measure in $\TB\Ha$ (we do not use this result).
\end{remark} 

We now define the maps $\abs_r$, and  $\abs^{i}_r$, $i=1,2$, described above. Given $r>0$, to every point  $(z,v) \in \TB{S}$, we associate a triangle from $\Tr(S)$.

\begin{definition} Fix $r>2$. Let $(z,v) \in \TB\Ha$, and let $(c_1,c_2,...)$, be the corresponding ordered set of cusps that $\gamma_{(z,v)}$ intersects.
\begin{itemize}
\item Define $\abs_r(z,v)=c_i$, if
 
\begin{equation}\label{h-difference}
\h_{c_{i}}(\z_{\max}(\gamma_{(z,v)},c_i))-\h_{c_{i}}(z) \ge r,
\end{equation}
and if this inequality does not hold for any cusp $c_j$, where $j<i$.

\item Let $0 \le \at_r(z,v) \le \infty$, be such that $\gamma_{(z,v)}(\at_r(z,v))$ is the first point of entry of the ray $\gamma_{(z,v)}$
in the horoball $\Hor_{\abs_{r}(z,v)}(1)$. If $(z,v)$ does not get absorbed, then $\at_r(z,v)=\infty$. 

\item If both $\abs_r(z,v)$ and $\abs_r(z,\omega v)$ are well defined, then  by $\abs^{1}_r(z,v) \in \Gamma(G)$, we denote the geodesic with the endpoints $\abs_r(z,v)$ and $\abs_r(z,\omega v)$. If all three cusps $\abs_r(z,v)$, $\abs_r(z,\omega v)$, and  $\abs_r(z,\omega^{2} v)$, are well defined, then by $\abs^{2}_r(z,v) \in \Tr(G)$, we denote the triangle with the vertices $\abs_r(z,v)$,  $\abs_r(z,\omega v)$,
and $\abs_r(z,\omega^{2} v)$. 

\end{itemize}

\end{definition}

Since $r>2$, we have that $|\ang_{\abs_r(z,v)}(z,v)| \le \pi e^{-2} < {{\pi}\over{6}}$. This implies that  $|\ang_{\abs_r(z,\omega v)}(z,v)|>{{\pi}\over{2}}$, so $\abs_r(z,v) \ne \abs_r(z,\omega v) \ne \abs_r(z,\omega^{2} v)\ne \abs_r(z,v)$. This shows that $\abs^{1}_r(z,v)$ and $\abs^{2}_r(z,v)$ are well defined. 

\begin{remark} It is useful to pause here and observe that if $c \in \Cus(G)$, is such that $\abs_r(z,v)=c$, then $|\ang_c(z,v)| \le \pi e^{-r}$. This follows from Proposition 5.2. We will often use this observation in the arguments that follow.
\end{remark}

We will see below that the function $\abs_r:\TB\Ha \to \Cus(G)$ is defined almost everywhere on $\TB\Ha$. Let $g \in G$. We have $g(\abs_r(z,v))=\abs_r(g(z,v))$. This shows that the set where $\abs_r$ is defined is invariant under $G$. 
The induced  map $\abs_r:\TB{S} \to \Cus(S)$, is also denoted by $\abs_r$. The absorption time $\at_r(z,v)$ is defined in the same way.
The maps $\abs^{1}_r:\TB{S} \to \Gamma(S)$, and $\abs^{2}_r:\TB{S} \to \Tr(S)$, are defined accordingly.
\vskip .1cm

\begin{proposition} Fix $r>2$, and $(z,v) \in \TB\Ha$. Then for every
$3r+ \max\{0,\h(z)\} \le t <\at_r(z,v)$, we have $\gamma_{(z,v)}(t) \in \Thick_G(1)$. Moreover, suppose that $z \in \Hor_{c_{1}}(r+2)$, for some $c_1 \in \Cus(G)$, and that $\abs_r(z,v) \ne c_1$. If $c_2  \in \Cus(G)$ is the the second cusp (after $c_1$) so that $\gamma_{(z,v)}$ visits the horoball $\Hor_{c_{2}}(1)$, then $\abs_r(z,v)=c_2$. 
\end{proposition}
\begin{proof} For each cusp $c_i \in \Cus(G)$, such that $\gamma_{(z,v)}$ enters this cusp, we let $t_{\ent}(i) \ge 0$, so that $\gamma_{(z,v)}(t_{\ent}(i))$ is the point of entry of the ray $\gamma_{(z,v)}$ in $\Hor_{c_{i}}(1)$. By $t_{\ex}(i) >0$, we denote the exit time, that is $\gamma_{(z,v)}(t_{\ex}(i))$ is the exit point  of the ray $\gamma_{(z,v)}$ from $\Hor_{c_{i}}(1)$. If $z \in \Hor_{c_{1}}(1)$, then $t_{\ent}(1)=0$. If $\gamma_{(z,v)}$ never leaves some $\Hor_{c_{i}}(1)$, then $t_{\ex}(i)=\infty$.
\vskip .1cm
It follows from Proposition 5.2 that if $t_{\ent}(i)>0$, then 
\begin{equation}\label{diff-1}
\Delta_{c_{i}}[\z_{\max}(\gamma_{(z,v)},c_i),z] \ge {{t_{\ex}(i)+t_{\ent}(i)}\over{2}}-\log 2.
\end{equation}
\noindent
If $z \in \Hor_{c_{1}}(1)$, then 
\begin{equation}\label{diff-2}
\Delta_{c_{i}}[\z_{\max}(\gamma_{(z,v)},c_i),z] \ge {{t_{\ex}(i)+\h_{c_{i}}(z)}\over{2}}-\log 2.
\end{equation}
\noindent
If $t<\at_r(z,v)$, and $\gamma_{(z,v)}(t) \in \Hor_{c_{i}}(1)$, then either $t_{\ent}(i)>0$, or $i=1$ and $z \in \Hor_{c_{1}}(1)$. In either case, we have $\at_r(z,v)>t_{\ex}(i)$, since $\abs_r(z,v) \ne c_i$ (the identity  $\abs_r(z,v)= c_i$ would contradict the assumption $t<\at_r(z,v)$).
In the first case, from (\ref{diff-1}) we obtain   $t \le t_{\ex} \le 2r+2\log 2<3r$. In the second case, from (\ref{diff-2}) we have $t \le t_{\ex}(1) \le 2r+\h_{c_{1}}(z)+2\log 2<3r+\h_{c_{1}}(z)$. This proves the first part of the proposition.
\vskip .1cm
Assume that $\h_{c_{1}}(z) \ge r+2$. Let $c_2 \in \Cus(G)$, be the cusp so that the horoball $\Hor_{c_{2}}(1)$ is the second $1$-horoball that the ray $\gamma_{(z,v)}$ enters. By (\ref{diff-1}) we have  
$$
\h_{c_{2}}(\z_{\max}(\gamma_{(z,v)},c_2))-\h_{c_{2}}(z) \ge t_{\ent}(2)-\log 2 \ge \h_{c_{1}}(z)-1-\log. 
$$
\noindent
Since $\h_{c_{1}}(z) \ge r+2$, and since $\abs_r(z,v) \ne c_1$, we conclude $\abs_r(z,v)=c_2$. 

\end{proof}

We want to show that the absorption map is well defined almost everywhere on $\TB\Ha$ (and on $\TB{S}$). We saw in the previous proposition that the ray $\gamma_{(z,v)}$, will get absorbed if it leaves $\Thick_S(1)$ after the time $3r+|\h(z)|$. In order to show that $\abs_r$ is defined almost everywhere, we need to estimate the Liouville measure of the geodesics segments of a given length that stay in $\Thick_S(1)$. The following lemma is probably known, but we prove it in the appendix.

\begin{lemma} Let $t>0$, and let $A(t) \subset \TB\Thick_S(1)$ be such that $(z,v) \in A(t)$  if the segment $\gamma_{(z,v)}[0,t]$ is contained in $\Thick_S(1)$. Then there are  constants $C(S),q(S)>0$, that depend only on $S$, such that $\LM(A(t))<C(S)e^{-q(S) t}$.
\end{lemma}

We have

\begin{proposition} Fix $r>2$. The map $\abs_r:\TB{S} \to \Cus(S)$, is defined almost everywhere. 
\end{proposition}

\begin{remark} The proof below uses Lemma 5.1. However one does not need Lemma 5.1 to prove that the absorptions map is defined almost everywhere on $\TB\Ha$. The fact that the geodesic flow $\flow_t$ is ergodic implies that the geodesic ray $\gamma_{(z,v)}$ leaves the thick part $\Thick_S(1)$, after the time $3r+|\h(z)|$, for almost every $(z,v) \in \TB{S}$. Then it follows from Proposition 5.6 that almost every $(z,v)$ will be absorbed. However  Lemma 5.1 will be used later in a similar manner, so we decide to state it here and present its first application.
\end{remark}

\begin{proof}
Let $s>0$, and let $F_{r,s} \subset \TB\Thick_G(s)$ be such that $(z,v) \in F_{r,s}$ if $\at_r(z,v)=\infty$, and $|\h(z)| \le s$. 
By Proposition 5.6 we have that  $\flow_{3r+s}(F_{r,s}) \subset A(t)$, for every $t>0$ (here $A(t)$ is the set defined in the statement of Lemma 5.1). Since $\LM(A(t)) \to 0$ when $t \to \infty$, we find that $\LM(\flow_{3r+s}(F_{r,s}))=\LM(F_{r,s})=0$, so the map $\abs_r$ is defined almost everywhere on $\TB\Thick_G(S)$, for every $s>0$. 

\end{proof}

\begin{definition} For every $r>2$, we define the map $\fta_r:\TB\Ha \to \NB\Gamma(G)$, as $\fta_r(z,v)=\ft_{\abs^{1}_{r}(z,v) }(\ct(\abs^{2}_r(z,v) ))$.
That is, $\fta_r \NB\Gamma(G)$ is the foot of the centre of the triangle $\abs^{2}_r(z,v) \in \Tr(G)$.
\end{definition}
The map $\fta_r$ commutes with the action of the group $G$, and we have that the induced map $\fta_r:\TB{S} \to \NB\Gamma(S)$, is well defined.
\begin{proposition} There exists a universal constant $C>0$, and  $r_0>0$, so that for $r>r_0$, we have that 
$\dis(\fta_r(z,v),\ft_{\abs^{1}_{r}(z,v)}(z)) \le Ce^{-r}$.
\end{proposition}
\begin{proof} It follows from Proposition 5.2 that $|\ang_{\abs_{r}(z,\omega^{j} v)}(z,\omega^{j} v)|\le \pi e^{-r}$. 
By Proposition 5.4, for $r$ large enough  we have $\dis(z,\ct(\abs^{2}_r(z,v))) \le C \pi e^{-r}$, for some universal constant $C>0$. This proves the proposition.

\end{proof}

Denote by $\Mes_{\LM}(\TB{S})$, the space of measures from $\Mes(\TB{S})$, that are absolutely continuous with respect to the Liouville measure $\LM$.
Let $\nu \in \Mes_{\LM}(\TB{S})$. Then $(\abs^{2}_r)_{*}(\nu) \in \Mes(\Tr(S))$, and $(\fta_r)_{*} (\nu) \in  \Mes(\NB\Gamma(S))$, are well defined since 
$\abs^{1}_r$ and $\abs^{2}_r$ are defined almost everywhere on $\TB{S}$. 
\vskip .1cm
Let $\varphi:S \to \R$, be an integrable,  non-negative function. We have the induced map $\varphi:\TB{S} \to \R$, given by $\varphi(z,v)=\varphi(z)$.
Then   $\varphi \, d\LM \in \Mes_{\LM}(\TB{S})$. Since $\varphi(z,v)=\varphi(z,\omega v)=\varphi(z,\omega^{2} v)$, and since $\abs^{2}_r(z,v)=\abs^{2}_r(z,\omega v)=\abs^{2}_r(z,\omega^{2} v)$, we have
\begin{equation}\label{boundary-measure-0}
3(\fta_r)_{*}(\varphi \, d\LM)=\ph (\abs^{2}_r)_{*}(\varphi \, d\LM),
\end{equation}
\noindent
where $\ph:\Mes(\Tr(S)) \to \Mes(\NB\Gamma(S))$, is the operator defined at the beginning of Section 4. 
\vskip .1cm
Let $(z,v) \in \TB\Ha$, and set $(z_1,v_1)=\refl(z,v)$. We say that  $(z,v) \in (\TB\Ha)^{+}$ (or that $(z,v)$ is above the geodesic $\p^{1}(z,v)$)  if $\h(z)> \h(z_1)$ and $(z,v) \in (\TB\Ha)^{-}$ (or that $(z,v)$ is below the geodesic $(\p^{1}(z,v)$) if  $\h(z)<\h(z_1)$. Since for almost every $(z,v) \in \TB\Ha$ we have that either  $\h(z)>\h(z_1)$ or $\h(z)<\h(z_1)$ we see that $(\TB\Ha)^{+} \cup (\TB\Ha)^{-}$ has  full measure in $\TB\Ha$. The sets $(\TB\Ha)^{+}$ and $(\TB\Ha)^{-}$
are disjoint and $\refl((\TB\Ha)^{-})=(\TB\Ha)^{+}$. The sets  $(\TB\Ha)^{+}$ and $(\TB\Ha)^{-}$ are invariant under $G$ and the sets 
$(\TB{S})^{+}$ and $(\TB{S})^{-}$ are defined accordingly (recall that $\h(z)$ is well defined for $z \in S$).

\begin{definition} Fix $r>2$, and let $(z,v) \in \TB\Ha$. We say that $(z,v) \in \A_r(G)$ if $(z,v) \in \TB\Thick_G(2r) \cap \refl\big(  \TB\Thick_G(2r) \big)$, and if $\abs^{1}_r(z,v)=\abs^{1}_r(\refl(z,v))$ 
(here we assume that both $\abs^{1}_r(z,v)$ and $\abs^{1}_r(\refl(z,v))$ are well defined). 
\end{definition}
The set $\A_r(G)$ is  invariant under $G$  and by $\A_r(S)$  we denote the corresponding subset of $\TB{S}$.

\begin{definition}  Fix $r>2$ and let $c \in \Cus(G)$. Let 

$$
Q_c=\{ (z,v) \in \TB\Hor_c(2r) \cup \refl\big(\Hor_c(2r)\big): \, c \quad \text{is not an endpoint of } \quad \abs^{1}_r(z,v) \}.
$$
\noindent
Set $\wt{Q}_c=Q_c \cap \refl(Q_c)$. Let

$$
\B_r(G)=\bigcup \{(z,v) \in \wt{Q}_c: \abs^{1}_r(z,v)=\abs^{1}_r(z',v'), \,\, \text{whenever}
\,\, (z',v') \in \wt{Q}_c \,\, \text{and} \,\, \p^{1}(z,v)=\p^{1}(z',v') \},
$$
\noindent
where the union is taken over all $c \in \Cus(G)$.
\end{definition}
The set $\B_r(G)$ is  invariant under $G$  and by $\B_r(S)$  we denote the corresponding subset of $\TB{S}$.
Note that 
$$
\B_r(G) \subset \TB\Thin_G(2r) \cup \refl\big(\TB\Thin_G(2r) \big),
$$
\noindent
so $\A_r(G) \cap \B_r(G)=\emptyset$.

\begin{definition} Let $r>2$ and let $c \in \Cus(G)$. 
Let 
$$
Q'_c=\{ (z,v) \in \TB\Hor_c(2r) \cap (\TB\Ha)^{+}: \, c \quad \text{is not an endpoint of } \quad \abs^{1}_r(z,v) \}.
$$
\noindent
Set
$$
\Col_r(G)=\bigcup_{c \in \Cus(G)}  Q'_c .
$$
\noindent
Let $(z,v) \in \Col_r(G)$. We say that $(z,v) \in \Col^{L}_r(G)$, if $\abs_r(z,v)=\abs_r(\omega(\inv^{L}(z,v)))$, and $(z,v) \in \Col^{R}_r(G)$, if  $\abs_r(z,\omega v)=\abs_r(\inv^{R}(z,v))$.
\end{definition}
The sets $\Col_r(G)$ and $\Col^{j}_r(G)$ are invariant under $G$ and the corresponding quotients are denoted by
$\Col_r(S)$ and $\Col^{j}_r(S)$.  In the definition of $\Col^{j}_r(G)$, we consider the maps $\inv^{L}, \inv^{R}$, as the maps of $\TB\Ha \setminus \TB\Thick_G(0)$. Since  $(z,v) \in (\TB\Ha)^{+}$, we have that $\inv^{L}(z,v)$, and $\inv^{R}(z,v)$, are well defined. Assume that $(z,v) \in \B_r(G)\cap (\TB\Ha)^{+}$. Then $z \in \Hor_c(1)$ for some $c \in \Cus(G)$ and $\h(z) \ge 2r$ (by definition of $\B_r(G)$). Moreover, neither $\gamma_{(z,v)}$ or $\gamma_{(z,\omega v)}$ gets absorbed by the cusp $c$. This shows that every such $(z,v)$ belongs to  $\Col_r(G)$, that is 
\begin{equation}\label{B-subset-C}
\B_r(G) \cap (\TB\Ha)^{+} \subset \Col_r(G).
\end{equation}

\begin{proposition} We have $\inv^{L}(\Col_r(G))=\Col_r(G)$, and $\inv^{L}(\Col^{L}_r(G) )=\Col^{R}_r(G)$.
\end{proposition}
\begin{proof} Let $(z,v) \in \Col_r(G)$, and let   $(z_1,v_1)=\inv^{L}(z,v)$. There exists a cusp $c \in \Cus(G)$, so that $z \in \Hor_c(2r)$. Set $G=G_c$, and $c=\infty$. It follows from the definition of $\inv^{L}$, that  $\h_{\infty}(z)=\h_{\infty}(z_1)$. Also $\ang_{\infty}(z,v)=-\ang_{\infty}(z_1,\omega v_1)$, and $\ang_{\infty}(z,\omega v)=-\ang_{\infty}(z_1,v_1)$. This implies that $\h_{\infty}(\z_{\max}(\gamma_{(z_{1},\omega v_{1} )},\infty))=\h_{\infty}(\z_{\max}(\gamma_{(z,v)},\infty))$, and $\h_{\infty}(\z_{\max}(\gamma_{(z_{1}, v_{1} )},\infty))=\h_{\infty}(\z_{\max}(\gamma_{(z,\omega v)},\infty))$.
\vskip .1cm
Since $(z,v) \in \Col_r(G)$, we have  $\abs_r(z,v) \ne \infty \ne \abs_r(z,\omega v)$. In order to show $(z_1,v_1) \in \Col_r(G)$, we need to show that
$\abs_r(z_1,v_1) \ne \infty \ne \abs_r(z_1,\omega v_1)$. Assume that $\abs_r(z_1,v_1)=\infty$. Then $\h_{\infty}(\z_{\max}(\gamma_{(z_{1}, v_{1} )},\infty))-\h_{\infty}(z_1) \ge r$. But then $\h_{\infty}(\z_{\max}(\gamma_{(z, \omega v )},\infty))-\h_{\infty}(z) \ge r$, so $\abs_r(z,\omega v)=\infty$, which is a contradiction. Similarly we show $\abs_r(z_1,\omega v_1) \ne \infty$. Putting this together proves $\inv^{L}(\Col_r(G))\subset \Col_r(G)$. In the same way we show  $\Col_r(G)\subset \inv^{L}(\Col_r(G))$. 
\vskip .1cm
The equality $\inv^{L}(\Col^{L}_r(G) )=\Col^{R}_r(G)$, follows directly from the definition.
\end{proof}

We are yet to see that $\A_r(S)$, $\B_r(S)$, and $\Col^{j}_r(S)$, are non-empty sets (see Lemma 5.3).
The following proposition is elementary and the proof is left to the reader.

\begin{proposition} Let $r>2$. If $(z,v) \in \TB\Thin_G(2r) \cup \refl\big(\TB\Thin_G(2r) \big) $  
then $0<2r- \log 3 < \h(z)$. 
\end{proposition}

\begin{definition} Let $r>2$, and let $(z,v) \in \Col_r(G)$. Let $c \in \Cus(G)$, so that $z \in \Hor_c(2r)$. Set $\taso_r(z,v)=c$. By $\taso^{1}_r(z,v)$, we denote the geodesic with the endpoints $\abs_r(z,v)$ and $\taso_r(z,v)$ (providing that $\abs_r(z,v)$ exists). By $\taso^{2}_r(z,v) \in \Tr(G)$, we denote the  triangle  with the vertices $\abs_r(z,v)$, $\abs_r(z,\omega v)$, and $\taso_r(z,v)$ (providing that $\abs_r(z,\omega^{j} v)$, $j=0,1$, exist). 
\end{definition}
From the definition of the set $\Col_r(G)$,  we have that $ \abs_r(z,v) \ne c \ne \abs_r(z,\omega v)$ (if $ \abs_r(z,\omega ^{i} v)= c$, $i=0,1$, then  $\at_r(z,\omega^{i} v)=0$). This shows that the maps $\taso_r(z,v)$, $\taso^{1}_r(z,v)$, and $\taso^{2}_r(z,v)$ are well defined almost everywhere on $\Col_r(G)$. Note that these three maps commute with the action of $G$, so we have the induced maps $\taso_r:\Col_r(S) \to \Cus(S)$, 
$\taso^{1}_r:\Col_r(S) \to \Gamma(S)$, and $\taso^{2}_r:\Col_r(S) \to \Tr(S)$.
The  edges of the triangle $\taso^{2}_r(z,v)$, are $\abs^{1}_r(z,v)$, $\taso^{1}_r(z,v)$, and $\taso^{1}_r(z,\omega v)$.
\vskip .1cm
For $(z,v) \in \Col_r(G)$, let $\tft_r(z,v)=\ft_{\abs^{1}_{r}(z,v)}(\ct(\taso^{2}_r(z,v)))$, that is $\tft_r(z,v)$ is the foot of the centre of the triangle $\taso^{2}_r(z,v) \in \Tr(G)$, with respect to the geodesic  $\abs^{1}_r(z,v) \in \Gamma(G)$. 
Set $\tft^{L}_r(z,v)=\ft_{\taso^{1}_{r}(z,v)}(\ct(\taso^{2}_r(z,v)))$, and $\tft^{R}_r(z,v)=\ft_{\taso^{1}_{r}(z,\omega v) }(\ct(\taso^{2}_r(z,v)))$.
The notation $\tft^{L}_r(z,v)$ indicates that the point $\tft^{L}_r(z,v)$ belongs to the vertical edge of $\taso^{2}_r(z,v)$ that is on  the ``left-hand" side of  $\taso^{2}_r(z,v)$, with the normalisation  $G=G_{\taso_{r}(z,v)}$. Similarly, the point $\tft^{R}_r(z,v)$ belongs to the vertical edge of $\taso^{2}_r(z,v)$ that is to the ``right-hand" side of $\taso^{2}_r(z,v)$.

\begin{proposition} There exists $r_0>0$, so that for $r>r_0$, the following holds. Let $(z,v) \in \Col^{L}_r(G)$, and set $\inv^{L}(z,v)=(z_1,v_1)$. Assume that  $\taso^{2}_r(z,v)$ and $\taso^{2}_r(z_1,v_1)$, exist. Then $\dis(\tft^{L}_r(z,v),\tft^{R}_r(z_1,v_1)) \le e^{r}$. Similarly,  if $(z,v) \in \Col^{R}_r(G)$, for $\inv^{R}(z,v)=(z_1,v_1)$, we have $\dis(\tft^{R}_r(z,v),\tft^{L}_r(z_1,v_1)) \le e^{r}$, providing that $\taso^{2}_r(z,v)$ and $\taso^{2}_r(z_1,v_1)$, exist.
\end{proposition}

\begin{proof} Suppose $(z,v) \in \Col^{L}_r(G)$. Let $c=\taso_r(z,v)$, and set $G=G_c$. Let $\p(z_1,v_1)=p_1$, $\p(z,\omega v)=p_2$, and  $\p(z_1,\omega v_1)=\p(z,v)=p_3$ (the identity  $\p(z_1,\omega v_1)=\p(z,v)$ follows from the definition of $\inv^{L}$).
Since $\h(z)=\h_{\infty}(z)=\h_{\infty}(z_1) \ge 2r>r+2$, we have by Proposition 5.6 that if $\abs_r(z_1,v_1)=c_1 \in \Cus(G)$, then $\Hor_{c_{1}}(1)$, is the first $1$-horoball that the geodesic ray $\gamma_{(z_{1},v_{1})}$ enters after leaving $\Hor_{\infty}(1)$. With this normalisation, we have that the Euclidean diameter of the horoball $\Hor_{c_{1}}(1)$, is at most $1$.  Since the ray $\gamma_{(z_{1},v_{1})}$ ends at the point $p_1$, we conclude that $|c_1-p_1| \le 1$, where $|c_1-p_1|$ is the Euclidean distance between the real numbers $c_1$ and $p_1$. 
Let $\abs_r(z,\omega v)=c_2 \in \Cus(G)$. Similarly, we see that  $|c_2-p_2| \le 1$. 
\vskip .3cm
Since $(z,v) \in \Col^{L}_r(G)$, we have  $\abs_r(z,v)=\abs_r(z_1,\omega v_1)=c_3 \in \Cus(G)$. Similarly we see that $|c_3-p_3| \le 1$.
From the fact that $\h_{\infty}(z)\ge 2r$, we have that $|p_1-p_3|=|p_2-p_3| \ge e^{2r}$. This implies that 
$$
\log {{|c_3-c_1|}\over{|c_3-c_2|}}\le {{2}\over{e^{2r}-1}} +o({{1}\over{e^{2r}}})<e^{-r},
$$
\noindent
for $r$ large enough. Since 
$$
\dis(\tft^{L}_r(z,v),\tft^{R}_r(z_1,v_1))=\log {{|c_3-c_1|}\over{|c_3-c_2|}},
$$
\noindent
the proposition follows.
\end{proof}

Let $\nu \in \Mes_{\LM}(\Col_r(S))$.  Then $(\taso^{2}_r)_{*}(\nu) \in \Mes(\Tr(S))$, and $(\tft_r)_{*} (\nu), (\tft^{L}_r)_{*} (\nu),
(\tft^{R}_r)_{*} (\nu) \in  \Mes(\NB\Gamma(S))$, are well defined since  $\abs^{1}_r$ and $\abs^{2}_r$ are defined almost everywhere on $\Col_r(S)$. By definition, we have 

\begin{equation}\label{boundary-measure-1}
\ph (\taso^{2}_r)_{*}(\nu)=(\tft_r)_{*} (\nu)+(\tft^{L}_r)_{*} (\nu)+(\tft^{R}_r)_{*} (\nu)
\end{equation}

\begin{definition} If $\abs_r(z,v)$ is defined, then the $r$-combinatorial length  $\cl_r(z,v)$ is defined as follows. Let $\gamma$ be the geodesic ray that connects $z$ and $\abs_r(z,v)$. Let $\iota(\gamma,\tau(G))$, be the number of (transverse) intersections between the ray $\gamma$ and the edges from $\lambda(G)$. If $z$ does not belong to  $\Hor_{\abs_{r}(z,v)}(1)$, then $\cl_r(z,v)=\iota(\gamma,\tau(G))$. If $z$ belongs to  $\Hor_{\abs_{r}(z,v)}(1)$, then $\cl_r(z,v)=e^{\h_{\abs_{r}(z,v)}(z)}=e^{\h(z)}>1$. 
\end{definition}

Again, the function $\cl_r:\TB{S} \to \R^{+}\cup \{0\}$, is defined almost everywhere.

\begin{proposition} There exists $r_0>0$, so that for $r>r_0$, the following holds.  For  $(z,v) \in \TB\Ha$, we have  

$$
\cl(\abs^{1}_r(z,v),\fta_r(z,v)) \le e(\cl_r(z,v)+\cl_r(z,\omega v)). 
$$
\noindent
Let  $(z,v) \in \Col_r(G)$. We have 
$$
\cl(\abs^{1}_r(z,v),\tft_r(z,v)) \le \cl_r(z,v)+\cl_r(z,\omega v),
$$
\noindent
$$
\cl(\taso^{1}_r(z,v),\tft^{L}_r(z,v)) \le 5(\cl_r(z,v)+\cl_r(z,\omega v)),
$$
\noindent
and 
$$
\cl(\taso^{1}_r(z,\omega v),\tft^{R}_r(z,v)) \le 5(\cl_r(z,v)+\cl_r(z,\omega v)).
$$

\end{proposition}

\begin{proof} 
Let $\gamma$ denote the geodesic ray that connects $z$ and $\abs_r(z,v)$, and let $\gamma'$ denote the geodesic ray that connects $z$ and 
$\abs_r(z,\omega v)$. Then every edge from $\lambda(G)$, that intersects (transversely) the geodesic $\abs^{1}_r(z,v)$, has to intersect one of the rays
$\gamma$ or $\gamma'$. This shows
$$
\iota(\abs^{1}_r(z,v),\tau(G)) \le \iota(\gamma,\tau(G))+\iota(\gamma',\tau(G)).
$$
\noindent
Assume that  $\fta_r(z,v)$ belongs to the $1$-horoball of one of  the cusps $\abs_r(z,v)$ or $\abs_r(z,\omega v)$, say $\fta_r(z,v)$ belongs to the $1$-horoball at the cusp $\abs_r(z,v)$. It follows from Proposition 5.8 that for $r$ large enough, we have  
$$
\dis(\fta_r(z,v),z) \le \dis(\ft_{\abs^{1}_{r}(z,v) }(z), \fta_r(z,v) )+\dis(z,  \ft_{\abs^{1}_{r}(z,v) }(z) )\le Ce^{-r}+ \log \sqrt{3}<1.
$$
\noindent
Therefore,  we have that  $e^{\h_{\abs_{r}(z,v) }(\fta_r(z,v) ) } \le e^{(1+\h_{\abs_{r}(z,v) }(z) ) }$. 
This proves the first inequality in this proposition. 
\vskip .1cm
Assume now that $(z,v) \in \Col_r(G)$. Since $\abs_r(z,v) \ne \taso_r(z,v) \ne \abs_r(z,\omega v)$,  
we have $\cl(\abs^{1}_r(z,v),\tft_r(z,v))=\iota(\abs^{1}_r(z,v),\tau(G))$. Let $\gamma$ and $\gamma'$ be as above. Then every edge from $\lambda(G)$, that intersects (transversely) the geodesic $\abs^{1}_r(z,v)$, has to intersect one of the rays
$\gamma$ or $\gamma'$. This proves the second inequality $\cl(\abs^{1}_r(z,v),\tft_r(z,v)) \le \cl_r(z,v)+\cl_r(z,\omega v)$.
\vskip .1cm
If we prove 
$$ 
\cl(\taso^{1}_{r}(z,v),\tft^{L}_r(z,v)) \le 5 \cl(\abs^{1}_r(z,v), \tft_r(z,v)),
$$
\noindent 
then the third inequality would follow from the second. Note that $\dis( \tft_r(z,v),\tft^{L}_r(z,v)) < \log 3$. We have
$$ 
\cl(\taso^{1}_r(z,v),\tft^{L}_r(z,v))\le \iota(\taso^{1}_r(z,v),\tau(G))+e^{h_{\max}(\tft^{L}_r(z,v))} \le  \iota(\taso^{1}_r(z,v),\tau(G))+e^{\log 3+h_{\max}(\tft_r(z,v))}.
$$
\noindent
The geodesic $\abs^{1}_r(z,v)$ intersects every edge from $\lambda(G)$, that is intersected by $\taso^{1}_r(z,v)$. In addition, the geodesic $\abs^{1}_r(z,v)$ intersects at least $2e^{(h_{\max}(\tft_r(z,v))-1)}$ edges from $\lambda(G)$, that all have $\taso_r(z,v)$ as their endpoints (since $\taso^{1}_r(z,v)$ has $\taso_r(z,v)$ as an endpoint, we see that $\taso^{1}_r(z,v)$ can not intersect (transversely) any edge from $\lambda(G)$, that ends at $\taso_r(z,v)$). We have

$$
\cl(\abs^{1}_r(z,v), \tft_r(z,v))\ge \iota(\taso^{1}_r(z,v),\tau(G))+ 2e^{(h_{\max}(\tft_r(z,v))-1)} \ge {{2}\over{3e}}\cl(\taso^{1}_r(z,v),\tft^{L}_r(z,v)),
$$
\noindent
which proves the third inequality. The fourth inequality is proved in the same way.

\end{proof}

\subsection{Certain special sets and their properties}
We have
\begin{definition} Let $r>2$. Define $\varphi_r:S \to \R$, by $\varphi_r(z)=1$, if $z \in \Thick_S(2r)$, and by
$$
\varphi_r(z)= e^{{{1}\over{2}} \min\{0,(2r-\h(z)) \} }.
$$
\end{definition}
Note that $\varphi_r$ is continuous on $S$, and $\varphi_r \, d\LM \in \Mes_{\LM}(\TB{S})$. The induced function $\varphi_r:\TB{S} \to \R$, is also denoted by $\varphi_r$.
\begin{remark} The reason that the factor ${{1}\over{2}}$ appears in the definition of $\varphi_r$, is Lemma 5.5 below. In fact, we could replace ${{1}\over{2}}$, with any number between $0$ and $1$. Also, note that if $(z,v) \in \A_r(S)$, then $\varphi_r(\refl(z,v))=\varphi_r(z,v)=1$.
\end{remark}

For $r>2$, the set $\FD_r(S) \subset \TB{S}$, is defined so that $(z,v) \in \FD_r(S)$, if $\h(z) \le 10r$, and if the inequalities $\at_r(z,v), \at_r(z,\omega v), \at_r(z,\omega^{2} v) \le r^{2}$, hold.
Note that if $(z,v) \in \FD_r(S)$, then for $r>4$, we have $\h(T) \le r^{2}$, where $T=\abs^{2}_r(z,v) \in \Tr(S)$ (recall Definition 3.4 for the definition of $\h(T)$).  The following lemma will be proved in Section 7.

\begin{lemma} There exists $r_0>0$, so that for $r>r_0$, we have

$$
\int\limits_{\TB{S} \setminus \FD_r(S)}  (\cl_r(z,v)+\cl_r(z,\omega v)+\cl_r(z,\omega^{2} v)+1) \varphi_r(z) \, d\LM \le P(r)e^{-r}.
$$
\end{lemma}

The following lemma will also be proved in Section 7.

\begin{lemma} There exists $r_0>0$, so that for $r>r_0$, we have

$$
\int\limits_{H_{r}} (\cl_r(z,v)+\cl_r(z,\omega v)+ \cl_r(z,\omega^{2} v)  +1) \varphi_r(z) \, d\LM \le P(r)e^{-r},
$$
\noindent
where 
$H_r=H'_r \cap \FD_r(S)$, and 
$$
H'_r=(\TB{S} \setminus (\A_r(S)\cup \B_r(S)))  \cup (\Col_r(S) \setminus \Col^{L}_r(S)) \cup (\Col_r(S) \setminus \Col^{R}_r(S)).
$$
\end{lemma}

\vskip .1cm
We now give a better description of the set $H_r$. The results that follow in this subsection will not be used in  the proof of Theorem 4.1 below, so the reader may skip the rest of this subsection and go to the next subsection and the proof of Theorem 4.1.
\vskip .1cm

\begin{proposition} Let $z_{-1},z_1 \in \Ha$, so that $\dis(z_{-1},z_1)\ge 2$. Let $\gamma$ be the geodesic that contains $z_{-1}$ and $z_1$, and assume that  $\z_{\max}(\gamma,\infty)=z_0$, where $z_0$ is the midpoint of the geodesic segment between $z_{-1}$ and $z_1$. Let $w_{-1},w_1 \in \Ha$, so that
$\dis(z_{-1},w_{-1}),\dis(z_1,w_1) \le \delta$, for some $0 \le \delta$, and let   $\gamma'$ be the geodesic that contains $w_{-1}$ and $w_1$.
There exists a universal constant $C>0$, so that for $\delta$ small enough, we have 
$\dis(\z_{\max}(\gamma,\infty),\z_{\max}(\gamma',\infty))\le C\delta$.
\end{proposition}
\begin{proof}  We work in the unit disc $\D$. The point $\infty \in \partial{\Ha}$, corresponds to the point $i \in \partial{\D}$. 
We may assume that  $z_0=0$. Then $\gamma$ is the geodesic that connects $-1$ and $1$. Moreover, $z_{-1}=-x$, and $z_1=x$, for some $x>(e-1)/(e+1)$. 
Let $q,q_1$, be the endpoints of $\gamma'$. From the assumption $\dis(z_{-1},w_{-1}),\dis(z_1,w_1) \le \delta$, and since $\dis(z_{-1},z_1)>2$, we see that there exists a universal constant $D>0$, so that $|q+1|,|q_1-1| \le D\delta$. The M\"obius transformation  $f$ is uniquely determined by the conditions $f(i)=i$,  $f(-1)=q$, and $f(1)=q_1$. Moreover,  $f$ depends smoothly on $q$ near $-1$, and $q_1$ near $1$. Therefore, there exists a universal constant $C>0$, so that $\dis(z_0,f(z_0))=\dis(0,f(0)) \le C\delta$. Since $f(i)=i$, we have $f(z_0)=\z_{\max}(\gamma',i)$, and this proves the proposition.
\end{proof}

Let $z_1,z_2,w_1,w_2 \in \Ha$, such that $\dis(z_1,w_1)=\dis(z_2,w_2)$. Let $\zeta_j$, $j=1,2$,  be the point on the geodesic segment between $z_j$ and $w_j$, so that $\dis(z_1,\zeta_1)=\dis(z_2,\zeta_2)=d$.  Then $\dis(w_1,\zeta_1)=\dis(w_2,\zeta_2)=d'$, for some $d' \ge 0$. We have the following elementary inequality in hyperbolic geometry
\begin{equation}\label{long-range-estimate}
\dis(\zeta_1,\zeta_2) \le D(\dis(z_1,z_2)e^{-d}+\dis(w_1,w_2)e^{-d'}),
\end{equation}
\noindent
for some universal constant $D>0$.

\begin{proposition} There exists $r_0>0$, so that for $r>r_0$, the following holds. Let $(z,v),(z',v') \in \TB\Ha$, such that $\p(z,v)=\p(z',v')$, and  $\dis(z,z') <10$. Suppose that  $\Delta_{\p(z,v)}[z,z']=0$. Let $p \in \partial{\Ha}$, and suppose $\Delta_p[\z_{\max}(\gamma_{(z,v)},p),z] \ge r$. Then 
\begin{equation}\label{geom-estimate-1}
\dis(\z_{\max}(\gamma_{(z,v)},p),\z_{\max}(\gamma_{(z',v')},p)) \le re^{-r}.
\end{equation}
\noindent
Moreover, we have
\begin{equation}\label{geom-estimate-2}
|\Delta_p[\z_{\max}(\gamma_{(z,v)},p),z]-\Delta_p[\z_{\max}(\gamma_{(z',v')},p),z']|<2re^{-r}.
\end{equation}

\end{proposition}

\begin{proof} The rays $\gamma_{(z,v)}$ and $\gamma_{(z',v')}$  end at the same point at $\partial{\Ha}$.  Since  $\Delta_{\p(z,v)}[z,z']=0$,
and since $\dis(z,z') <10$, for any $t>0$, applying (\ref{long-range-estimate}) we obtain

\begin{equation}\label{distance-1}
\dis(\gamma_{(z,v)}(t),\gamma_{(z',v')}(t)) \le 10De^{-t}.
\end{equation}
\noindent
Let $t_0 >0$, so that $\gamma_{(z,v)}(t_0)=\z_{\max}(\gamma_{(z,v)},p)$. By Proposition 5.2 we have $t_0>\Delta_p[\z_{\max}(\gamma_{(z,v)},p),z] \ge r$.
Let $\zeta=\gamma_{(z,v)}(t_0-1)$, and $\zeta_1=\gamma_{(z,v)}(t_0+1)$. Also, let $w=\gamma_{(z',v')} (t_{0}-1)$, and $w_1=\gamma_{(z',v')}(t_0+1)$. By (\ref{distance-1}) we have
$\dis(\zeta,w),\dis(\zeta_1,w_1) \le 10e^{1-t_{0}}$, so it follows from the previous proposition that for $r$ large enough, we have
$$
\dis(\z_{\max}(\gamma_{(z,v)},p),\z_{\max}(\gamma_{(z',v')},p)) \le 10DCe^{1-t_{0}}<re^{-r}.
$$
\noindent
This shows that (\ref{geom-estimate-1}) holds.
\vskip .1cm 
We may assume that $\p(z,v)=\infty$. Let $\alpha:\R \to \Ha$, denote the naturally parametrised horocircle (with respect to $\infty$), so that the oriented angle between the vectors $\alpha'(0)$ and $v$, is ${{\pi}\over{2}}$. Then for every $s$, we have
$$
\Delta_{\infty}[z,\alpha(s)]=0.
$$
\noindent
Moreover, there exists  $s_0 \in \R$, so that $\alpha(s_0)=z'$. Assume $s_0 \ge 0$ (the other case is handled in the same way). The number $s_0$ depends only on the upper bound of the hyperbolic distance between $z$ and $z'$ which is bounded above by $10$.
\vskip .1cm
By $\alpha_r:\R \to \Ha$ denote the naturally parametrised horocircle (with respect to $p$), so that the oriented angle between the vectors
$\alpha'_r(0)$ and $v$ is positive.  It follows from Proposition 5.2  that $|\ang_p(z,v)|<\pi e^{-r}$. Therefore, the angle between the vectors $\alpha'(0)$ and $\alpha'_r(0)$ is at most $\pi e^{-r}$. We find that there exists  a constant $K>0$ that depend only on $s_0$ with the following properties. For $0\le s \le (s_0+1)$ we have 
\begin{equation}\label{hyp-estimate-1}
|\Delta_{\infty}[z,\alpha_r(s)]| \le Ke^{-r}, 
\end{equation}
\vskip .1cm
Let $\gamma:\R \to \Ha$ denote the naturally parametrised geodesic that contains the geodesic ray $\gamma_{(z',v')}$ such that $\gamma(0)=z'$.
For $r$ large enough, we have that $\alpha_r$ intersects the geodesic $\gamma$. This is true because when $r \to \infty$ we have that  the horoball at $p$ that contains $z$ converges on compact sets in $\Ha$ to the horoball at $\infty$ that contains $z$ and on the other hand we have that $\alpha$ intersect $\gamma$ orthogonally at $z'$. Let $z''=\alpha_r \cap \gamma$. Let $0 \le  s'$, be such that $\alpha_r(s')=z''$. Then for $r$ large enough, we have  $0 \le  s'  \le (s_0+1)$.    We have $\dis(z',z'')=|\Delta_{\infty}[z',z'']|=|\Delta_{\infty}[z,z'']|=|\Delta_{\infty}[z,\alpha_r(s')]|$ so it follows from (\ref{hyp-estimate-1})  that 
$$
\dis(z',z'') \le Ke^{-r}.
$$  
\vskip .1cm
Note $\Delta_p[z,z'']=0$. It follows from  (\ref{geom-estimate-1}) that
$$
|\Delta_p[\z_{\max}(\gamma_{(z,v)},p),z]-\Delta_p[\z_{\max}(\gamma,p),z'']| \le  \dis(\z_{\max}(\gamma_{(z,v)},p),\z_{\max}(\gamma,p)) \le re^{-r}.
$$
\noindent
Since $\dis(z',z'') \le Ke^{-r}$, we have for $r$ large, that
$$
|\Delta_p[\z_{\max}(\gamma_{(z,v)},p),z]-\Delta_p[\z_{\max}(\gamma_{(z',v')},p),z']|\le  re^{-r}+Ke^{-r}=2re^{-r},
$$
\noindent
which proves (\ref{geom-estimate-2}).

\end{proof}

The following proposition states that if two unit vectors in $\TB\Ha$ are nearby and related by the horocyclic flow then either they are absorbed by the same cusp or the first vector ''just misses" being absorbed by some cusp or it just barely gets absorbed by some cusp.

\begin{proposition} Let $r>4$ and let $(z,v), \, (z',v') \in \TB\Ha$ such that $(z,v) \in \FD_r(G)$ and 
such that
\begin{itemize}
\item $\p(z,v)=\p(z',v')$.  
\item $\dis(z,z') \le 10$ and $\Delta_{\p(z,v)}[z,z']=0$.
\end{itemize}
Assume that $\abs_r(z,v) \ne \abs_r(z',v')$. Then at least one of the following two conditions holds
\begin{enumerate}
\item There exists a cusp $c \in \Cus(G)$, such that 
$$
r-r^2e^{-r}<\h_{c}(\z_{\max}(\gamma_{(z,v)},c))-\h_c(z)<r+r^2e^{-r}, 
$$
\noindent
and  $\z_{\max}(\gamma_{(z,v)},c) \in \Hor_c(1)$.
\item There exists a cusp $c \in \Cus(G)$, such that  
$$
1-r^2e^{-r}<\h_c(\z_{\max}(\gamma_{(z,v)},c))<1+r^2e^{-r},
$$
\noindent
and  $\z_{\max}(\gamma_{(z,v)},c)=\gamma_{(z,v)}(t)$, for some $r-1<t<r^2+1$.
\end{enumerate}

\end{proposition}

\begin{proof} 
Set $\abs_r(z,v)=c \in \Cus(G)$. We have  $c \ne \abs_r(z',v')$.
First consider the case when the geodesic ray $\gamma_{(z',v')}$ does not intersect 
$\Hor_c(1)$. In this case we show that the condition ($2$) holds.
It follows from (\ref{geom-estimate-1}) that 
$$
\dis(\z_{\max}(\gamma_{(z,v)},c), \z_{\max}(\gamma_{(z',v')},c)) \le re^{-r}.
$$
\noindent
Since $\z_{\max}(\gamma_{(z,v)},c) \in \Hor_c(1)$ and since in this case $\z_{\max}(\gamma_{(z',v')},c)$ does not belong to  $\Hor_c(1)$,
we conclude that  $1\le \h_c(\z_{\max}(\gamma_{(z,v)},c))<1+r^2e^{-r}$. 
Let $t_0>0$ be such that $\z_{\max}(\gamma_{(z,v)},c)=\gamma_{(z,v)}(t_0)$.
In order to show that ($2$) holds  it remains to prove that $r-1<t_0<r^2+1$. 
Let $t_1 \ge 0$ so that $\gamma_{(z,v)}(t_1)$ is the first point of entry of the geodesic ray $\gamma_{(z,v)}$ into $\Hor_c(1)$. Then $\at_r(z,v)=t_1$ and by the assumption $(z,v) \in \FD_r(G)$ we have $t_1 \le r^{2}$. Since $\h_c(\gamma_{(z,v)}(t_0))-\h_c(\gamma_{(z,v)}(t_1) \le r^{2}e^{-r}$ and since by Proposition 5.2 we have
$t_0-t_1 <\h_c(\gamma_{(z,v)}(t_0))-\h_c(\gamma_{(z,v)}(t_1) +\log 2$, we have  $\at_r(z,v) \le t_0< \at_r(z,v)+1 \le r^{2}+1$. 
On the other hand, since $\abs_r(z,v)=c$ we have  $\h_c(\gamma_{(z,v)}(t_0))-\h_c(z) \ge r$. Then from Proposition 5.2 we have $t_0 >r-1$. This shows that  the condition ($2$) holds.
\vskip .1cm
From now on we assume that  $\gamma_{(z',v')}$ enters the horoball $\Hor_c(1)$ but that $c \ne \abs_r(z',v')$. Then there are two possible reasons why 
$c \ne \abs_r(z',v')$. The first one  is that $\Delta_c[\z_{\max}(\gamma_{(z',v')},c),z']=\h_c(\z_{\max}(\gamma_{(z',v')},c))-\h_c(z')<r$. Since $\Delta_c[\z_{\max}(\gamma_{(z,v)},c),z]=\h_c(\z_{\max}(\gamma_{(z,v)},c))-\h_c(z) \ge r$, from (\ref{geom-estimate-2}) we find that 
$r \le \h_c(\z_{\max}(\gamma_{(z,v)},c))-\h_c(z) < r+r^2e^{-r}$. So in this case   the condition ($1$) holds 
because we already know that $\z_{\max}(\gamma_{(z,v)},c) \in \Hor_c(1)$. 
\vskip .1cm
The second reason is that the geodesic ray $\gamma_{(z',v')}$ gets absorbed before entering the horoball $\Hor_c(1)$. Set $\abs_r(z',v')=c' \in \Cus(G)$. If $t_0 \ge 0$ is such that $\z_{\max}(\gamma_{(z,v)},c')=\gamma_{(z,v)}(t_0)$ then $t_0<\at_r(z,v) \le r^2$. 
Assume that  $\z_{\max}(\gamma_{(z,v)},c')$ does 
not belong to $\Hor_{c'}(1)$. Then we show that the condition ($2$) holds (with respect to the cusp $c'$).
Since  $\z_{\max}(\gamma_{(z',v')},c') \in \Hor_{c'}(1)$ by (\ref{geom-estimate-1}) we have that
$\z_{\max}(\gamma_{(z,v)},c') \in \Hor_{c'}(1-r^2e^{-r})$. Since $\h_{c'}(\z_{\max}(\gamma_{(z',v')},c'))-\h_{c'}(z') \ge r$, from (\ref{geom-estimate-2})
we get that 
$$
\h_{c'}(\gamma_{(z,v)}(t_0) )-\h_{c'}(z)>\h_{c'}(\z_{\max}(\gamma_{(z',v')},c'))-\h_{c'}(z')-r^2e^{-r}>r-r^{2}e^{-r}.
$$
\noindent
This yields that $r-1 < t_0$. We have already seen that $t_0<\at_r(z,v) \le r^2$. This proves the statement.
\vskip .1cm
It remains to examine the case $\z_{\max}(\gamma_{(z,v)},c') \in \Hor_{c'}(1)$.  We show that in this case ($2$) holds.
Again let $t_0 \ge 0$  such that $\z_{\max}(\gamma_{(z,v)},c')=\gamma_{(z,v)}(t_0)$.
It follows from (\ref{geom-estimate-2}) that 
$$
\h_{c'}(\gamma_{(z,v)}(t_0) )-\h_{c'}(z)>\h_{c'}(\z_{\max}(\gamma_{(z',v')},c'))-\h_{c'}(z')-r^2e^{-r}>r-r^{2}e^{-r}.
$$
\noindent
On the other hand, we have that $\h_{c'}(\gamma_{(z,v)}(t_0) )-\h_{c'}(z) <r$ because otherwise we would have that $\abs_r(z,v)=c'$. The last two estimates put together give us that $r-r^{2}e^{-r}< \h_{c'}(\gamma_{(z,v)}(t_0) )-\h_{c'}(z) <r$.
This proves the proposition.

\end{proof}

The following proposition replaces the hypotheses $\dis(z,z') \le 10$ and $\Delta_{\p(z,v)}[z,z']=0$ of Proposition 5.15 with the condition that $z$ and $z'$ are in $\Thick_G(2r-\log 3)$.  

\begin{proposition} Let $r>4$ and $c \in \Cus(G)$. Let $(z,v), \, (z',v') \in \TB\Hor_c(2r-\log 3)$ where $(z,v) \in \FD_r(G)$  such that $\p(z,v)=\p(z',v')$ and such that $\abs_r(z,v) \ne c \ne \abs_r(z',v')$. If $\abs_r(z,v) \ne \abs_r(z',v')$ then  the following holds
\begin{itemize}
\item There exists a cusp $c \in \Cus(G)$, such that  
$$
1-r^2e^{-r}<\h_c(\z_{\max}(\gamma_{(z,v)},c))<1+r^2e^{-r},
$$
\noindent
and  $\z_{\max}(\gamma_{(z,v)},c)=\gamma_{(z,v)}(t)$ for some $r-1<t<r^2+1$.
\end{itemize}

\end{proposition}

\begin{proof} Since $\abs_r(z,v) \ne c \ne \abs_r(z',v')$ we have that $\p(z,v)=\p(z',v') \ne c$. 
Since $2r- \log 3>r+2$ (because $r>4$), it follows from  Proposition 5.6 that $\gamma_{(z,v)}$ gets absorbed by the 
first $1$-horoball it hits after leaving $\Hor_c(1)$. The same is true for the ray $\gamma_{(z',v')}$. 
\vskip .1cm
Let $\eta$ be the horocircle at $\p(z,v)$ that is tangent to $\Hor_c(2r-\log 3)$. Let $\gamma_j$, $j=1,2$,  be the two geodesics that start at $\p(z,v)$ and that are tangent to  $\Hor_c(2r-\log 3)$. Let $\eta_1$ be the subsegment of $\eta$ that is bounded by the points $\gamma_j \cap \eta$ (since $\gamma_j$ starts at the same point on $\R$ where $\eta$ touches $\R$, there exists a unique intersection point $\gamma_j \cap \eta$ in $\Ha$). We have that the hyperbolic length of $\eta_1$ is equal to $1$ and that $\eta_1$ is contained in $\Hor_c(r+2)$. Let $z_1$ and $z'_1$ be the points of intersection between  $\eta_1$ and the geodesic rays $\gamma_{(z,v)}$ and $\gamma_{(z',v')}$ respectively. Observe that $\Delta_{\p(z,v)}[z_1,z'_1]=0$ and $\dis(z_1,z'_1) <1$. Also $z_1,z'_1 \in \Hor_c(r+2)$. 
\vskip .1cm
Let $(z_1,v_1) \in \TB\Ha$ be such that the ray $\gamma_{(z_{1},v_{1})}$ is contained in the ray $\gamma_{(z,v)}$. Similarly let $(z'_1,v'_1) \in \TB\Ha$ be such that the ray $\gamma_{(z'_{1},v'_{1})}$ is contained in the ray $\gamma_{(z'v')}$. Since $z_1,z'_1 \in \Hor_c(r+2)$ from Proposition 5.6 we have that $\abs_r(z_{1},v_{1})=\abs_r(z,v)=c_1$ 
and $\abs_r(z'_{1},v'_1)=\abs_r(z',v')$.  Also $ r+2 \le \at_r(z_1,v_1)<\at_r(z,v) \le r^{2}$. We now apply the previous proposition to $(z_1,v_1)$ and $(z'_1,v'_1)$. Since $r+2 \le \at_r(z_1,v_1)$ we see that the condition ($1$) from the previous proposition can not hold so we have that the condition ($2$) holds for $(z_1,v_1)$ and hence for $(z,v)$.  This proves the proposition.

\end{proof}

We can now define three ''error sets " and show that any point of the ''bad set"  $H_r$ belongs to one of these three sets.

\begin{definition} Let $r>2$. Define the sets $\E^{i}_{r}(G) \subset \FD_r(G)$, $i=1,2,3$, as follows.

We say that $(z,v) \in \E^{1}_{r}(G)$  if 
\begin{itemize}
\item We have $(z,v) \in \FD_r(G)$.
\item  We have $\h(z) \ge 2r-\log 3$.
\item  Let $c \in \Cus(G)$ such that $z \in \Hor_c(2r-\log 3)$. Then $c$ is an endpoint of  $\abs^{1}_r(z,v)$. 
\end{itemize}
We say that $(z,v) \in \E^{2}_{r}(G)$ if
\begin{itemize}
\item There exists a cusp $c \in \Cus(G)$, such that $r-r^2e^{-r}<\h_{c}(\z_{\max}(\gamma_{(z,v)},c))-\h_c(z)<r+r^2e^{-r}$ and  $\z_{\max}(\gamma_{(z,v)},c) \in \Hor_c(1)$.
\item We have $(z,v) \in \FD_r(G)$ and $(z,v)$ does not belong to $\E^{1}_r(G)$.
\end{itemize}
We say that $(z,v) \in \E^{3}_{r}(G)$ if 
\begin{itemize}
\item There exists a cusp $c \in \Cus(G)$, so that  $1-r^2e^{-r}<\h_c(\z_{\max}(\gamma_{(z,v)},c))<1+r^2e^{-r}$
and so that $\z_{\max}(\gamma_{(z,v)},c)=\gamma_{(z,v)}(t)$ for some $r-1<t<r^2+1$. 
\item We have $(z,v) \in \FD_r(G)$ and  $(z,v)$ does not belong to $E^{1}_r(G)$.
\end{itemize}
\end{definition}
Note that the set $\E^{i}_{r}(G)$, $i=1,2,3$, is invariant under the action of $G$ and the corresponding quotient is denoted by $\E^{i}_{r}(S)$.

\begin{proposition} Suppose that $(z,v) \in H_r$. Then $(z,v) \in \E^{i}_r(G)$, for some $i=1,2,3$,  
or $(z,\omega v)\in \E^{i}_r(G)$ for some   $i=2,3$.
\end{proposition}

{
\newcommand{\Th}{{\mathbf Th}}
\renewcommand{\a}{{\mathbf a}}
\newcommand{\I}{{\mathcal I}}
\renewcommand{\H}{{\mathbf H}}
\newcommand{\HH}{{\mathcal H}}
\newcommand{\Q}{{\mathcal Q}}
\newcommand{\Cusp}{\operatorname{Cusp}}

\setlength{\unitlength}{2em}

\begin{figure}[p]
  \input{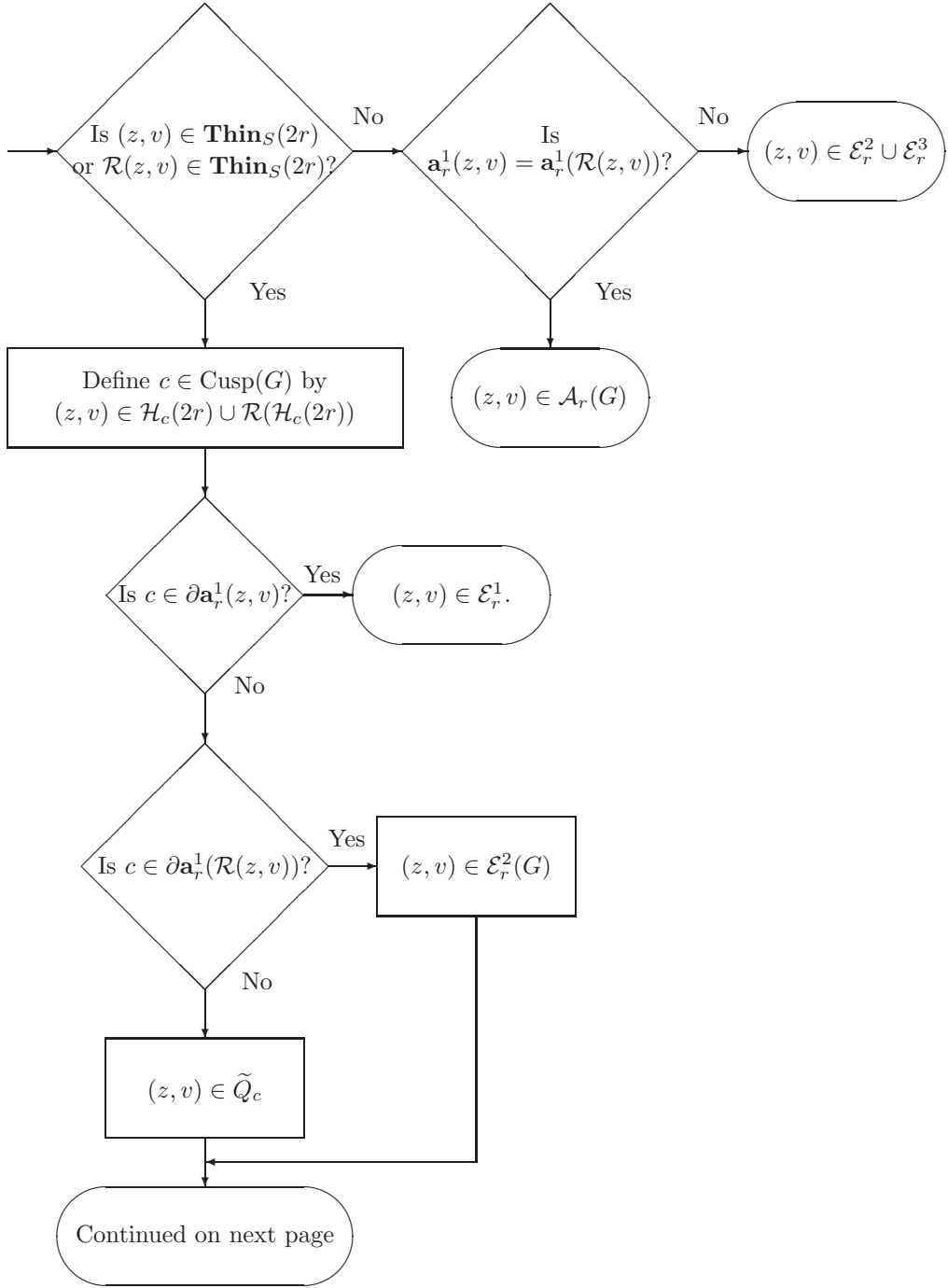}
\caption{Flow chart part 1}
\end{figure}

\begin{figure}
\addtocounter{figure}{-1}
  \centering
  \input{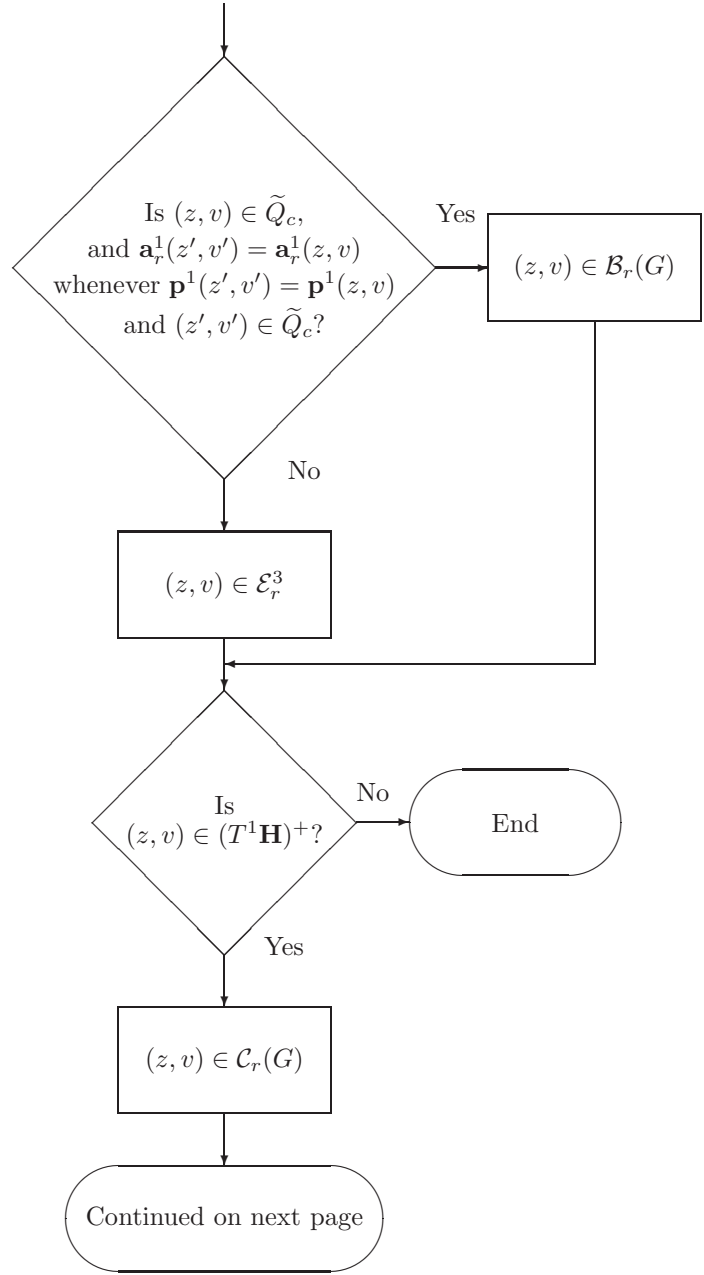}
  \caption{Flow chart part 2}
  \label{fig:part2}
\end{figure}

\begin{figure}
\addtocounter{figure}{-1}
  \centering
  \input{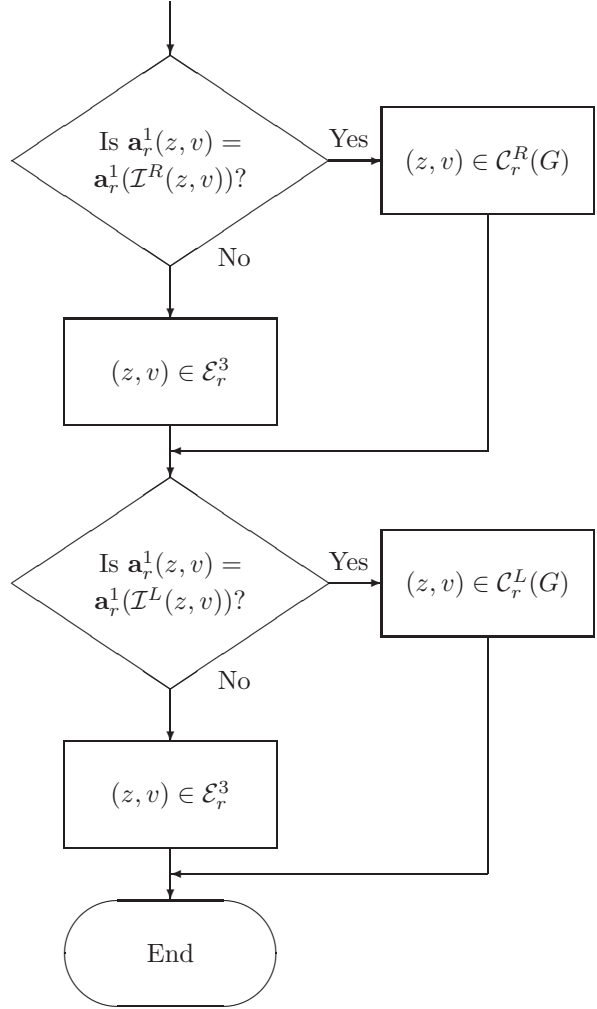}
  \caption{Flow chart part 3}
  \label{fig:part3}
\end{figure}
}

\begin{proof}   We refer the reader to the flow chart for the logic of the proof below. Assume that $(z,v) \in \big(\TB\Ha \setminus (\A_r(G) \cup \B_r(G)) \big)$ and that $(z,v) \in \FD_r(G)$. 
\vskip .1cm
If $(z,v) \in \TB\Thick_G(2r) \cap \refl\big(\TB\Thick_G(2r)\big)$ then 
by the definition of $\A_r(G)$ we find that either $\abs_r(z,v) \ne \abs_r(\omega(\refl(z,v)))$ or 
$\abs_r(z,\omega v) \ne \abs_r(\refl(z,v))$. Suppose that  $\abs_r(z,v) \ne \abs_r(\omega(\refl(z,v)))$ (the other case is handled in the same way), and set $(z',v')=\omega(\refl(z,v))$. Then $\p(z,v)=\p(z',v')$ and 
$\Delta_{\p(z,v)}[z,z']=0$. Also, $\dis(z,z')=\log 3<10$. Then by Proposition 5.15  we have that $(z,v)$ belongs to one of the sets $\E^{j}_r(G)$, $j=2,3$.
\vskip .1cm
If $(z,v) \in \TB\Thin_G(2r) \cup \refl\big(\TB\Thin_G(2r)\big)$ then  by Proposition 5.10 we have $z \in \Hor_{c}(2r-\log 3)$ for some 
$c \in \Cus(G)$. Suppose that $(z,v)$ does not belong to $\B_r(G)$. There are two reasons why this can happen. The first one is that
$(z,v)$ does not belong to the set $\wt{Q}_c=Q_c \cap \refl(Q_c)$ defined in Definition 5.6. Set $(z',v')=\refl(z,v)$. Then $\abs^{1}_r(z,v)$ or $\abs^{1}_r(z',v')$ has $c$ as its endpoint. If $\abs^{1}_r(z,v)$ has $c$ as its endpoint then $(z,v) \in \E^{1}_r(G)$. 
If $\abs_r(z',\omega v')=c$ then by Proposition 5.15 we have that $(z,v)$ belongs to one of the sets $\E^{j}_r(G)$, $j=2,3$. If 
$\abs_r(z',v')=c$  then $(z,\omega v)$ belongs to one of the sets $\E^{j}_r(G)$, $j=2,3$.
\vskip .1cm
Assume $(z,v) \in \wt{Q}_c$. If  $(z,v)$ does not belong to $\B_r(G)$ then there exists   
$(z',v') \in  \TB\Hor_c(2r) \cup \refl\big(\TB\Hor_c(2r)\big)$ such that 
\begin{itemize}
\item $\p^{1}(z,v)=\p^{1}(z',v')$. 
\item $\abs^{1}_r(z',v')$ does not have $c$ as its endpoint.
\item  $\abs^{1}_r(z,v) \ne \abs^{1}_r(z',v')$. 
\end{itemize}
Then $z,z' \in \Hor_{c}(2r-\log 3)$. The condition $\abs^{1}_r(z,v) \ne \abs^{1}_r(z',v')$ means that 
$\abs_r(z,v) \ne \abs_r(z',v')$ or $\abs_r(z,\omega v) \ne \abs_r(z',\omega v')$.
Applying the previous proposition to $(z,v)$ and $(z',v')$ (or  $(z,\omega v)$ and $(z',\omega v')$)
we conclude that $(z,v)$ or $(z,\omega v)$  belongs to the set $\E^{3}_r(G)$. 
\vskip .1cm
Assume that $(z,v) \in (\Col_r(G) \setminus \Col^{L}_r(G))$ and that $(z,v) \in \FD_r(G)$. Then 
$z \in \Hor_c(2r)$ for some cusp $c \in \Cus(G)$ and $c$ is not an endpoint of $\abs^{1}_r(z,v)$ (by the definition of $\Col_r(G)$). 
Then $\p(z,v) \ne c \ne \p(z,\omega v)$.  Set $(z',v')=\inv^{L}(z,v)$ (note that $\inv^{L}(z,v)$ is well defined since $\p(z,v) \ne c \ne \p(z,\omega v)$). Then  $\abs^{1}_r(z',v')$ since by Proposition 5.9 $(z',v') \in \Col_r(G)$. Suppose that $\abs_r(z,v) \ne \abs_r(z',\omega v')$. Then by the previous proposition we have that $(z,v) \in \E^{3}_r(G)$. The case $(z,v) \in (\Col_r(G) \setminus \Col^{R}_r(G))$ is treated in the same way.
\end{proof}

\addtocounter{figure}{-1}

\newpage
\subsection{The proof of Theorem 4.1} Recall the statement of Theorem 4.1.

\begin{theorem} There exist a constant $r_0(S)=r_0$ that depends only on $S$, so that for every $r>r_0$  there exists a finite measure $\mu(r) \in \Mes(\Tr(S))$ so that the total measure $|\mu(r)|$ satisfies the inequality

$$
{{\LM(\TB{S})}\over{2}} < |\mu(r)|<{{3\LM(\TB{S})}\over{2}},
$$
\noindent
and  with the following properties. There exist measures $\alpha(r),\alpha_1(r),\beta(r) \in \Mes(\NBG(S))$ so that the measure $\ph\mu(r) \in \Mes(\NBG(S))$ can be written  as $\ph\mu(r)=\alpha(r)+\alpha_1(r)+\beta(r)$ and  the following holds
\begin{enumerate}
\item Let $\wh{\mu}(r)$ denote the restriction of the measure $\mu(r)$ to the set of triangles $T \in \Tr(S)$ for which $\h(T) \ge r^2$. 
Then
$$
\int\limits_{\NBG (S) } (\cl(\gamma^{*},z)+1) d \ph \wh{\mu}(r) \le P(r)e^{-r}.
$$
\item The measure $\alpha(r)$ is $re^{-r}$-symmetric, and the measure $\alpha_1(r)$ is $Q$-symmetric, for any   $Q >200$.
\item We have
$$
\int\limits_{\NBG(S)}(\cl(\gamma^{*},z)+1) d\beta(r) \le P(r)e^{-r}.
$$
\noindent
\item We have
$$
\int\limits_{\NBG(S)}d\alpha_1(r) \le e^{-r}.
$$
\noindent

\end{enumerate}

\end{theorem}

\begin{proof}

Let $\vartheta_r:\TB{S} \to \R$ be defined as follows. If $(z,v)$ does not belong to $\Col_r(S)$ then $\vartheta_r(z,v)=0$.
For $(z,v) \in \Col_r(S)$ we have 
$$
\vartheta(z,v)=\varphi_r(\refl(z,v))-\varphi_r(z,v).
$$
\noindent
If $(z,v) \in \Col_r(G)$ we have that $z \in \Hor_{\taso_{r}(z,v)}(2r)$. Therefore, $\h(z)=\h_{\taso_{r}(z,v) }(z)$. It follows from the definition of $\varphi_r$ that $\varphi_r(z)$ depends only on that $\h_{\taso_{r}(z,v) }(z)$ and is decreasing in $\h_{\taso_{r}(z,v) }(z)$. 
Since $(z,v) \in \Col_r(S) \subset (\TB{S})^{+}$  we have that $\vartheta_r$ is a non-negative function. 
\vskip .1cm
We define the measure 
\begin{equation}\label{main-measure}
\mu(r)=(\abs^{2}_r)_{*}(\varphi_r \, d\LM)+3(\taso^{2}_r)_{*}(\vartheta_r \, d\LM).
\end{equation}
\noindent
Since $\varphi_r \le 1$ on $\TB{S}$ and since $\varphi_r(z,v)=1$ for every $z \in \Thick_S(r)$ we conclude that  
the total measure of $\varphi_r \, d\LM$ is approaching $\LM(\TB{S})$ when $r \to \infty$. Since $\vartheta_r \le 2$ on $\TB{S}$ and since $\vartheta_r(z)=0$ for every $(z,v) \in \TB\Thick_S(r)$, we have that the total measure of $\vartheta_r \, d\LM$ tends to zero when $r \to \infty$.
So for $r$ large enough we have that  the total measure $|\mu(r)|$ satisfies the inequality
$$
{{\LM(\TB{S})}\over{2}} < |\mu(r)|<{{3\LM(\TB{S})}\over{2}}.
$$
\noindent
Let $\wh{\mu}(r)$ be the restriction of the measure $\mu(r)$ to the set of triangles $T \in \Tr(S)$ for which $\h(T) \ge r^2$. Let
$(z,v) \in \TB{S}$ be such that the triangle $\abs^{2}_r(z,v)$ belongs to the support of $\wh{\mu}(r)$. Then at least one of the inequalities $\at_r(z,\omega^{j} v)> r^2$, $j=0,1,2$, is satisfied, so  $(z,v) \in \TB{S} \setminus \FD_r(S)$. From Proposition 5.12 and Lemma 5.2, we have that for $r$ large enough, the following holds,
\begin{align*}
\int\limits_{\NBG (S) } (\cl(\gamma^{*},z)+1) d \ph \wh{\mu}(r) & <  5\int\limits_{\TB{S} \setminus \FD_r(S)} (\cl_r(z,v)+\cl_r(z,\omega v)+\cl_r(z,\omega^{2} v)+1)\varphi_r(z) \, d\LM \\ 
&< P(r)e^{-r}.
\end{align*}
\vskip .1cm

Next, we define the measures $\alpha(r)$, $\alpha_1(r)$, and $\beta(r)$. Let $\varphi^{A}_r(z,v)=\varphi_r(z)$ if $(z,v) \in \A_r(S)$, and $0$ otherwise. Let $\varphi^{B}_r(z,v)=\varphi_r(z)$ if $(z,v) \in \B_r(S)$, and $0$ otherwise. Let $\vartheta^{B}_r(z,v)=\vartheta_r(z,v)$ if $(z,v) \in \B_r(S)$, and $0$ otherwise. Let $\vartheta^{1}_r(z,v)=\vartheta_r(z,v)$ if $(z,v) \in \Col^{L}_r(S)$, and $0$ otherwise, and let  
$\vartheta^{2}_r(z,v)=\vartheta_r(z,v)$ if $(z,v) \in \Col^{R}_r(S)$, and $0$ otherwise. 
\vskip .1cm
Set 
$$
\alpha'(r)=3(\fta_r)_{*}(\varphi^{A}_r \, d\LM), \, \, \alpha_1(r)=3(\fta_r)_{*}(\varphi^{B}_r\, d\LM)+3(\tft_r)_{*}(\vartheta^{B}_r\, d\LM),
$$
\noindent
and
$$
\alpha''(r)=3(\tft^{L}_r)_{*}(\vartheta^{1}_r\, d\LM)+3(\tft^{R}_r)_{*}(\vartheta^{2}_r\, d\LM).
$$
\noindent
Set $\alpha(r)=\alpha'(r)+\alpha''(r)$. In addition, let  
$$
\beta(r)=3(\fta_r)_{*}((\varphi_r-\varphi^{A}_r-\varphi^{B}_r)\, d\LM)+ 3(\tft_r)_{*}(\vartheta_r-\vartheta^{B}_r)+
$$
$$
+3(\tft^{L}_r)_{*}( (\vartheta_r-\vartheta^{1}_r)\, d\LM)+3(\tft^{R}_r)_{*}( (\vartheta_r-\vartheta^{2}_r)\, d\LM).
$$
\noindent
We have
$$
\alpha(r)+\alpha_1(r)+\beta(r)=3(\fta_r)_{*}(\varphi_r \, d\LM)+3(\tft_r)_{*}(\vartheta_r \, d\LM)+3(\tft^{L}_r)_{*}(\vartheta_r \, d\LM)+3(\tft^{R}_r)_{*}(\vartheta_r \, d\LM),
$$
\noindent
so we conclude from (\ref{boundary-measure-0}) and (\ref{boundary-measure-1}) that
$\ph\mu(r)=\alpha(r)+\alpha_1(r)+\beta(r)$. In order to prove Theorem 4.1, it remains to prove that the measures $\alpha(r)$, $\alpha_1(r)$, and $\beta(r)$ satisfy the corresponding properties.
\vskip .1cm 
We first show that $\alpha'(r)$ is $re^{-r}$-symmetric. Let $\gamma \in \Gamma(G)$ and let $\gamma^{*}_1$ and $\gamma^{*}_2$ denote the two orientations on $\gamma$. Let $Y_i(\gamma) \subset \A_r(G)$, $i=1,2$, so that $(z,v) \in Y_i(\gamma)$ if $\fta_r(z,v) \in \NB\gamma^{*}_i$. By the 
definition of the set $\A_r(G)$ we have that $Y_1(\gamma)$ and $Y_2(\gamma)$ are disjoint, and $\refl(Y_1(\gamma))=Y_2(\gamma)$. For $(z,v) \in Y_1(\gamma)$ set $\refl(z,v)=(z_1,v_1)$. Combining  Proposition 5.5 and Proposition 5.8, we have that 
\begin{align*}
\dis(\fta_r(z,v),\fta_r(z_1,v_1)) & \le \dis(\fta_r(z,v),\ft_{\abs^{1}_{r}(z,v)}(z))+\dis(\fta_r(z_1,v_1),\ft_{\abs^{1}_{r}(z,v)}(z_1)) \\
& \le Ce^{-r}+Ce^{-r}<re^{-r},
\end{align*}
\noindent
for $r$ large enough. From the property $(3)$ of Proposition 4.1,
and since $\Jac(\refl)=1$ we see that the measures $(\fta_r)_{*}\nu_1$ and $(\fta_r)_{*}\nu_2$ are 
$re^{-r}$-equivalent on $\gamma$, where $\nu_i$ is the restriction of the measure $3(\fta_r)_{*}(\varphi_r \, d\LM)$ on $Y_i(\gamma)$. This shows that $\alpha'(r)$ is $re^{-r}$-symmetric. 
\vskip .1cm
Next, we show that $\alpha''(r)$ is $e^{-r}$-symmetric (and therefore this measure is $re^{-r}$-symmetric). Let $(z,v) \in \Col_r(G)$. By Proposition 5.9 we have that $(z,v) \in \Col^{L}_r(G)$ if and only if $\inv^{L}(z,v) \in \Col^{R}_r(G)$. Also, for almost every such $(z,v)$ we have that $\abs^{1}_r(z,v)$,  $\abs^{1}_r(\inv^{L}(z,v))$, and $\abs^{1}_r(\inv^{R}(z,v))$ are well defined. We only need to consider such points. 
Then, by Proposition 5.11 we have that  $\dis(\tft^{L}_r(z,v),\tft^{R}_r(\inv^{L}(z,v))) < e^{-r}$ for $r$ large enough. 
\vskip .1cm
Let $\gamma^{*}$ be the orientation on $\gamma=\taso^{1}_r(z,v)=\taso^{1}_r(\omega(\inv^{L}(z,v)))$ so that the endpoint $\abs_r(z,v)$ comes before the endpoint $\taso_r(z,v)$ on $\gamma^{*}$. Let $Y_1(\gamma) \subset \Col^{L}_r(G)$ so that $(z,v) \in Y_1$ if $\tft^{L}_r(z,v) \in \NB\gamma^{*}$. Set $Y_2(\gamma)=\inv^{L}(Y_1(\gamma))$. We have $Y_2(\gamma) \subset \Col^{R}_r(G)$ and for $(w,u) \in Y_2(\gamma)$ we have $\tft^{R}_r(w,u) \in \NB(-\gamma^{*})$.
Since  $\Jac(\inv^{L})=1$ and from the property $(3)$ of Proposition 4.1,  we see that the measures $(\tft^{L}_r)_{*}\nu_1$ and $(\tft^{R}_r)_{*}\nu_2$ are  $e^{-r}$-equivalent on $\gamma$. Here $\nu_1$ is the restriction of the measure $\vartheta_r\, d\LM$ on $Y_1(\gamma)$, and  $\nu_2$ is the restriction of the measure $\vartheta_r\, d\LM$ on $Y_2(\gamma)$.
\vskip .1cm
Since both measures $\alpha'(r)$ and $\alpha''(r)$ are  $re^{-r}$-symmetric, we see that $\alpha(r)$ is $re^{-r}$-symmetric. 
\vskip .1cm
From Proposition 5.10 we have that if $(z,v) \in \B_r(S)$ then $\h(z) \ge 2r-\log 3$. For $r$ large enough, we have
$$
\int\limits_{\NBG(S)}d\alpha_1(r) \le 3\int\limits_{\B_{r}(S) }(\varphi_r+\vartheta_r)\, d\LM \le 6 \LM(\TB{S} \setminus \TB\Thick_S(2r-\log 3) ) \le e^{-r} .
$$
\noindent
The following lemma will be proved in the next section.
\begin{lemma} The measure $\alpha_1(r)$ is $Q$-symmetric, for any  $Q>200$.
\end{lemma}
It remains to analyse $\beta(r)$. From the definition of $\beta(r)$ (and from (\ref{B-subset-C}) ) 
we see that if either of the points $\fta_r(z,v)$, $\tft_r(z,v)$, $\tft^{L}_r(z,v)$, or $\tft^{R}_r(z,v)$ belongs to the support of $\beta(r)$ 
we have that $(z,v) \in H_r$ (recall the definition of $H_r$ from Lemma 5.3).  
Note that $\vartheta_r(z,v) \le \varphi_r(z,v)$ for every $(z,v) \in \TB{S}$.  From Lemma 5.2, Lemma 5.3, and Proposition 5.12, we have
$$
\int\limits_{\NBG(S)}(\cl(\gamma^{*},z)+1) d\beta(r) \le 
5\int\limits_{H_{r}}(\cl_r(z,v)+\cl_r(z,\omega v)+1) \varphi_r(z) \, d\LM +
$$
$$
+5\int\limits_{\TB{S} \setminus \FD_r(S)}(\cl_r(z,v)+\cl_r(z,\omega v)+1) \varphi_r(z) \, d\LM \le P(r)e^{-r},
$$
\noindent
for $r$ large enough. This completes the proof of Theorem 4.1.
\end{proof}

\section{The proof of Lemma 5.4} 

\subsection{Preliminary propositions} We have the following preliminary propositions.

\begin{proposition} Let $(z,v) \in \TB\Ha$ and set  $(z_1,v_1)=\refl(z,v)$. Let $p \in \partial{\Ha}$. 
Set $w=\ft_{\gamma}(z)=\ft_{\gamma}(z_1)$ and  denote by $u$ a unit vector at $w$ that is tangent to $\gamma$. 
Then
$$
|\Delta_{p}[z,z_1]| < 2|\ang_{p}(w,u)| \sinh(\log \sqrt{3}).
$$
\end{proposition}

\begin{proof} Recall that $\dis(w,z)=\dis(w,z_1)=\log \sqrt{3}$. Let $f \in \Mob$ be the rotation centred at $w$, and so that $\ang_{p}(f(w,u))=0$. Then $f$ is a rotation for the angle $|\ang_{p}(w,u)|$. Let $(z',v')=f(z,v)$ and $(z'_1,v'_1)=f(z_1,v_1)=\refl(z',v')$ (here we use that $\refl$ commutes with $f$). Then $\Delta_{p}[z',z'_1]=0$.

\vskip .1cm

Recall $\dis(w,z)={{1}\over{2}}\log 3$. Using the formula that says that the hyperbolic circumference of the circle of radius $s$ is $2\pi \sinh(s)$ we  
conclude that

$$
\dis(z,z'), \dis(z_1,z'_1) < |\ang_{p}(w,u)| \sinh(\log \sqrt{3}).
$$

\noindent

Since $\Delta_{p}[z',z'_1]=0$ and from the previous inequality we get

$$
|\Delta_{p}[z,z_1]| < |\Delta_{p}[z',z'_1]|+ \dis(z,z')+\dis(z_1,z'_1) < 2|\ang_{p}(w,u)| \sinh(\log \sqrt{3}).
$$

\end{proof}

Let  $\eta$ be a geodesic in $\Ha$ and  let $\phi:\R \to \eta $ be the natural parametrisation of $\eta$ so that $\phi(0)=\z_{\max}(\eta,\infty)$, and so that $\infty$ is to the left of $\eta^{*}$, where $\eta^{*}$ is the push-forward of the orientation on $\R$ by the map $\phi$. Let $(z_0,v_0) \in \Ha$ be the unique point so that $\p^{1}(z_0,v_0)=\eta$, and so that $\ang_{\infty}(z_0,\omega^{2} v_0)=0$. Set $\tra_t(z_0,v_0)=(z_t,v_t)$.
Let $(z'_0,v'_0)=\refl(z_0,v_0)$, and set $\tra_{(-t)}(z'_0,v'_0)=(z'_t,v'_t)$. Then $\ft_{\eta}(z'_{t})=\ft_{\eta}(z_{t})=\phi(t) \in \eta$.

\begin{definition} Let $q(t)=\h_{\infty}(z_t)$, and $p(t)=\h_{\infty}(z'_t)$.
\end{definition}

\begin{proposition} We have $0<q(t)-p(t) \le 2\pi\sinh(\log \sqrt{3}) e^{-|t|}$, and \newline 
$|q(t)-(q(0)-|t|)|<\log 6$.
\end{proposition}

\begin{proof} Note that $q(t)-p(t)=\h_c(z_t)-\h_c(z'_t)=\Delta_c[z_t,z'_t]$. Let $w_t=\ft_{\eta}(z_t)=\ft_{\eta}(z'_{t})$, and let $u_t$ be the unit vector so that $0<|\ang_c(w_t,u_t)| \le {{\pi}\over{2}}$. Such  vector exists because $c$ is not an endpoint of $\eta$ (when $t=0$ there are two such vectors and we choose either one). Also, $\dis(w_0,w_t)=|t|$. From Proposition 6.1 we have

$$
q(t)-p(t)< 2|\ang_{c}(w_t,u_t)| \sinh(\log \sqrt{3}).
$$

The identity (\ref{ang-estimate-2}) yields
$$
|t|=\log \left(  \csc(\ang_c(w_t,u_t))+\cot(\ang_c(w_t,u_t)) \right),
$$
\noindent
and therefore, we have
$$
{{\sin(\ang_c(w_t,u_t))}\over{2}} \le e^{-|t|} \le \sin(\ang_c(w_t,u_t)).
$$
\noindent
Since for $0 \le \theta \le {{\pi}\over{2}}$, we have $\theta \le {{\pi}\over{2}} \sin \theta$, we find $|\ang_{c}(w_t,u_t)| \le \pi e^{-|t|} $. This shows
$$
q(t)-p(t)<\pi e^{-|t|} 2\sinh(\log \sqrt{3})=2\pi\sinh(\log \sqrt{3})e^{-|t|},
$$
\noindent
which proves the first inequality. 
\vskip .1cm
It follows from Proposition 5.2 that  $\h_c(w_0)-\h_c(w_t)>|t|-\log 2$. Since $\dis(w_t,z_t)=\log \sqrt{3}$, we have $q(0)-q(t)>|t|-\log 2-\log 3$, which proves $|q(t)-(q(0)-|t|)|<\log 6$.

\end{proof}

\begin{proposition} Let $T\ge 0$ and let
$$
\delta_T(\eta)=\int\limits_{-T}^{T} \left( e^{ {{2r-p(t)}\over{2}} }-e^{ {{2r-q(t)}\over{2}} } \right) \, dt.
$$
\noindent
Then 
$$
\delta_T(\eta) \le 24(e^{{{\pi}\over{\sqrt{3}}} } -1)e^{ {{2r-q(0)}\over{2}} } 
$$
\end{proposition}

\begin{proof} We compute 
$$
\delta_T(\eta) \le \int\limits_{-\infty}^{\infty} \left( e^{ {{2r-p(t)}\over{2}} }-e^{ {{2r-q(t)}\over{2}} } \right) \, dt=
2e^{ {{2r-q(0)}\over{2}} }  \int\limits_{0}^{\infty} e^{{{q(0)-q(t)}\over{2}} }\big(e^{ {{q(t)-p(t)}\over{2}} }-1\big)  \, dt.
$$
\noindent
We apply Proposition 6.2 to the right hand side of the above inequality and since $e^{kx}-1 \le (e^{k}-1)x$  for $x \in [0,1]$, we
get
$$
\delta_T(\eta) \le 2e^{ {{2r-q(0)}\over{2}} }  \int\limits_{0}^{\infty} 6 e^{ {{t}\over{2}} }\big(e^{ {{\pi e^{-t} }\over{\sqrt{3} }} }-1\big)  \, dt \le
12\big(e^{ {{ \pi}\over{\sqrt{3} }} }-1\big) e^{ {{2r-q(0)}\over{2}} } \int\limits_{0}^{\infty}  e^{ {{t}\over{2}} } e^{-t }  \, dt. 
$$
\noindent
The identity
$$
\int\limits_{0}^{\infty}  e^{ -{{t}\over{2}} }  \, dt=2,
$$
\noindent
completes the proof of the proposition.
\end{proof}

\begin{proposition} For any $T \ge 0$ let $\mu^{+}_{\eta,T}$ and $\mu^{-}_{\eta,T}$ be the measures on $\R$ that are supported on $[-T,T]$ and given by
$$
\mu^{+}_{\eta,T}=e^{ {{2r-q(t)}\over{2}} }\, dt,
$$
\noindent
and
$$
\mu^{-}_{\eta,T}=e^{ {{2r-p(t)}\over{2}} }\, dt,
$$
\noindent
Let $\delta_{\eta,T}$ be the measure on $\R$ that is supported at the point $0$ and such that 

$$
\delta_{\eta,T}(\R)=\delta_{\eta,T}(\{0\})=\mu^{-}_{\eta,T}(\R)-\mu^{+}_{\eta,T}(\R).
$$
\noindent
Then the measures $(\mu^{+}_{\eta,T}+\delta_{\eta,T})$ and $\mu^{-}_{\eta,T}$ are $Q$-equivalent for $Q=24\big(e^{ {{\pi}\over{\sqrt{3}}} }-1)$.

\end{proposition}

\begin{proof} We write $\mu^{+}$, $\mu^{-}$ and $\delta$ for $\mu^{+}_{\eta,T}$, $\mu^{-}_{\eta,T}$ and  $\delta_{\eta,T}$ respectively.
To prove that the corresponding measures are $Q$-equivalent we use Proposition 4.1. Observe that for any $t \in \R$,
\begin{equation}\label{1}
\mu^{+}(-\infty,t] \le \mu^{-}(-\infty,t],
\end{equation}
\noindent
and 
\begin{equation}\label{2}
\delta(-\infty,t] \le Qe^{ {{2r-q(0)} \over{2}}}.
\end{equation}
\noindent
Moreover when $-T \le t \le t+Q \le T$ we have
\begin{equation}\label{3}
\mu^{-}(t,t+Q) \ge  \int\limits_{t}^{t+Q} e^{ {{2r-q(0)} \over{2}}} \, ds=Qe^{ {{2r-q(0)} \over{2}}}.
\end{equation}
\noindent
Finally by the definition of $\delta$ we have
\begin{equation}\label{4}
(\mu^{+}+\delta)(\R)= \mu^{-}(\R).
\end{equation}
\vskip .1cm
We now show that for any $t \in \R$ we have
\begin{equation}\label{5}
(\mu^{+}+\delta)(-\infty,t]\le  \mu^{-}(-\infty,t+Q).
\end{equation}
\noindent
For $t \le -T$ the left hand side of the above inequality is $0$. For $-T \le t <t+Q \le T$ the inequality follows from (\ref{1}), (\ref{2}) and (\ref{3}).
For $t>T$ the inequality follows from (\ref{4}). This shows (\ref{5}).
\vskip .1cm
By $s \to -s$ symmetry we have 
\begin{equation}\label{6}
\mu^{+}+\delta)[-t,\infty) \le \mu^{-}(-t-Q,\infty),
\end{equation}
\noindent
and replacing $-t$ with $t+Q$ we obtain
\begin{equation}\label{7}
(\mu^{+}+\delta)[t+Q,\infty) \le  \mu^{-}(t,\infty).
\end{equation}
\noindent
Thus by (\ref{4}) we have
\begin{equation}\label{8}
\mu^{-}(-\infty,t] \le  (\mu^{+}+\delta)(-\infty,t+Q).
\end{equation}
\noindent
We have verified the hypothesis from Proposition 4.1 and this  proves the proposition.

\end{proof}

\subsection{The proof of Lemma 5.4}  Let $r>2$ and let  $\gamma \in \Gamma(G)$. We need to show that the restriction of the measure $\alpha_1(r)$ to $\NB\gamma$ is $Q$-symmetric. For $r$ large, we have that the support of the measure $\alpha_1(r)$ is deep into the thin part of $\NB\gamma$. In fact, it follows from Proposition 5.10 that $\B_r(G) \subset \TB\Thin_G(2r-\log 3)$. Note that for $(z,v) \in \B_r(G)$ we have that 
$\h(\ft_{\p^{1}(z,v)}(z) \ge 2r-\log 3$. Combining this with Proposition 5.5, we have that for $r$ large enough, the
restriction of the measure $\alpha_1(r)$ to $\NB\gamma$ is supported in $\NB\gamma \cap (\TB\Ha \setminus \TB\Thick_G(2r-\log 3-1))$.
Therefore, for each cusp $c \in \Cus(G)$ we consider the restriction of the measure $\alpha_1(r)$ to $\NB\gamma \cap \TB\Hor_{c}(2r-\log 3-1)$ which we denote by $\alpha_1(\gamma,c,r)$. To prove Lemma 5.4, it is enough to show that each $\alpha_1(\gamma,c,r)$ is $Q$-symmetric. There are only finitely many cusps $c \in \Cus(G)$, so that $\alpha_1(\gamma,c,r)$ is a non-zero measure,  and the restriction of $\alpha_1(r)$ to $\NB\gamma$ is the finite sum of these measures $\alpha_1(\gamma,c,r)$.  Note that if $c$ is an endpoint of $\gamma$ then $\alpha^{c}_1(r)$ is the zero measure.
\vskip .1cm
Let $G=G_c$ (then $c=\infty$).  We say that $(z,v) \in \B^{*}_r(\gamma,c)$, where $*\in \{+,- \}$, if the following holds
\begin{itemize}
\item We have $(z,v) \in \B_r(G)$.
\item We have $(z,v) \in (\TB\Ha)^{*}$. 
\item $\fta_r(z,v) \in \NB(\gamma \cap \Hor_c(1))$. 
\end{itemize}

We say that $(z_0,v_0) \in U^{+} \subset \B^{+}_r(\gamma,c)$, if $\ang_{c}(z_0,\omega^{2} v_0)=0$. Set $(z'_0,v'_0)=\refl(z_0,v_0)$, and let  $U^{-}=\refl(U^{+})$.  By definition of $\B_r(G)$, we have $U^{-} \subset \B^{-}_r(\gamma,c)$. 
For $(z_0,v_0) \in \U^{+}$ and $t \in \R$ set  $(z_t,v_t)=\tra_t(z_0,v_0)$. Also $(z'_t,v'_t)=\tra_{(-t)}(z'_0,v'_0)$.
We have

\begin{proposition} Fix $(z_0,v_0) \in U^{+}$. Then there exists $T=T(z_0,v_0) \ge 0$ such that 
the set $t \in \R: \, (z_t,v_t) \in \B^{+}_r(\gamma,c)$ is a symmetric interval $(-T,T)$ or $[-T,T]$.
\end{proposition}

\begin{proof} Observe that $\p(z_0,v_0)=\p(z_t,v_t)$, for any $t \in \R$. The condition $(z_t,v_t) \in \B^{+}_r(\gamma,c)$ it equivalent with
the following two conditions. The first one is that $c$ is not an endpoint of $\abs^{1}_r(z_t,v_t)$ or of $\abs^{1}_r(z'_t,v'_t)$.
The second one is that $(z_t,v_t) \in \big(\Hor_c(2r)\cup \refl(\Hor_c(2r))\big)$. One directly verifies that each of these two conditions is satisfied on a symmetric interval.
\end{proof}

\begin{proposition} There exists $r_0>0$, so that for $r>r_0$, the following holds.  Let $(z_0,v_0) \in U^{+}$. 
For every $t \in (-T(z_0,v_0),T(z_0,v_0))$, we have 
$$
\dis(\psi(t),\fta_r(z_t,v_t)), \dis(\psi(t),\fta_r(z'_{t},v'_{t} ) ) \le re^{-r},
$$
\noindent
where $\psi:\R \to \gamma$ is the natural parametrisation of $\gamma$ such that $\psi(0)=\z_{\max}(\gamma,c)$.
\end{proposition}

\begin{proof}  Assume that $G=G_c$ (then $c=\infty$). Since $(z_0,v_0) \in \B_r(G)$, and since $(z_0,v_0)$ is above $\p^{1}(z_0,v_0)$, we have that $\h_{\infty}(z_0) \ge 2r>r+2$. By Proposition 5.6 we have that $\Hor_{\abs_{r}(z_{0},v_{0})}(1)$, is the first $1$-horoball that the geodesic  ray $\gamma_{(z_{0},v_{0})}$ enters after leaving $\Hor_{\infty}(1)$. Since with this normalisation the Euclidean diameter of any $1$-horoball (except the one based at $\infty$) is at most $1$, and since the ray $\gamma_{(z_{0},v_{0})}$ ends at $\p(z_0,v_0)$, 
we have $|\abs_r(z_0,v_0)-\p(z_0,v_0)|<1$.  Similarly, $|\abs_r(z_0,\omega v_0)-\p(z_0,\omega v_0)|<1$. 
From $\h_{\infty}(z_0) \ge 2r$, we have that 
$$
|\p(z_0,v_0)-\p(z_0,\omega v_0)| \ge 2e^{-\h_{\infty}(z_{0})-\log \sqrt{3}}>e^{2r }.
$$ 
Let $f \in \Mobi$, be the unique M\"obius transformation so that $f(\p^{1}(z_0,v_0))=\gamma$ (note that $\ft_{\gamma}(f(z_0))=\psi(0)$). 
Then for $\zeta \in \Ha$, we have $f(\zeta)=l_1\zeta+l_2$, where $l_1>0$, and $l_2 \in \R$. The coefficients $l_1,l_2$, are determined by the conditions $f(\p(z_0,v_0))=\abs_r(z_0,v_0)$, and $f(\p(z_0,\omega v_0))=\abs_r(z_0,\omega v_0)$.
We have 
$$
|\log l_1| \le \log(1+2e^{-2r}) < 2e^{-2r}
$$
\noindent
and $|l_2|<1$.  Since $\h_{\infty}(z_t),\h_{\infty}(z'_t) >2r-\log 3$, for   $t \in [-T(z_0,v_0),T(z_0,v_0)]$ we have that
$$
\dis(z_t,f(z_t)),\dis(z'_t,f(z'_t)) \le 2e^{-2r}+3e^{-2r}=5e^{-2r}.
$$
\noindent
Since $\ft_{\gamma}(f(z_t))=\psi(t)$ and since $\ft_{\gamma}$ does not increase the hyperbolic distance we obtain 
$$
\dis(\ft_{\gamma}(z_t) ,\psi(t)),\dis(\ft_{\gamma}(z'_t),\psi(t)) \le 2e^{-2r}+3e^{-2r}=5e^{-2r}.
$$
\noindent
Together with Proposition 5.8 this completes the proof.
\end{proof}

Let $(z_0,v_0) \in U^{+}$. Define the mappings $f^{+/-}_{(z_{0},v_{0})}:(-T(z_0,v_0),T(z_0,v_0)) \to \NB\gamma$,  by
$$
f^{+}_{(z_{0},v_{0})}(t)=\fta_r(z_t,v_t), \quad \quad  f^{-}_{(z_{0},v_{0})}(t)=\fta_r(z'_t,v'_t),
$$
\noindent
and let
$$
\mu^{+/-}_{(z_{0},v_{0})}=(f^{+/-}_{(z_{0},v_{0})})_{*}\big( \mu^{+/-}_{\p^{1}(z_{0},v_{0}),T(z_{0},v_{0})} \big),
$$
\noindent
where the measures $\mu^{+/-}_{\eta,t}$ were defined in Proposition 6.4. We have that $\mu^{+/-}_{(z_{0},v_{0})} \in \Mes(\NB\gamma)$.
Note that $\mu^{+}_{(z_{0},v_{0})}$ is supported at the point $(\psi(t),\vec{n}(t))$, where $\vec{n}(t)$ is the unit normal vector at $z$ that points towards the cusp $c$ (so $|\ang_c(\psi(t),\vec{n}(t))| <{{\pi}\over{2}}$). The measure $\mu^{-}_{(z_{0},v_{0})}$ is supported at the point $(\psi(t),\vec{n}(t))$, where $\vec{n(t)}$ is the unit normal vector at $z$ that points away from the cusp $c$ (so $|\ang_c(\psi(t),\vec{n}(t))| >{{\pi}\over{2}}$).
\vskip .1cm
Let $\delta_{(z_{0},v_{0})}$ be the atomic measure that on $\NB\gamma$ that is supported at the point $(\psi(0),\vec{n}_0)$, where
$\ang_c(\psi(0),\vec{n}_0)=0$, and the mass of $\delta_{(z_{0},v_{0})}$ is  $\delta_{T(z_{0},v_{0})}(\p^{1}(z_0,v_0))$ (recall the definition of $\delta_T(\eta)$ from Proposition 6.3).

\begin{proposition}  For every $(z_0,v_0) \in U^{+}$ the measure $(\mu^{+}_{(z_{0},v_{0})}+\delta_{(z_{0},v_{0})})+\mu^{-}_{(z_{0},v_{0}}$ is $Q+1$-symmetric.
\end{proposition}

\begin{proof} We can by abuse of notation write $\mu^{+/-}_{(z_{0},v_{0})}$ and $\delta_{(z_{0},v_{0})}$ as measures on $\gamma$ rather than on $\NB\gamma$. Then we must show that $(\mu^{+}_{(z_{0},v_{0})}+\delta_{(z_{0},v_{0})})$ and $\mu^{-}_{(z_{0},v_{0}}$ are $Q+1$-equivalent.
\vskip .1cm
We let $\eta=\p^{1}(z_0,v_0)$ and $T=T(z_0,v_0)$. By the previous proposition for $t \in (-T,T)$ we have
$\dis(f^{+/-}_{(z_{0},v_{0}}(t),\psi(t)) \le re^{-r} <{{1}\over{2}}$, for $r$ large enough. Therefore by Proposition 4.1 the measures
$\mu^{+/-}_{(z_{0},v_{0})}$ and $\psi_{*} \mu^{+/-}_{\eta,T}$ are ${{1}\over{2}}$-equivalent. Also $\delta_{(z_{0},v_{0})}=\psi_{*}\delta_{\eta,T}$.
The proof now follows from Proposition 6.4 and Proposition 4.6.

\end{proof}

Let $E$ be a Borel subset of $U^{+}$. For $t_1,t_2 \in \R$, $t_1 \le t_2$, let
$$
E(t_1,t_2)=\{\tra_t(z,v):(z,v) \in U^{+},t_1 \le t \le t_2\}.
$$
\noindent
Since $\tra_t$ is a measure preserving flow (one-parameter Abelian group)  and since $\tra_t(U^{+}) \cap U^{+}=\emptyset$, there exists a non-negative constant $\sigma^{+}(E)$ that depends only on $E$, so that $\LM(E(t_1,t_2))=\sigma^{+}(E)(t_2-t_1)$. The constant $\sigma^{+}(E)$ is called the cross sectional area of the set $E$. This way we construct a positive Borel measure $\sigma^{+}$ on $U^{+}$. Similarly, we get the measure $\sigma^{-}$ on $U^{-}$. If $E \subset U^{+}$, then $\refl(E) \subset U^{-}$, and $\sigma^{+}(E)=\sigma^{-}(\refl(E))$ (because $\Jac(\refl)=1$, and $\refl \circ \tra_t=\tra_{(-t)} \circ \refl$). This shows that the restriction of the map $\refl:U^{+} \to U^{-}$, satisfies that
$\refl^{*}\sigma^{-}=\sigma^{+}$.  If $\sigma^{+}(U^{+})=0$, then the measure $\alpha^{c}_1(r)$ is the zero measure. 
\vskip .1cm
It follows from the definition of $\alpha_1(r)$ and from (\ref{B-subset-C})  that
$$
\alpha_1(\gamma,c,r)=3\int\limits_{U^{+}} \big( (\mu^{+}_{(z_{0},v_{0})}+\delta_{(z_{0},v_{0})})+\mu^{-}_{(z_{0},v_{0}} \big) \, d\sigma^{+}(z_0,v_0).
$$
\noindent
Then Lemma 5.4 follows from the previous proposition and Proposition 4.7. One sees from Proposition 6.4 that $Q+1 \le 200$.

\section{Estimates on the combinatorial length}
The following definition and propositions  will be used throughout this section.
Recall that for $(z,v) \in \TB\Ha$ the map $\gamma_{(z,v)}:[0,\infty) \to \Ha$ is the natural parametrisation of the geodesic ray that starts at $z$ and that is tangent to the vector $v$ at $z$ that is $\gamma'_{(z,v)}(0)=v$. By $\gamma_{(z,v)}[0,t]$ we denote the corresponding geodesic segment.

\begin{definition} Let $(z,v) \in \TB\Ha$ and $t \in \R$. The geometric intersection number $\iota(\gamma_{(z,v)}(t),\tau(G))$ is defined as the number of (transverse) intersections between the geodesic segment $\gamma_{(z,v)}[0,t]$ and the edges from $\lambda(G)$.  For $(z,v) \in \TB{S}$ the intersection number $\iota(\gamma_{(z,v)}(t),\tau(S))$ is defined in the same way. 
\end{definition}

Next we establish the connection between the $r$-combinatorial length $\cl_r(z,v)$ (Definition 5.7) and the intersection number $\iota(\gamma_{(z,v)}(\at_r(z,v)),\tau(S))$.

\begin{proposition} Let $(z,v) \in \TB{S}$ such that $0<\at_r(z,v)<\infty$. Then 
$$
\cl_r(z,v) \le  \iota(\gamma_{(z,v)}(\at_r(z,v)),\tau(S)).
$$
\end{proposition}

\begin{proof} Let $(z,v)$ also denotes a lift of $(z,v)$ in $\TB\Ha$. Let $\gamma$ be the geodesic ray that connects $z$ and $\abs_r(z,v) \in \Cus(G)$.
The assumption $\at_r(z,v)>0$ implies  that $z$ does not belong to the horoball $\Hor_{\abs_{r}(z,v) }(1)$ and that is why $\cl_r(z,v)=\iota(\gamma,\tau(G))$. Let $z_1$ be the intersection point between $\gamma$ and the horocircle
$\partial{ \Hor_{\abs_{r}(z,v)}(1) }$  (since $\gamma$ ends at $\abs_r(z,v)$ there is exactly only one such intersection point $z_1$). 
Let $\gamma_1$ be the geodesic sub-segment of $\gamma$ bounded by the points $z$ and $z_1$. By definition the point $\gamma_{(z,v)}(\at_r(z,v))$ also belongs to the horocircle $\partial{\Hor_{\abs_{r}(z,v)}(1)}$. Let $G=G_{\abs_{r}(z,v)}$.  Then the Euclidean distance between the points $z_1$ and $\gamma_{(z,v)}(\at_{r}(z,v))$ is at most $1$ and the corresponding $y$-coordinate of both these points is equal to $e$. 
\vskip .1cm
Let $\alpha$ be an edge from $\lambda(G)$ such that  $\gamma$ intersects $\alpha$. Then $\alpha$ does not end at $\infty$ (since $\gamma$ ends at $\infty$ we see that $\gamma$ can not (transversely) intersect any other geodesic that ends at $\infty$). Moreover we have that the segment $\gamma_1$ intersects $\alpha$.  But then $\alpha$ either intersects intersects the segment $\gamma_{(z,v)}[0,\at_r(z,v)]$ or the horocyclic segment between the points $z_1$ and  $\gamma_{(z,v)}(\at_r(z,v))$. Every point on this horocyclic segment has the $y$-coordinate equal to $e$ so we conclude that if $\alpha$ intersects this horocyclic segment it has to end at $\infty$. Therefore we find that $\alpha$ intersects the segment $\gamma_{(z,v)}[0,\at_r(z,v)]$.
This proves the proposition.
\end{proof}

\begin{proposition} Let $(z,v) \in \TB{S}$. There exists a constant $K(S)>0$ that depends only on $S$ such that is $z \in \Thick_S(1)$ then 
$\iota(\gamma_{(z,v)}(1),\tau(S)) < K(S)$. If $z \in (\TB{S} \setminus \Thick_S(1))$ then $ \iota(\gamma_{(z,v)}(1),\tau(S)) \le N(S)e^{\h(z) }$. In particular for any $(z,v) \in \TB{S}$ we have $\iota(\gamma_{(z,v)}(1),\tau(S)) \le \max \{K(S), N(S)e^{|\h(z)|} \}$. 
\end{proposition}

\begin{proof}

If $z \in \Thick_S(1)$ then the geodesic segment remains in $\Thick_S(2)$. Since the space of geodesic segments of the hyperbolic length $1$ that remain 
in  $\Thick_S(2)$ is compact, we find that there exists $K(S)>0$ so that every such segment intersects at most $K(S)$ edges from $\lambda(S)$.
\vskip .1cm
Assume that $z$ does not belong to $\Thick_S(1)$ (this equivalent to saying that $\h(z)>1$). Then $z \in \Hor_{c_{i}(S)}(1)$ for exactly one cusp $c_i(S) \in \Cus(S)$. Let $G=G_{c_{i}(S)}$. Let $(z,v) \in \TB\Ha$ also denote a lift of $(z,v)$ so that $z \in \Hor_{\infty}(1)$. Then the geodesic segment  $\gamma_{(z,v)}[0,1]$ can intersect only edges from $\lambda(G)$ that end at $\infty$. This is why $\gamma_{(z,v)}[0,1]$ can intersect 
at most as many edges from $\lambda(G)$ as the horocyclic segment (horocyclic with respect to $\infty$) of the hyperbolic length $1$ that begins at $z$. Therefore $\gamma_{(z,v)}[0,1]$ intersects at most $N(S)e^{\h_{\infty}(z)}=N(S)e^{\h(z)}$ edges from $\tau(G)$. This is easily seen by observing that  $e^{\h_{\infty}(z) } \ge 1$ agrees with the $y$-coordinate of the point $z \in \Ha$.
\end{proof}

\begin{proposition} There exists a constant $D(S)>0$ (that depends only on $S$) and  $r_0>0$ so that for $r>r_0$ the following holds. 
Let $(z,v) \in \TB\Ha$. If $z \in \Thick_G(10r)$ then 
\begin{equation}\label{CL-estimate-thick}
\cl_r(z,v) \le  D(S)\big(e^{ 11r}+K(S)(\at_r(z,v)+1) \big).
\end{equation}
\noindent
If $z \in \Hor_c(10r)$ for some $c \in \Cus(G)$ then 

\begin{equation}\label{CL-estimate-thin}
\cl_r(z,v) \le  D(S) \big( \chi_r(z,v) e^{\h(z)}+K(S)(\at_r(z,v)+1) \big).
\end{equation}
\noindent
Here $\chi_r(z,v)=\csc|\ang_c(z,v)|$ if $|\ang_c(z,v)|< {{\pi}\over{2}}$ and  $\csc|\ang_c(z,v)| \le e^{r}$. Otherwise set $\chi_r(z,v)=1$.

\end{proposition}

\begin{remark} Note that if $\ang_{c_{1}}(z,v)$ gets smaller the upper bound in  (\ref{CL-estimate-thin}) increases. 
However once  $|\ang_{c_{1}}(z,v)|$ is small enough so that  $\sin |\ang_{c_{1}}(z,v)| \le e^{-r}$ then the upper bound 
in (\ref{CL-estimate-thin}) decreases sharply. In fact $\chi_r(z,v)$ takes values in the interval $[1,e^{r}]$.
\end{remark}

\begin{proof} Assume first that $\at_r(z,v)=0$. Then by definition $z \in \Hor_c(1)$ for some $c \in \Cus(G)$. Moreover we have   
$|\ang_c(z,v)|< {{\pi}\over{2}}$ and  $\csc(|\ang_c(z,v)|) \le e^{r}$ (see (\ref{ang-estimate-1})). In this case we have $\cl_r(z,v)=e^{\h(z)}$ because $z$ belongs to  the $1$-horoball at the cusp which absorbs $\gamma_{(z,v)}$. If  $z \in \Thick_G(10r)$ we have that $\cl_r(z,v) \le e^{10r}$ so the inequality (\ref{CL-estimate-thick}) holds for such $(z,v)$. If $(z,v) \in \Hor_c(10r)$ then (\ref{CL-estimate-thin}) holds as well.
From now on we assume that $\at_r(z,v)>0$. Then $\cl_r(z,v) \le \iota(\gamma_{(z,v)}(\at_r(z,v)),\tau(G) )$ by Proposition 7.1.
\vskip .1cm
Let $(z,v) \in \TB\Ha$ and let $t_{*}=\min \{(3r+\h(z)), \,  \at_r(z,v) \}$. By Proposition 5.6 we have that
the segment $\gamma_{(z,v)}[t_{*},\at_r(z,v)]$ is contained in $\Thick_G(1)$ and therefore by Proposition 7.2 we have 
\begin{equation}\label{thick-cl}
\iota(\gamma_{(z,v)}[t_{*},\at_r(z,v)],\tau(G)) \le K(S)(\at_r(z,v)+1)
\end{equation}
\noindent
It remains to estimate $\iota(\gamma_{(z,v)}[0,t_{*}],\tau(G))$. 
Fix $0 \le t <t_{*}$.  Let $t_0=0$, $t_k=t$, and let $\{t_1,t_2,...,t_{k-1} \}$ be the ordered set of points where the segment $\gamma_{(z,v)}[0,t]$ intersects  the $1$-horocircles. Note that there can be at most $2t_0+2$ such segments, that is $k \le 2t_0$.
Each segment $\gamma_{(z,v)}[t_i,t_{i+1}]$ is either contained in $\Thick_G(1)$ or in some $\Hor_c(1)$. If $\gamma_{(z,v)}[t_i,t_{i+1}]$ is  
contained in $\Thick_G(1)$ then by Proposition 7.2 we have 
\begin{equation}\label{ocena-nulta}
\iota(\gamma_{(z,v)}[t_i,t_{i+1}],\tau(G)) \le K(S)(|t_{i+1}-t_i|+1). 
\end{equation}
\vskip .1cm
Assume that $\gamma_{(z,v)}[t_i,t_{i+1}]$ is contained  in some $\Hor_c(1)$ (note that $\abs_r(z,v) \ne c$). 
There are two cases. The first one is when $z \in \Hor_c(1)$ (then clearly $i=0$). Then the entire segment  $\gamma_{(z,v)}[t_0,t_1]$ is contained in $\Hor_c(1)$. Let $z_0=\z_{\max}(\gamma_{(z,v)},c)$. Then the segment
$\gamma_{(z,v)}[t_0,t_1]$ can intersect at most $2N(S)e^{\h(z_{0})}$ edges from $\lambda(G)$. On the other hand from (\ref{ang-estimate-1}) we have that $\h_c(z_0)-\h_c(z) \le \log (\csc|\ang_c(z,v)|)$  if $|\ang_c(z,v)| \le {{\pi}\over{2}}$ 
and $\h_c(z_0)=\h_c(z)$ if $|\ang_c(z,v)|> {{\pi}\over{2}}$. This shows that
\begin{equation}\label{ocena-prva}
\iota(\gamma_{(z,v)}[t_0,t_1],\tau(G)) \le 2N(S)\chi_r(z,v)e^{\h(z)}.
\end{equation}
\noindent
For any other segment $\gamma_{(z,v)}[t_i,t_{i+1}]$ that is contained in some $\Hor_c(1)$ we have that $\h_c(z_0)-1=\h_c(z_0)-\h_c(\gamma_{(z,v)}(t_i))<r$ where
$z_0=\z_{\max}(\gamma_{(z,v)},c)$. This follows from the fact that $\abs_r(z,v) \ne c$, and from $\h_c(z_0)-\h_c(\gamma_{(z,v)}(t_i))<\h_c(z_0)-\h_c(z)$.
Then the segment $\gamma_{(z,v)}[t_i,t_{i+1}]$ can intersect at most $2N(S)e^{\h(z_{0})}$ edges from $\lambda(G)$, that is
\begin{equation}\label{ocena-druga}
\iota(\gamma_{(z,v)}[t_i,t_{i+1}],\tau(G)) \le 2N(S)e^{r+1}.
\end{equation}
\vskip .1cm
Since $t_0 \le 3r+\h(z)$ and from (\ref{ocena-nulta}), (\ref{ocena-prva}) and (\ref{ocena-druga})  we obtain 
$$
\iota(\gamma_{(z,v)}[0,t_{*}],\tau(G)) \le   2N(S)\chi_r(z,v)e^{\h(z)}+(2(3r+\h(z))+1)2N(S)e^{r+1}+(2(3r+\h(z))+1)K(S),
$$
\noindent
(here we also used that there are at most $2t_0+2$ segments $\gamma_{(z,v)}[t_i,t_{i+1}]$). 
\vskip .1cm
Let  $z \in \Thick_G(10r)$. If $z \in \Hor_c(1)$ then from  $1 \le \chi_r(z,v) \le r$  
we obtain 
$$
\iota(\gamma_{(z,v)}[0,t_{*}],\tau(G)) \le   2N(S)e^{11r}+(26r+1)2N(S)e^{r+1}+(26r+1)K(S).
$$
\noindent
Together with (\ref{thick-cl}) this proves (\ref{CL-estimate-thick}). 
\vskip .1cm
Suppose that $z \in \Hor_c(10r)$ for some $c \in \Cus(G)$. Then similarly from (\ref{ocena-nulta}), (\ref{ocena-prva}) and (\ref{ocena-druga})  we get

$$
\iota(\gamma_{(z,v)}[0,t_{*}],\tau(G)) \le   2N(S)\chi_r(z,v)e^{\h(z)}+(2(3r+\h(z))+1)2N(S)e^{r+1}+(2(3r+\h(z) )+1)K(S).
$$
\noindent
Since $\h(z) \ge 10r$ for $r$ large enough we have 
$$
(2(3r+\h(z))+1)2N(S)e^{r+1}+(2(3r+\h(z))+1)K(S) \le e^{\h(z)}.
$$
\noindent
Together with (\ref{thick-cl}) this proves (\ref{CL-estimate-thin}).
 
\end{proof}

\begin{proposition} Let $G$ be a normalised group and let $z \in \Hor_{\infty}(0)$, that is  $y \ge 1$. Let 
$z_1, z_2 \in \Ha $, where $z_j=x_j+iy_j$, $j=1,2$, and let  $\eta$ be the geodesic segment 
between $z_1$ and $z_2$. Moreover let $\eta_j$ be the vertical geodesic segment connecting $z_j$ with the point $x_j+iy$ (that is $\eta_j$ is orthogonal to the $0$-horocircle at $\infty$).
Then 
$$
\iota(\eta,\tau(G)) \le \iota(\eta_1,\tau(G))+\iota(\eta_2,\tau(G))+ N(S)|x_1-x_2|.
$$
\end{proposition}

\begin{proof} Let $\gamma \in \lambda(G)$. If $\gamma$ intersects $\eta$ then $\gamma$ either intersects $\eta_1$ or $\eta_2$ or the horocyclic segment between the points $x_1+iy$ and $x_2+iy$. Since $y \ge 1$ and since $G$ is normalised we have that if $\gamma$ intersects this horocyclic segment then $\gamma$ ends at $\infty$. On the other hand there are at most $N(S)|x_1-x_2|$ such vertical geodesics in $\lambda(G)$. This proves the proposition.

\end{proof}

\subsection{The proof of Lemma 5.2} 
Let $t \ge 0$ and set
$$
E_t=\{(z,v) \in \TB{S}: \, \at_r(z,v) \ge t\}.
$$
\noindent
We have

\begin{proposition} There exists $r_0>0$ so that for $r>r_0$ and every $t \ge r^{2}$ we have 
$$
\LM(E_t) \le e^{-q(S) {{t}\over{4}} },
$$
\noindent
where $q(S)>0$ is the constant from Lemma 5.1.
\end{proposition}

\begin{proof} Fix $t \ge r^{2}$. We have 
$$
\TB{S}=\TB\Thick_S( {{t}\over{2}} ) \bigcup ( \TB{S} \setminus \TB\Thick_S({{t}\over{2}}) ).
$$
\noindent
One finds that 
$$
\LM\big( (\TB{S} \setminus \TB\Thick_S({{t}\over{2}})) \big)=\numb e^{-{{t}\over{2}} },
$$
\noindent
where $\numb$ is the number of cusps in $\Cus(S)$.  Since $t \ge r^2$ for $r$ large enough we have
\begin{equation}\label{prva}
\LM\big( (\TB{S} \setminus \TB\Thick_S({{t}\over{2}})) \big)\le  {{1}\over{2}}e^{-{{t}\over{4}} }.
\end{equation}
\vskip .1cm
It follows from Proposition 5.6 that for every $(z,v) \in \TB{S}$ and every $3r+\h(z) \le t' \le t$ 
we have that $\flow_{t'}(z,v) \subset A(t-t')$, where $A(t)$ is the set defined in Lemma 5.1. If $(z,v) \in \TB\Thick_S({{t}\over{2}})$
then 
$$
\flow_{(3r+{{t}\over{2}} )}(z,v) \subset A({{t}\over{2}} -3r). 
$$
\noindent
It follows from Lemma 5.1 that
$$
\LM\big( E_t \cap \TB\Thick_S({{t}\over{2}}) \big)= \LM\big( \flow_{(3r+{{t}\over{2}} )} ( E_t \cap \TB\Thick_S({{t}\over{2}})) \big) \le
C(S)e^{ -q(S)({{t}\over{2}} -3r)}.
$$
\noindent
Since $t \ge r^2$ for $r$ large enough we have 
$$
\LM\big( E_t \cap \TB\Thick_S({{t}\over{2}}) \big)\le {{1}\over{2}}e^{-q(S){{t}\over{4}}}.
$$
\noindent
Together with (\ref{prva}) this proves the proposition.
\end{proof}

We are ready to prove Lemma 5.2.  Let 
$$
T_r=\{(z,v) \in \TB{S}: \, \h(z)\ge 10r \}.
$$
\noindent
We have $\TB{S} \setminus \FD_r(S)=E_{r^{2}} \cup T_r$. It follows from Proposition 7.3 and from the definition of $\varphi_r(z,v)$ that
$$
\int\limits_{\TB{S} \setminus \FD_r(S)} (\cl_r(z,v)+ \cl_r(z,\omega v)+\cl_r(z,\omega^{2} v)+1) \varphi_r(z) \, d\LM \le
\int\limits_{E_{r^{2}} \setminus T_{r}}  D(S)e^{11r} \, d\LM+
$$

$$
+\int\limits_{T_{r}} D(S) e^{\h(z)}\chi_r(z,v) e^{(r-{{\h(z)}\over{2}}) } \, d\LM+ \int\limits_{E_{r^{2}}} D(S)K(S)(\at_r(z,v)+1) \, d\LM +
$$
$$
+\int\limits_{T_r \setminus E_{r^{2}}} D(S)K(S)(\at_r(z,v)+1) \, d\LM.
$$
\noindent
We estimate each of the four integrals on the right hand side.
\vskip .1cm
From Proposition 7.5 we have 
$$
\int\limits_{E_{r^{2}} }  D(S)e^{4r} \, d\LM = D(S)e^{4r} \LM(E_{r^{2}}) \le D(S)e^{4r} e^{ -q(S){{r^{2}}\over{4}} } \le e^{-2r}=\pe(2r),
$$
\noindent
for $r$ large enough. 
\vskip .1cm

Next we estimate the second integral
$$
\int\limits_{T_{r}} D(S) e^{\h(z)}\chi(z,v) e^{(r-{{\h(z)}\over{2}}) } \, d\LM.
$$
\noindent
We have that $T_r=\cup_{c_{i}(S) \in \Cus(S)}\TB \Hor_{c_{i}(S)}(10r)$. Let $c_i \in \Cus(G)$ such that $[c_i]_G=c_i(S)$ and set $G=G_{c_{i}}$. 
Let $z \in \Hor_{c_{i}}(10r)$. Then $(z,v)=(x,y,\theta)=(z,\theta)$ in the polar coordinates on $\TB\Ha$. We have 
$$
\int\limits_{-\pi}^{\pi} \chi_r(z,\theta)\, d\theta \le Cr,
$$
\noindent
where $C$ is some universal constant. This shows that for $r$ large enough we have 
$$
\int\limits_{\TB\Hor_{c_{i}}} D(S) e^{\h(z)}\chi(z,v) e^{(r-{{\h(z)}\over{2}}) } \, d\LM 
\le Cr \int\limits_{e^{10r}}^{\infty} {{D(S)ye^{r}}\over{\sqrt{y}}} \, {{dy}\over{y^{2}}}=
$$
$$
= D(S)Cr e^{r}\int\limits_{e^{10r}}^{\infty} y^{-{{3}\over{2}}} \, dy \le e^{-4r}<e^{-r}.
$$
\vskip .1cm

Next we estimate the third integral. Recall that if $f:X \to [0,\infty)$ is an integrable function on a measure space $(X,\mu)$ then
$$
\int\limits_{X} f \, d\mu=\int\limits_{0}^{\infty} \mu(f^{-1}[t,\infty) ) \, dt.
$$
\noindent
Set $(X,\mu)=(E_{r^{2}},\LM)$ and $f(z,v)=\at_r(z,v)$. We find
$$
\int\limits_{E_{r^{2}}} \at_r(z,v) \, d\LM=\int\limits_{0}^{\infty} \LM(E_t \cap E_{r^{2}}) \, dt=r^{2}\LM(E_{r^{2}}) +
\int\limits_{r^{2}}^{\infty}\LM(E_t) \, dt.
$$
\noindent
This together with  Proposition 7.5 gives
$$
\int\limits_{E_{r^{2}}} D(S)K(S)(\at_r(z,v)+1) \, d\LM \le D(S)K(S)\LM(E_{r^{2}})+ D(S)K(S) \big( r^{2}\LM(E_{r^{2}}) +
\int\limits_{r^{2}}^{\infty}\LM(E_t) \, dt \big) \le
$$
$$
\le (1+r^{2})D(S)K(S)e^{-q(S) {{r^{2}}\over{4}} }+ D(S)K(S)\int\limits_{r^{2}}^{\infty}e^{-q(S) {{t}\over{4}} } \, dt,
$$
\noindent
so for $r$ large enough we have 
$$
\int\limits_{E_{r^{2}}} D(S)K(S)(\at_r(z,v)+1) \, d\LM \le e^{-2r}=\pe(2r).
$$
\vskip .1cm

If $(z,v) \in T_r \setminus E_{r^{2}}$ then $(\at_r(z,v)+1)\le r^2+1$. This implies that
$$
\int\limits_{T_r \setminus E_{r^{2}}} D(S)K(S)(\at_r(z,v)+1) \, d\LM \le D(S)K(S)(r^{2}+1) \LM(T_r) \le r^{3}e^{-10r}<e^{-r}.
$$
\noindent
We have estimated all four integrals and we find that
$$
\int\limits_{\TB{S} \setminus \FD_r(S)} (\cl_r(z,v)+ \cl_r(z,\omega v)+\cl_r(z,\omega^{2} v)+1) \varphi_r(z) \, d\LM \le P(r)e^{-r}.
$$
\noindent
This proves Lemma 5.2.

\subsection{The proof of Lemma 5.3} From Proposition 5.16  we have
$$
\int\limits_{H_{r}} (\cl_r(z,v)+\cl_r(z,\omega v)+ \cl_r(z,\omega^{2} v)+1) \varphi_r(z) \, d\LM \le  
$$
\begin{equation}\label{H-estimate}
\end{equation}
$$
\le \sum_{i=1}^{3} \int\limits_{\E^{i}_{r}(S)} (\cl_r(z,v)+\cl_r(z,\omega v)+\cl_r(z,\omega^{2} v)+1) \varphi_r(z) \, d\LM,
$$
\noindent
so in order to prove Lemma 5.3 we need to estimate from above the integrals on the right hand side.
\vskip .1cm
We first estimate the integral
$$
\int\limits_{\E^{1}_{r}(S)} (\cl_r(z,v)+\cl_r(z,\omega v)+\cl_r(z,\omega^{2} v)+1) \varphi_r(z) \, d\LM.
$$
\noindent
Let  $(z,v) \in \E^{1}_r(S)$. Then $(z,v) \in \FD_r(S)$. Let $(z,v)$ also denote a lift of $(z,v)$ to $\Ha$ . Then 
$z \in \Hor_c(2r-\log 3)$, for some $c \in \Cus(G)$ and $c$ is an endpoint of $\abs^{1}_r(z,v)$. Here $[c]_G=c_i(S)$ for some $c_i(S) \in \Cus(S)$. Let $X_0(c_i(S))$  be the set of those $(z,v) \in \E^{1}_r(G)$ so that  $z \in \Hor_{c}(2r-\log 3)$ and $\abs_r(z, v)=c$. Let  $X_1(c_i(S))$  be the set of those $(z,v) \in \E^{1}_r(G)$ so that  $z \in \Hor_{c}(2r-\log 3)$ and $\abs_r(z,\omega v)=c$. Then $\E^{1}_r(G)=X_0(c_i(S)) \cup X_1(c_i(S))$.
\vskip .1cm
Let $(z,v) \in X_0(c_i(S))$.  Since   $(z,v) \in \FD_r(S)$ we have
$\max\{\at_r(z,v),\at_r(z,\omega v),\at_r(z,\omega^{2}v)\} \le  r^{2}$. It follows from  Proposition 5.2 that 
$|\ang_c(z,\omega v)| \le \pi e^{-r}$. This implies  that  $|\ang_c(z,\omega^{j} v)|>{{\pi}\over{2}}$, $j=1,2$.  Since $\abs_r(z,v)=c$ we have $\cl_r(z,v)=e^{\h(z)}$. Combining this with (\ref{CL-estimate-thin}) (recall  the definition of $\chi_r(z,v)$) we get
$$
\cl_r(z,v)+\cl_r(z,\omega v)+\cl_r(z,\omega^{2} v) \le  e^{\h(z)}+ 2D(S)\big(e^{\h(z)}+K(S)(r^{2}+1) \big).
$$
A single lift of the set  $X_0(c_i(S))$ to $\Ha$ (with respect to the normalised group $G=G_c$) is contained in the set 
$$
\{(x,y,\theta) \in \TB\Ha: \, 0\le x\le 1, \,  {{e^{2r}}\over{3}} \le y \le \infty , \, |\theta| \le \pi e^{-r} \}.
$$
\noindent
Passing with the integration to the universal cover we get for $r$ large that
$$
\int\limits_{ X_0(c_i(S))} (\cl_r(z,v)+\cl_r(z,\omega v)+\cl_r(z,\omega^{2} v)+1) \varphi_r(z) \, d\LM <
$$
$$
<\int\limits_{-\pi e^{-r}}^{\pi e^{-r}}\left(  
\int\limits_{{{e^{2r}}\over{3}}}^{\infty} \left(y+2D(S)\big( y+K(S)(1+r^{2}) \big)+1 \right){{e^{r}}\over{\sqrt{y}}} \,
{{dy}\over{y^{2}}} \right) \, d\theta \le 
$$
$$
\le r^{3}e^{-r}  \int\limits_{{{e^{2r}}\over{3}}}^{\infty}e^{r}y^{ -{{3}\over{2}}}\, dy <3r^{3}e^{-r}=P(r)e^{-r}.
$$
\noindent
We repeat this for every $c_i(S) \in \Cus(S)$. After  repeating the same argument for $X_j(c_i(S))$, $j=1,2$, we obtain 
\begin{equation}\label{E-1-estimate}
\int\limits_{\E^{1}_{r}(S)} (\cl_r(z,v)+\cl_r(z,\omega v)+\cl_r(z,\omega^{2} v)+1) \varphi_r(z) \, d\LM< P(r)e^{-r}.
\end{equation}
\noindent
for $r$ large enough.
\vskip .3cm
It is somewhat more delicate to estimate the integrals over the sets $\E^2_{r}(S)$ and $\E^3_{r}(S)$. We need to prove several propositions first.

\begin{proposition} Let $s,t>0$. Then for $t$ large enough, we have 
$$
I(s,t)=\int\limits_{\TB\Thick_{S}(s)} \iota(\gamma_{(z,v)}(t),\tau(S))\, d\LM \le t^2(s+t).
$$
\end{proposition}
\begin{proof} Denote the above integral by $I(s,t)$.
Consider $\delta_{(z,v)}(t)=\iota(\gamma_{(z,v)}(t+1),\tau(S))-\iota(\gamma_{(z,v)}(t),\tau(S))$. 
It follows from Proposition 7.2 that 
$$
\delta_{(z,v)}(t) \le \max \{K(S),  N(S)e^{\h(\gamma_{(z,v)}(t))} \}.
$$
\noindent
so we conclude that for any $(z,v) \in \TB{S}$ we have
$$
\delta_{(z,v)}(t) \le K(S)+N(S)e^{\h(\gamma_{(z,v)}(t))}.
$$
\noindent
It follows that
$$
I(t+1)-I(t)=\int\limits_{\TB\Thick_{S}(s)}  \delta_{(z,v)}(t) \, d\LM \le \int\limits_{\TB\Thick_{S}(s)} (K(S)+N(S)e^{\h(\gamma_{(z,v)}(t))}) \, d\LM.
$$
\noindent
Since $\flow_t(\TB\Thick_S(s)) \subset \TB\Thick_S(s+t)$ and since $\flow_t$ is measure preserving, we have
$$
I(t+1)-I(t) \le \int\limits_{\TB\Thick_{S}(s)} (K(S)+N(S)e^{\h(\gamma_{(z,v)}(t))}) \, d\LM=\int\limits_{\flow_t(\TB\Thick_{S}(s))} (K(S)+N(S)e^{\h(z)}) \, d\LM \le
$$
$$
\le \int\limits_{\TB\Thick_{S}(s+t)} (K(S)+N(S)e^{\h(z)}) \, d\LM=\int\limits_{\TB\Thick_{S}(0)} (K(S)+N(S)e^{\h(z)}) \, d\LM+
$$
$$
+\sum_{i=1}^{n}\int\limits_{H(c_{i}(S))}   (K(S)+N(S)e^{\h(z)}) \, d\LM,
$$
\noindent
where $H(c_i(S))=\TB\Hor_{c_{i}(S)}(0) \setminus \TB\Hor_{c_{i}(S) }(s+t)$.
We have 
$$
\int\limits_{\TB\Thick_{S}(0)} (K(S)+N(S)e^{\h(z}) \, d\LM \le \int\limits_{\TB\Thick_{S}(0)} (K(S)+N(S)) \, d\LM \le C,
$$
\noindent
where $C>0$ depends only on $S$. On the other hand, by passing to the group $G=G_{c_{i}(S)}$ and using the expression for the volume element $d\LM$ we have 
$$
\int\limits_{H(c_{i}(S))}(K(S)+N(S)e^{\h(z)}) \, d\LM=\int\limits_{\TB\Hor_{\infty}(0) \setminus \TB\Hor_{\infty}(s+t)}(K(S)+N(S)e^{\h_{\infty}(z)}) \, d\LM=
$$
$$
=\int\limits_{0}^{1} \left( \int\limits_{1}^{e^{s+t}} {{K(S)+N(S)y}\over{y^{2}}} \, dy \right) \, dx \le K(S)+N(S)(s+t).
$$
\noindent
Combining these estimates we have $I(t+1)-I(t) \le C+K(S)+N(S)(s+t)$
and therefore $I(t) \le t(C+K(S)+N(S)(s+t) )$.  For $t$ large enough we have $I(t) \le t^2(t+s)$ which proves the lemma.
\end{proof}

We have already discussed (in Section 6) the cross sectional area for two dimensional subsets of $\TB\Ha$  with respect to the equidistant flow $\tra_t$. 

\begin{definition} Let $E \subset \Ha$ be a domain, and let $u(z)$ be a smooth unit vector field on $E$. Then $E(u)=\{ (z,u(z)): z \in E \}$ 
is a surface sitting inside the three dimensional manifold $\TB\Ha$. Let $0 \le t_1 \le t_2 \le \infty$. Define  $\U(E(u),t_1,t_2)=
\cup_{t_{1} \le s \le t_{2}}\flow_s(E(u))$. For simplicity, set $\U(E(u),0,t)=\U(E(u),t)$.
We use the same notation for $E \subset S$.
\end{definition}

We say that  a unit vector field $u$ defined on a domain $E \subset \Ha$ is transverse to $E$ if $\flow_s(E(u)) \cap E(u)=\emptyset$ for every $s>0$. Assume that $u$ is  transverse to $E$.
Since $\flow_t$ preserves the volume, we have that the quotient $\LM(\U(E(u),t))/t$ does not depend on $t \ge 0$. We call  $\LM(\U(E(u),t))/t$ the cross sectional area of $E$ with respect to $u(z)$. In fact, the flow $\flow_t$ induces an area form $d\eta(u)$ on $E$ so that the cross sectional area of $E$ agree with the area of $E$ with respect to $d\eta(u)$.
The two form $d\eta(u)$ is obtained by  contracting $|d\LM|$ by the vector field $u(z)$. One can verify that since each vector  $u(z)$ has the unit length, and since $d\LM=y^{-2} \, dx\wedge dy \wedge d\theta$  we have that $d\eta(u)=\sigma(z)\, dx \wedge dy$ where $0 \le \sigma(z) \le y^{-2}$. That is the density of the two form $d\eta(u)$ is always bounded above by the density of the hyperbolic metric.  
\vskip .1cm
Let $\psi=\psi_{E(u)}:\U(E(u),\infty) \to E(u)$ be the projection map, that is on each slice $\flow_s(E(u))$, $0 \le s < \infty$, the map $\psi$ agrees with $(\flow_s)^{-1}$. Since $u$ is transverse to $E$ the map $\psi$ is well defined. 
\vskip .1cm
Let $f_t:E(u) \to \R$, $t_1 \le t \le t_2$, and $f:\U(E(u),t) \to \R$ be integrable functions. If for every such $t$ we have $f(z,v) \le f_t(\psi_{E(u)}(z,v))$ for every $(z,v) \in \flow_t(E(u))$ 
then
\begin{equation}\label{psi-1}
\int\limits_{\U(E(u),t_1,t_2)}f\, d\LM \le \int\limits_{t_{1}}^{t_{2}} \left(\int\limits_{E(u)}f_t d\eta(u) \right)dt,
\end{equation}
\noindent
and if $f_t(\psi_{E(u)}(z,v)) \le f(z,v)$ for every $(z,v) \in \flow_t(E(u))$ then
\begin{equation}\label{psi-2}
\int\limits_{t_{1}}^{t_{2}} \left(\int\limits_{E(u)}f d\eta(u)\right)dt \le \int\limits_{\U(E(u),t)}f\, d\LM.
\end{equation}
\vskip .1cm
We are particularly interested in the vector fields $u_1(z)$ and $u_2(z)$. The vector $u_1(z)$ is the unique vector such that the point $(z,u_1(z)) \in \TB\Ha$ corresponds to the coordinates $(x,y,-{{\pi}\over{2}})$.  Observe that $u_1$ is transverse to $\Ha$.
The vector $u_2(z)$ is the unique vector such that the point $(z,u_2(z)) \in \TB\Ha$ corresponds to the coordinates $(x,y,-\pi)$.
One can verify that $d\eta(u_1)=y^{-2}\, dxdy$ and that  $d\eta(u_2)=0$. It is not surprising that $d\eta(u_2)=0$
since  the corresponding set $\U(E(u_2),t)$ is two dimensional for any set $E \subset \Ha$ and therefore the cross sectional area of $E$ with respect to 
$u_2(z)$ is equal to zero. 
\vskip .1cm
Fix $c_i(S) \in \Cus(S)$. For $y \ge 1$ set
$$
A_j(t,y,c_{i}(S))=\int\limits_{0}^{1}\iota(\gamma_{(z,u_{j}(z))}(t),\tau(G_{c_{i}(S)} ) ) \, dx,
$$
\noindent
where $j=1,2$. Set $G=G_{c_{i}(S)}$. Note that if $z \in \gamma \cap \Hor_{\infty}(0)$ and $\gamma \in \lambda(G)$ then $\gamma$ has $\infty$ as its endpoint.
This implies that for $1 \le y$ we have
\begin{equation}\label{A-estimate-1}
A_2(t,1,c_{i}(S))= A_2(t+\log y,y,c_{i}(S)),
\end{equation}
\noindent
and  $A_2(t,y,c_{i}(S))=0$ for $t \le \log y$.
\vskip .1cm
Fix $y \ge 1$. For every $t \ge 0$  the set of points $\gamma_{(z,u_{1}(z))}(t)$ is a horocircle in $\Ha$ that bounds a horoball at $\infty$ (recall $z=x+iy$). Moreover, there exist numbers $d(t) \ge 0$ and $r(t) \ge 0$ (that depend only on $t$),  such that for 
$z'=z+yd(t)$ we have $\gamma_{(z,u_{1}(z))}(t)=\gamma_{(z',u_{2}(z'))}(r(t))$. The functions $d(t)$ and $r(t)$ are increasing in $t$. It is elementary to verify that 
$0 \le d(t) \le 1$ and $0 \le t-r(t)<1$. 
\vskip .1cm
We apply Proposition 7.4 to the segment $\gamma_{(z,u_{1}(z))}[0,t]$. We have
$$
\iota(\gamma_{(z,u_{1}(z))}(t),\tau(G)) \le \iota(\gamma_{(z',u_{2}(z'))}(r(t)),\tau(G))+N(S)yd(t).
$$
\noindent
This shows that for every $t\ge 0$ we have 
$$
\iota(\gamma_{(z',u_{2}(z'))}(t-1),\tau(G)) \le  \iota(\gamma_{(z',u_{2}(z'))}(r(t)),\tau(G)) \le \iota(\gamma_{(z,u_{1}(z))}(t),\tau(G)) \le 
$$
$$
\iota(\gamma_{(z',u_{2}(z'))}(r(t)),\tau(G))+N(S)y \le
\iota(\gamma_{(z',u_{2}(z'))}(t),\tau(G))+N(S)y,
$$
\noindent
that is 

$$
\iota(\gamma_{(z',u_{2}(z'))}(t-1),\tau(G)) \le \iota(\gamma_{(z,u_{1}(z))}(t),\tau(G)) \le \iota(\gamma_{(z',u_{2}(z'))}(t),\tau(G))+N(S)y.
$$
\vskip .1cm
Since the vector $u_2(z')$ is obtained by translating the vector $u_2(z)$ for $yd(t)$ and $d(t)$ does not depend on $x$ we get

\begin{equation}\label{A-estimate-2}
A_2(t-1,y,c_{i}(S)) \le A_1(t,y,c_{i}(S)) \le  A_2(t,y,c_{i}(S)) +N(S)y. 
\end{equation}

\begin{proposition} For $y \ge 1$ and any $c_i(S)$ we have 
$$
A_2(t,y,c_{i}(S)) \le (t+3)^4,
$$
\noindent
and
$$
A_1(t,y,c_{i}(S)) \le (t+3)^4+N(S)y,
$$
\noindent
for $t$ large enough.

\end{proposition}

\begin{proof} Fix $c_i(S) \in \Cus(S)$ and let $G=G_{c_{i}(S)}$. 
Let $E \subset \Ha$ be the set given by $E=\{z:0 \le x \le 1, \, \,  \text{and} \, \,  4 \le y \le 5\}$. 
Let $\U(E(u_1),1) \subset \TB\Ha$ be the corresponding set
(see the above definition). Note that  $\U(E(u_1),1) \subset \TB\Thick_G(\log 5)$. Moreover, the set  $\U(E(u_1),1) $ injects into $\TB\Thick_S(\log 5)$
under the standard covering map. Therefore, from Proposition 7.6 for $t$ large enough we have
$$
\int\limits_{\U(E(u_{1}),1)} \iota(\gamma_{(z,v)}(t),\tau(G))\, d\LM \le \int\limits_{\TB\Thick_S(\log 5)} 
\iota(\gamma_{(z,v)}(t),\tau(S))\, d\LM=
$$
\begin{equation}\label{A-estimate-3}
\end{equation}
$$
=\int\limits_{\TB\Thick_S(\log 5)} \iota(\gamma_{(z,v)}(t),\tau(S))\, d\LM \le 
t^2(\log 5+t)<t^2(2+t).
$$
\vskip .1cm
Let $\psi:\U(E(u_1),1) \to E(u_1)$ be the restriction of the projection map  $\psi_{E(u_{1})}:\U(E(u_1),\infty) \to E(u_1)$
introduced above.
It holds that $\iota(\gamma_{\psi(z,v)}(t),\tau(G)) \le \iota(\gamma_{(z,v)}(t),\tau(G))+C_1$ for every $t>0$ 
where $C_1$ is a constant that depends only on $S$. The constant $C_1$ bounds above the number of intersections between the geodesic segment $\gamma_{(w,u_{1}(w))}[0,1]$ and $\tau(G)$ when $w \in E$.  From (\ref{psi-2}) we have
$$
\int\limits_{E} \iota(\gamma_{(z,u_{1}(z))}(t),\tau(G))d\eta(u_1) \le 
\int\limits_{\U(E(u_{1}),1)} (\iota(\gamma_{(z,v)}(t),\tau(G))+C_1)\, d\LM.
$$
\noindent
Combining this with  (\ref{A-estimate-3}) yields

$$
\int\limits_{E} \iota(\gamma_{(z,u_{1}(z))}(t),\tau(G))d\eta(u_1) \le (t+1)^{2}(3+t)+C_1\LM(\U(E(u_{1}),1)) \le 2t^{3} ,
$$
\noindent
for $t$ large enough.  Since $d\eta(u_1)=y^{-2}\, dxdy$ this gives 
$$
\int\limits_{4}^{5} A_1(t,y,c_i(S))\, dy \le 25\int\limits_{4}^{5} A_1(t,y,c_i(S))y^{-2}\, dy=
$$
\begin{equation}\label{A-estimate-4}
\end{equation}
$$
=25\int\limits_{4}^{5} \int\limits_{E} \iota(\gamma_{(z,u_{1}(z))}(t),\tau(G))d\eta(u_1) \le  t^{4} .
$$

From (\ref{A-estimate-2}) and  (\ref{A-estimate-4})  we obtain 
$$
\int\limits_{4}^{5} A_2(t,y,c_i(S))\, dy \le (t+1)^{4}.
$$
\noindent
Observe that for $y \in [4,5]$ we have $A_2(t,1,c_i(S)) \le A_2(t+\log y,y,c_i(S)) \le A_2(t+2,y,c_i(S))$. Therefore
$$
A_2(t,1,c_i(S)) \le \int\limits_{4}^{5} A_2(t+2,y,c_i(S))\, dy \le (t+3)^{4}.  
$$
\noindent
Combining this with (\ref{A-estimate-1}) we conclude that for every $y \ge 1$ we have 
$A_2(t,y,c_i(S))\le A_2(t,1,c_i(S)) \le (t+3)^{4}$. The second estimate in the proposition follows from the first one, and from (\ref{A-estimate-2}).

\end{proof}

\begin{remark} If instead of $u_1(z)$ we use $-u_1(z)$ to define $A_1(t,y,c_i(S))$ that is
$$
\wh{A}_1(t,y,c_i(S))=\int\limits_{0}^{1}\iota(\gamma_{(z,-u_{1}(z))}(t),\tau(G_{c_{i}(S)})) \, dx.
$$
\noindent
we obtain the same estimate $\wh{A}_1(t,y,c_i(S)) \le (t+3)^{4}+N(S)y$.
\end{remark}

\vskip .1cm

Again fix $c_i(S) \in \Cus(S)$. Let $z_t=\gamma_{(z,u_{1}(z))}(t)$  and $v_t=\gamma'_{(z,u_{1}(z))}(t)$. 
For $y \ge 1$ set
$$
I(t,y,c_i(S))=\int\limits_{0}^{1} 
\left( \iota(\gamma_{(z,u_{1}(z))}(t),\tau(G_{c_{i}}))+ \sum_{j=1}^{2} 
\iota(\gamma_{(z_{t},\omega^{j} (-v_{t}) )}(r^2+1),\tau(G_{c_{i}})) \right) \, dx.
$$
\noindent
Set $G=G_{c_{i}}(S)$. We want to estimate $I(t,y,c_i(S))$.   
For every $t \ge 0$ and $y \ge 1$  the set of points $\gamma_{(z_{t},\omega^{j} (-v_{t}))}(r^2+1)$ 
is a horocircle in $\Ha$ that bounds a horoball at $\infty$ (recall $z=x+iy$). 
Since $0<|\ang_{\infty}(z_t, (-v_t))| \le {{\pi}\over{2}}$ we see that
$|\ang_{\infty}(z_t,\omega^{j} (-v_t))| \ge {{\pi}\over{6}}$ for any value of $t$. This  implies that 
the point $\gamma_{(z_{t},\omega^{j} (-v_{t}))}(r^2+1)$  lies below the the horocircle $\partial{\Hor_{\infty}(\log y)}$ for $r$ large enough.
\vskip .1cm
Fix $y \ge 1$. Recall the functions $d(t)$ and $r(t)$ defined above (and the corresponding point $z'=z+yd(t)$). 
Similarly, we define the functions $d_j(t)$ and $r_j(t)$ as follows.  Let $z''_j \in \Ha$ so that $z''_j$ lies on the same horocircle $\partial{\Hor_{\infty}(\log y)}$ as the points $z$ and $z'$, and so that  
$\gamma_{(z''_{j},u_{2}(z''_{j}))}(r_j(t))=  \gamma_{(z_{t},\omega^{j} (-v_{t}))}(r^2+1) $. Then $z''_j=z+yd_j(t)$.
The functions $d_j(t)$ and $r_j(t)$ are increasing in $t$ and do not depend on $z$. It is elementary to verify that 
$-1 < d_j(t) < 2$ and $0 \le r(t)<t+r^2+1$. 
\vskip .1cm

As we explained (and stated) above from Proposition 7.4 we have 
$$
\iota(\gamma_{(z,u_1(z))}(t),\tau(G) ) \le  \iota(\gamma_{(z',u_{2}(z'))}(r(t)),\tau(G))+N(S)y.
$$
\noindent
Similarly we  apply Proposition 7.4 to the segment $\gamma_{(z_{t},\omega^{j} (-v_{t}) )}(r^2+1)$ and get

$$
\iota( \gamma_{(z_{t},\omega^{j} (-v_{t}) ) }(r^2+1),\tau(G)) \le  \iota(\gamma_{(z',u_{2}(z'))}(r(t)),\tau(G))+ \iota(\gamma_{(z''_{j},u_{2}(z''_{j}))}(r_{j}(t)),\tau(G)) +2N(S)y,
$$
\noindent
(here we used that the Euclidean length of the horocyclic segment between the points $z'$ and $z''_j$ is at most $2$).
This yields that
$$
I(t,y,c_i(S)) \le 3A_2(t,y,c_i(S))+2A_2(t+r^2+1,y,c_i(S))+5N(S)y,
$$
\noindent
so from Proposition 7.7, for $r$ large enough  we get
$$
I(t,y,c_i(S)) \le 3(t+3)^4+ 2(t+r^2+3)^4+5N(S)y,
$$
\noindent
that is
\begin{equation}\label{I-estimate}
I(t,y,c_i(S)) < 3(t+r^2+3)^4+ 5N(S)y.
\end{equation}

\begin{remark} If instead of $u_1(z)$ we use $-u_1(z)$ to define $I(t,y,c_i(S))$, that is 
$$
I(t,y,c_i(S))=\int\limits_{0}^{1} 
\left( \iota(\gamma_{(z,-u_{1}(z))}(t),\tau(G_{c_{i}}))+ \sum_{j=1}^{2} 
\iota(\gamma_{(z_{t},\omega^{j} (-v_{t}) )}(r^2+1),\tau(G_{c_{i}}) \right) \, dx.
$$
\noindent
we get that the same estimate (\ref{I-estimate}) holds.
\end{remark}

\vskip .1cm

We are ready to estimate the integral
$$
\sum_{i=2}^{3} \int\limits_{\E^{i}_{r}(S)} (\cl_r(z,v)+\cl_r(z,\omega v)+\cl_r(z,\omega^{2} v)+1) \varphi_r(z) \, d\LM.
$$
\noindent
Note that for $c_i(S) \in \Cus(S)$ the vector fields  $u_j(z)$  are well defined on $\TB\Hor_{c_{i}(S)}(0) \subset \TB{S}$.  
\vskip .1cm
Let $E=\Thick_S(12r) \setminus \Thick_S(0)$. 
Let $(z,v) \in \E^{2}_r(G)$. We also use $(z,v)$ to denote a lift of $(z,v)$ to $\Ha$. Then 
$$
r-r^2e^{-r}<\h_{\max }(\z_{\max}(\gamma_{(z,v)},c))-\h_{\max }(z)<r+r^2e^{-r}, 
$$
\noindent
and  $\z_{\max}(\gamma_{(z,v)},c) \in \Hor_{\infty}(1-r^2e^{-r})$ for some cusp  $c \in \Cus(G)$. Set $G=G_c$. It follows from (\ref{ang-estimate-3}) that for $r$ large enough, we have  $\z_{\max}(\gamma_{(z,v)},c))=\gamma_{(z,v)}(t)$ for some $(r-r^2e^{-r})+\log 2-e^{-r}<t<(r+r^2e^{-r})+\log 2+e^{-r}$. Combining this with the fact that $z \in \Thick_G(10r)$ (recall that $E^2_r(S) \subset \FD_r(S) \subset \TB\Thick_S(10r)$), 
this implies that $(z,-v) \in \U(E(u^{*}_1),r+\log 2 -2r^2 e^{-r} ,r+\log 2 +2r^2 e^{-r})$ where $u^{*}_1$ is either equal to $u_1$ or to $-u_1$.
\vskip .1cm
Recall that if  $(z,v) \in \E^{2}_r(S)$ then $(z,v)$ does not belong to  $\E^{1}_r(S)$ so for $j=0,1,2$, by Proposition 7.1 we have
$$
\cl_r(z, \omega ^{j} v)\le \iota(\gamma_{(z,\omega^{j} v ) }(r^2+1),\tau(S)).
$$
\noindent
In particular for $(z,v) \in \flow_t(E(u^{*}_1))$ we have the estimate
$$
\cl_r(z, v)\le \iota(\gamma_{(z,-v)}(r^2+1),\tau(S))<
\iota(\gamma_{(z,(-v))}(t),\tau(S))+\iota(\gamma_{-\psi(z,v)}(r^2+1),\tau(S) ),
$$
\noindent
where $\psi:\U(E(u^{*}_1), \, (r+\log 2 -2r^2 e^{-r}), \, (r+\log 2 +2r^2 e^{-r})) \to E(u^{*}_{1})$ is the projection map. Define
$f:\U(E(u^{*}_1), \, (r+\log 2 -2r^2 e^{-r}), \, (r+\log 2 +2r^2 e^{-r})) \to \R$ by 
$$
f(z,v)=1+\sum_{i=0}^{2}\cl_r(z, \omega ^{j} v),
$$
\noindent
and  $f_t:E(u^{*}_1) \to \R$ by 
$$
f_t(z)=1+(\iota(\gamma_{(z,u_{1}(z))}(t),\tau(S) )+\iota(\gamma_{(z,-u_{1}(z))}(r^{2}+1),\tau(S) )+
$$
$$
+\iota(\gamma_{(z_{t},\omega (-v_{t}) )}(r^2+1),\tau(S))+\iota(\gamma_{(z_{t},\omega^{2} (-v_{t}) )}(r^2+1),\tau(S)).
$$
\noindent
We apply (\ref{psi-1}) separately for $u^{*}_1=u_1$ and $u^{*}_1=-u_1$. Adding the two inequalities, we obtain 

$$
\int\limits_{\E^{2}_{r}(S)} (\cl_r(z,v)+\cl_r(z,\omega v)+\cl_r(z,\omega^{2} v)+1) \varphi_r(z) \, d\LM < 
\int\limits_{\E^{2}_{r}(S)} (\cl_r(z,v)+\cl_r(z,\omega v)+\cl_r(z,\omega^{2} v)+1)\, d\LM \le
$$
$$
2\int\limits_{r+\log 2 -2r^2 e^{-r}}^{r+\log 2 +2r^2 e^{-r}} \left( \int\limits_{E} f_t d\eta(u_1)\right)dt.
$$
\noindent
Passing to the universal cover for each cusp $c_i(S)$ and applying the Fubini theorem to the last integral, yields  that (recall $\wh{A}_1(t,y,c_i(S))$ from the remark after the proof of Proposition 7.7) 
$$
\int\limits_{E} f_t d\eta(u_1) \le \numb \int\limits_{1}^{e^{12r}} \left( 1+I(t,y,c_i(S))+\wh{A}_1(r^2+1,y,c_i(S)) \right)y^{-2}\, dy
< 
$$
$$
< \numb \int\limits_{1}^{e^{12r}} \left( 1+ (t+r^2+3)^4+ 5N(S)+(t+3)^{4}+N(S) \right) \, {{dy}\over{y}},
$$
\noindent
where $\numb$ is the number of cusps in $\Cus(S)$. By using (\ref{I-estimate}) and Proposition 7.7 and since $t \le r+1$ 
we have
$$
\int\limits_{\E^{2}_{r}(S)} (\cl_r(z,v)+\cl_r(z,\omega v)+\cl_r(z,\omega^{2} v)+1) \varphi_r(z) \, d\LM<
$$
\begin{equation}\label{E-2-estimate}
\end{equation}
$$
<\int\limits_{r+\log 2 -2r^2 e^{-r}}^{r+\log 2 +2r^2 e^{-r}} \left( \int\limits_{1}^{12r}
\big( 1+ (t+r^2+3)^4+ 5N(S)+(t+3)^{4}+N(S) \big){{dy}\over{y}}\right)dt \le P(r)e^{-r}.
$$
\noindent
for $r$ large enough. 
\vskip .3cm

Let $E=\Thick_S(1+r^2e^{-r})  \setminus \Thick_S(1-r^2e^{-r})$. 
Recall that if  $(z,v) \in \E^{3}_r(G)$  then  for some $r-1<t<r^2+1$ and some $c \in \Cus(G)$ we have that $\z_{\max}(\gamma_{(z,v)},c)=\gamma_{(z,v)}(t)$ and 
$$
1-r^2e^{-r}<\h_c(\z_{\max}(\gamma_{(z,v)},c))<1+r^2e^{-r}.
$$
\noindent
This implies that if $(z,v) \in \E^{3}_r(S)$ then  $(z,-v) \in \U(E(u^{*}_1), \, (r-1) , \, (r^2+1))$ where $u^{*}_1$ is either equal to $u_1$ or to $-u_1$.
\vskip .1cm
Recall that if  $(z,v) \in \E^{3}_r(S)$ then $(z,v)$ does not belong to  $\E^{1}_r(S)$,  for $j=0,1,2$. 
By Proposition 7.1 we have
$$
\cl_r(z, \omega ^{j} v)\le \iota(\gamma_{(z,\omega^{j} v) }(r^2+1),\tau(S)).
$$
\noindent
In particular, for $(z,-v) \in \flow_t(E(u^{*}_1))$ we have the estimate
$$
\cl_r(z, v)\le \iota( \gamma_{(z,v)}(r^2+1),\tau(S))
<\iota(\gamma_{(z,v)}(t),\tau(S) )+\iota(\gamma_{-\psi(z,v)}(r^2+1),\tau(S) ),
$$
\noindent
where $\psi:\U(E(u^{*}_1),\, (r-1), \, (r^2+1)) \to E(u^{*}_{1})$ is the projection map. Define the maps
$f:\U(E(u^{*}_1),\, (r-1) , \, (r^2+1)) \to \R$ and  $f_t:E(u^{*}_1) \to \R$ as above.
We apply (\ref{psi-2}) separately for $u^{*}_1=u_1$ and $u^{*}_1=-u_1$. Adding the two inequalities we obtain 
$$
\int\limits_{\E^{3}_{r}(S)} (\cl_r(z,v)+\cl_r(z,\omega v)+\cl_r(z,\omega^{2} v)+1) \varphi_r(z) \, d\LM < 
\int\limits_{\E^{3}_{r}(S)} (\cl_r(z,v)+\cl_r(z,\omega v)+\cl_r(z,\omega^{2} v)+1)\, d\LM \le
$$
$$
2\int\limits_{r-1}^{r^2+1} \left( \int\limits_{E} f_t d\eta(u_1)\right)dt.
$$
\noindent
Passing to the universal cover for each cusp $c_i(S)$ and applying the Fubini theorem to the last integral, yields  that 
$$
\int\limits_{E} f_t d\eta(u_1) \le \numb \int\limits_{e^{1-r^{2}e^{-r} }}^{e^{1+r^{2}e^{-r} } } \big( 1+I(t,y,c_i(S))+
\wh{A}_1(r^2+1,y,c_i(S)) \big)y^{-2}\, dy ,
$$
\noindent
where $\numb$ is the number of cusps in $\Cus(S)$. By using (\ref{I-estimate}) and Proposition 7.7 we have

$$
\int\limits_{\E^{3}_{r}(S)} (\cl_r(z,v)+\cl_r(z,\omega v)+\cl_r(z,\omega^{2} v)+1) \varphi_r(z) \, d\LM<
$$

\begin{equation}\label{E-3-estimate}
\end{equation}
$$
<2\numb \int\limits_{r-1}^{r^2+1}      \left( \int\limits_{e^{1-r^{2}e^{-r} }} ^{e^{1+r^{2}e^{-r}} }
\big(1+ (t+r^2+3)^4+ 5N(S)+(t+3)^{4}+N(S)  \big)  \, {{dy}\over{y}} \right)dt  =\pe(r),
$$
\noindent
since $t \le r^2+1$. 
\vskip .1cm
Putting together (\ref{E-1-estimate}), (\ref{E-2-estimate}), and (\ref{E-3-estimate}), and replacing these in 
(\ref{H-estimate}), we get
$$
\int\limits_{H_{r}} (\cl_r(z,v)+\cl_r(z,\omega v)+1) \varphi_r(z) \, d\LM =\pe(r).
$$
\noindent
which proves Lemma 5.3.

\section{Appendix}
To prove Lemma 5.1 it is enough to prove the following somewhat more general theorem.

\begin{theorem} Let $S$ be a hyperbolic non-compact finite-area Riemann surface and let $K \subset S$ be compact. Then we can find $C, q> 0$ such that
$$
A_t = \{ (z, v) \in \TB{S} : \,  \gamma_(z, v)[0, t] \subset K \}.
$$
satisfies $\LM(A_t(K)) \le C e^{-qt}$ for any $t \ge 0$.
\end{theorem}
\begin{proof} We may assume that $K = \Thick_S(h_0)$ for some $h_0 > 0$. We fix a proper ideal triangulation $\tau$ for $S$,  and let $\lambda = \lambda(\tau)$ denote the set of geodesic edges of $\tau$. For any $(z, v) \in \TB{S}$, we can record the sequence of left and right turns taken by the geodesic ray  $\gamma_{(z, v)}[0, \infty)$ to obtain a sequence $R^{a_1}L^{a_2}\ldots$, where $a_i=a_i(z,v) \in \N$. This sequence is finite if and only if $\p(z, v) \in \Cus(G)$. We let $(s_i)_{i=0}^\infty$ be the times such that  $\gamma_{(z, v)}(s_i) \subset \lambda$,
and we observe that we can find $C_1 > C_0 > 0$  ($C_i = C_i(S, \Thick_S(h_0))$) such that $C_0 < s_{i+1} - s_i < C_1$ as long as $\gamma_{(z,v)}[s_i, s_{i+1}] \subset \Thick_S(h_0)$. Moreover,  if $\gamma_{(z,v)}[0, s(\sum_{i=1}^k a_i)] \subset \Thick_S(h_0)$ (where $s(i) = s_i$), then $a_i \le B$ for $i = 1, \dots, k$, where $B \in \N$, depends only on $\Thick_S(h_0)$.
\vskip .1cm
We let
$$
V_n(B) = \{ (z, v) \in \Thick_S(h_0):\,  a_1, \ldots ,  a_{n+1} \text{ is well-defined, and $a_i \le B$ for $i = 1, \ldots n$}  \}.
$$
\noindent  
We will show that there are $C, q > 0$ such that $\LM(V_n(B)) \le Ce^{-qn}$;
this will complete the proof of the theorem. 
\vskip .1cm
To this end,  we fix a triangle $T \in \tau_G$  ($\tau_G$ is the lift of $\tau$ to $\Ha$).
We observe that if $(z, v) \in T$, then the sequence $(a_i(z, v))$ depends only on $\p(z, v)$. We let $W_n(B) = \p(V_n(B))$.
For any $(b_i)_{i=1}^k$, where $b_i$ are positive integers, the set
$$
\{ \p(z, v) :\, (a_i(z, v))_{i=1}^{k+1} \text{ is well-defined, and $a_i = b_i$ for $1 \le i \le k$} \}
$$
\noindent
is an open interval in $\partial \Ha$ whose endpoints are joined by an element of $\lambda_G$.
Therefore $W_n(B)$ is a disjoint union of $B^n$ such intervals.
For any such interval $I \subset W_n(B)$, we let $J \subset U$ be the least interval such that $I \cap W_{n+2}(B) \subset J$. The following two facts are central to our argument:
\begin{enumerate}
\item  The closure of $J$ is a compact subset of $I$.
\item The cross ratio $R(J, I)$ can take on only finitely many values (depending only on $\tau$ and $B$).
\end{enumerate}
Here
$$
R(J, I) \equiv \frac{(j_1 - j_0)(i_1 - i_0)}{(j_0 - i_0)(i_1 - j_1)},
$$
\noindent
where $I = (i_0, i_1)$ and $J = (j_0, j_1)$, is a M\"obius invariant of $(J, I)$. It follows that 
$$
\frac {|J|}{|I|} \le \eta(\tau, B) < 1,
$$
\noindent
where we fix a unit disk model $\D$ for $\Ha$, and let $|J|$ be the arc length for $J\subset \partial \D$.
Therefore
$$
|W_{n+2}(B) \cap I| \le \eta(\tau, B)|I|,
$$
\noindent  
so
$$
|W_{2n}(B)| \le 2\pi (\eta(\tau, B))^{n} = 2\pi e^{-nq},
$$
\noindent  
where $q=-\log \eta(\tau, B)>0$. Since  for any interval $A \subset \partial \D$ we have
$$
\LM(\{(z, v)\in \Thick_G(h_0)\cap T: \,  \p(z, v) \in A \}) \le C(\Thick_S(h_0),\tau) |A|,
$$ 
\noindent
we find that 
$$
|V_{2n}(B)| \le Ce^{-nq},
$$
\noindent  
where $C$ depends only on $\Thick_S(h_0)$ and $\tau$.

\end{proof}

\end{document}